\documentclass[12pt,]{article}
\usepackage{stmaryrd}
\usepackage[normal]{subfigure}
\usepackage{graphicx}
\usepackage{amssymb}
\usepackage{amscd}
\usepackage{amsmath}
\usepackage{amsfonts}
\usepackage{amsbsy}
\usepackage{epsfig}
\usepackage{tikz}
\usepackage{cite}
\usepackage{indentfirst}
\usepackage{multicol}
\usepackage{wasysym}
\usepackage[top=1.25in, bottom=0.6in, left=1.0in, right=1.0in]{geometry}
\usepackage{bbm}
\usepackage{hyperref}
\usepackage{longtable,booktabs}
\usepackage{ntheorem}

\newtheorem{thm}{Theorem}[section]%
\newtheorem{lem}[thm]{Lemma}%
\newtheorem{cor}[thm]{Corollary}%
\newtheorem{defi}[thm]{Definition}%
\newtheorem{pro}[thm]{Proposition}%
\newtheorem{rem}[thm]{Remark}%
\newenvironment{proof}{{\indent \it Proof.}}{\hfill$\square$}

\newtheorem*{thm A}{Theorem A}
\newtheorem*{thm B}{Theorem B}

\allowdisplaybreaks[4]

\begin{document}
\date{}

\newenvironment{ar}{\begin{array}}{\end{array}}
\baselineskip12pt \addtocounter{page}{0}

\title{Finite-dimensional Hopf algebras over the Hopf algebra $H_{b:1}$ of Kashina }

\author{Yiwei Zheng $^{\mathrm{a}}$,
  \ \ Yun Gao $^{\mathrm{a,b}}$,  \  \ Naihong Hu $^{\mathrm{c,}}$
\thanks{\small E-mails:\  nhhu@math.ecnu.edu.cn, \quad ygao@yorku.ca,
\quad zhengyiwei12@foxmail.com} \\
{\small  $^{\mathrm{a}}$Department of Mathematics, Shanghai University,
Shanghai 200444, China}\\
{\small $^{\mathrm{b}}$Department of Mathematics and Statistics
, York University, Toronto M3J 1P3, Canada}\\
{\small  $^{\mathrm{c}}$School of Mathematical Sciences, Shanghai Key Laboratory of PMMP,} \\
{\small East China Normal University,
Shanghai 200241, China}}

\date{}

\maketitle

\small\linespread{1.2}
\begin{abstract}
Let $H$ be the 16-dimensional nontrivial
(namely, noncommutative and noncocommutative) semisimple Hopf
algebra $H_{b:1}$ appeared in Kashina's work~\cite{K00}. We obtain
all simple Yetter-Drinfeld modules over $H$ and then determine all
finite-dimensional Nichols algebras satisfying $\mathcal{B}(V)\cong
\bigotimes_{i\in I}\mathcal{B}(V_i)$, where $V=\bigoplus_{i\in
I}V_i$, each $V_i$ is a simple object in $_H^H\mathcal{YD}$.
Finally, we describe all liftings of those
$\mathcal{B}(V)$ over $H$.
\end{abstract}

 \normalsize

\section{Introduction}
Let $\mathbbm{k}$ be an algebraically closed field of characteristic
zero. In 1975, Kaplansky \cite{K} posed ten conjectures including
the number of iso-classes of finite-dimensional Hopf algebras over
$\mathbbm{k}$. Until now, although there are many classification
results, there are no standard methods, except that in 1998,
Andruskiewitsch and Schneider \cite{AS} put forward the lifting
method which later on plays a very important role in the
classification of finite-dimensional (co-)pointed Hopf algebras.

Actually, the lifting method works well for those classes of
finite-dimensional Hopf algebras with Radford biproduct structure
or their deformations by certain Hopf $2$-cocycles. Let $A$ be a
Hopf algebra. If the coradical $A_0$ is a subalgebra, then the
coradical filtration $\{A_n\}_{n\geq 0}$ is a Hopf algebra
filtration, and the associated graded coalgebra
$grA=\bigoplus_{n\geq 0}A_n/A_{n-1}$ is also a Hopf algebra, where
$A_{-1}=0$. Let $\pi:gr A\rightarrow A_0$ be the homogeneous
projection. By a theorem of Radford~\cite{R85}, there exists a
unique connected graded braided Hopf algebra $R=\bigoplus_{n\geq
0}R(n)$ in the monoidal category $^{A_0}_{A_0}\mathcal{YD}$ such
that $grA\cong R\sharp A_0$. We call $R$ or $R(1)$ the diagram or
infinitesimal braiding of $A$, respectively. Moreover, $R$ is
strictly graded, that is, $R(0)=\mathbbm{k}, ~R(1)=\mathcal{P}(R)$
(see Definition 1.13 \cite{AS2}).

The lifting method has been applied to classify some
finite-dimensional pointed Hopf algebras such as
\cite{ACG15a},~\cite{ACG15b},~\cite{AFG10},~\cite{AFG11},~\cite{AHS},~\cite{AS3},~\cite{FG},~\cite{GGI},
etc., and copointed Hopf algebras \cite{AV},~\cite{GIV}, etc.
Nevertheless, there are a few classification results on
finite-dimensional Hopf algebras whose coradical is neither a group
algebra nor the dual of a group algebra, for instance, \cite{AGM15},
\cite{CS}, \cite{GG16}, \cite{HX16}, \cite{HX18}, \cite{S16},
\cite{X19}, etc. In fact, Shi began a program in \cite{S16} to
classify the objects of finite-dimensional growth from a given
nontrivial semisimple Hopf algebra $A_0=H_8$ via some relevant
Nichols algebras $\mathcal B(V)$ derived from its semisimple
Yetter-Drinfeld modules $V\in {^{A_0}_{A_0}\mathcal{YD}}$.

The classification of finite-dimensional Hopf algebras $A$ whose
coradical $A_0$ is a (Hopf) subalgebra needs the following main steps:

(a) Determine those $V$ in $^{A_0}_{A_0}\mathcal{YD}$ such that the Nichols algebra $\mathcal{B}(V)$ is finite-dimensional and
present $\mathcal{B}(V)$ by generators and relations.

(b) If $R=\bigoplus_{n\geq 0}R(n)$ is a finite-dimensional Hopf algebra in $^{A_0}_{A_0}\mathcal{YD}$ with $V=R(1)$, decide if $R\cong \mathcal{B}(V)$.

(c) Given $V$ as in (a), classify all $A $ such that $grA \cong \mathcal{B}(V)\sharp A_0$. We call $A$ a lifting of $\mathcal{B}(V)$ over $A_0$.

In this paper, we fix a $16$-dimensional non-trivial semisimple Hopf
algebra $H=H_{b:1}$, which appeared in \cite{K00}, and study these
questions. For the definition of Hopf algebra $H$, see Definition
\ref{def3.1}. We first determine the Drinfeld double $D:=D(H^{cop})$
of $H^{cop}$ and describe the simple $D$-modules. In fact, we prove
in Theorem \ref{simple} that there are $32$ one-dimensional simple objects
$\mathbbm{k}_{\chi_{i,j,k,l}}$ with $0\leq i,~k,~l<2, ~0\leq j<4$,
and $56$ two-dimensional simple objects $V_{i,j,k,l,m,n}$ with
$(i,j,k,l,m,n)\in \Omega$, $W_{i,j,k,l}$ with $(i,j,k,l)\in
\Lambda^1$, $U_{i,j,k,l}$ with $(i,j,k,l)\in \Lambda^2$. Using the
fact that the baided categories $_D\mathcal{M}$ and
$_H^H\mathcal{YD}$ are monoidally isomorphic (see \cite{CK}), we
actually get the simple objects in $_H^H\mathcal{YD}$. Furthermore, we get
all the possible finite-dimensional Nichols algebras of semisimple
objects satisfying $\mathcal{B}(V)\cong \bigotimes_{i\in
I}\mathcal{B}(V_i)$. Here is our first main result.

\begin{thm A}
Let $V=\bigoplus_{i\in I}V_i$, each $V_i$ be a simple object in
$_H^H\mathcal{YD}$. Then the Nichols algebra $\mathcal{B}(V)$
satisfying $\mathcal{B}(V)\cong \bigotimes_{i\in I}\mathcal{B}(V_i)$
is finite-dimensional if and only if~ $V$ is isomorphic to one of
the following Yetter-Drinfeld modules

$(1)~\Omega_1(n_1,n_2,n_3,n_4,n_5,n_6,n_7,n_8)\triangleq\bigoplus_{j=1}^8 V_j$ with $\sum_{j=1}^8n_j\geq 1$.

$(2)~\Omega_2(n_1,n_2,n_3,n_4)\triangleq V_3^{\oplus n_1}\oplus V_4^{\oplus n_2}\oplus V_7^{\oplus n_3}\oplus V_8^{\oplus n_4}\oplus M_1$ with $\sum_{j=1}^4n_j\geq 0$.

$(3)~\Omega_3(n_1,n_2,n_3,n_4)\triangleq V_1^{\oplus n_1}\oplus V_2^{\oplus n_2}\oplus V_5^{\oplus n_3}\oplus V_6^{\oplus n_4}\oplus M_2$ with $\sum_{j=1}^4n_j\geq 0$.

$(4)~\Omega_4(n_1,n_2,n_3,n_4)\triangleq V_1^{\oplus n_1}\oplus V_3^{\oplus n_2}\oplus V_5^{\oplus n_3}\oplus V_7^{\oplus n_4}\oplus M_3$ with $\sum_{j=1}^4n_j\geq 0$.

$(5)~\Omega_5(n_1,n_2,n_3,n_4)\triangleq V_1^{\oplus n_1}\oplus V_4^{\oplus n_2}\oplus V_5^{\oplus n_3}\oplus V_8^{\oplus n_4}\oplus M_4$ with $\sum_{j=1}^4n_j\geq 0$.

$(6)~\Omega_6(n_1,n_2,n_3,n_4)\triangleq V_2^{\oplus n_1}\oplus V_4^{\oplus n_2}\oplus V_6^{\oplus n_3}\oplus V_8^{\oplus n_4}\oplus M_5$ with $\sum_{j=1}^4n_j\geq 0$.

$(7)~\Omega_7(n_1,n_2,n_3,n_4)\triangleq V_2^{\oplus n_1}\oplus V_3^{\oplus n_2}\oplus V_6^{\oplus n_3}\oplus V_7^{\oplus n_4}\oplus M_6$ with $\sum_{j=1}^4n_j\geq 0$.

$(8)~\Omega_{8}(n_1,n_2,n_3,n_4)\triangleq V_1^{\oplus n_1}\oplus V_2^{\oplus n_2}\oplus V_5^{\oplus n_3}\oplus V_6^{\oplus n_4}\oplus M_7$ with $\sum_{j=1}^4n_j\geq 0$.

$(9)~\Omega_{9}(n_1,n_2,n_3,n_4)\triangleq V_3^{\oplus n_1}\oplus V_4^{\oplus n_2}\oplus V_7^{\oplus n_3}\oplus V_8^{\oplus n_4}\oplus M_{8}$ with $\sum_{j=1}^4n_j\geq 0$.

$(10)~\Omega_{10}(n_1,n_2,n_3,n_4)\triangleq V_2^{\oplus n_1}\oplus V_4^{\oplus n_2}\oplus V_6^{\oplus n_3}\oplus V_8^{\oplus n_4}\oplus M_{9}$ with $\sum_{j=1}^4n_j\geq 0$.

$(11)~\Omega_{11}(n_1,n_2,n_3,n_4)\triangleq V_2^{\oplus n_1}\oplus V_3^{\oplus n_2}\oplus V_6^{\oplus n_3}\oplus V_7^{\oplus n_4}\oplus M_{10}$ with $\sum_{j=1}^4n_j\geq 0$.

$(12)~\Omega_{12}(n_1,n_2,n_3,n_4)\triangleq V_1^{\oplus n_1}\oplus V_3^{\oplus n_2}\oplus V_5^{\oplus n_3}\oplus V_7^{\oplus n_4}\oplus M_{11}$ with $\sum_{j=1}^4n_j\geq 0$.

$(13)~\Omega_{13}(n_1,n_2,n_3,n_4)\triangleq V_1^{\oplus n_1}\oplus V_4^{\oplus n_2}\oplus V_5^{\oplus n_3}\oplus V_8^{\oplus n_4}\oplus M_{12}$ with $\sum_{j=1}^4n_j\geq 0$.

$(14)~\Omega_{14}\triangleq M_1\oplus M_1$, $\Omega_{15}\triangleq M_1\oplus M_2$, $\Omega_{16}\triangleq M_1\oplus M_7$, $\Omega_{17}\triangleq M_3\oplus M_3$, $\Omega_{18}\triangleq M_3\oplus M_5$, $\Omega_{19}\triangleq M_3\oplus M_9$, $\Omega_{20}\triangleq M_4\oplus M_4$, $\Omega_{21}\triangleq M_4\oplus M_6$, $\Omega_{22}\triangleq M_7\oplus M_7$, $\Omega_{23}\triangleq M_7\oplus M_8$, $\Omega_{24}\triangleq M_{13}\oplus M_{13}$, $\Omega_{25}\triangleq M_{13}\oplus M_{14}$, $\Omega_{26}\triangleq M_{15}\oplus M_{15}$, $\Omega_{27}\triangleq M_{15}\oplus M_{16}$, $\Omega_{28}\triangleq M_{17}\oplus M_{17}$, $\Omega_{29}\triangleq M_{17}\oplus M_{18}$.

$(15)~\Omega_{30}\triangleq M_2\oplus M_2$, $\Omega_{31}\triangleq M_2\oplus M_8$, $\Omega_{32}\triangleq M_4\oplus M_{10}$, $\Omega_{33}\triangleq M_5\oplus M_5$, $\Omega_{34}\triangleq M_5\oplus M_{11}$, $\Omega_{35}\triangleq M_6\oplus M_6$, $\Omega_{36}\triangleq M_6\oplus M_{12}$, $\Omega_{37}\triangleq M_8\oplus M_8$, $\Omega_{38}\triangleq M_9\oplus M_9$, $\Omega_{39}\triangleq M_9\oplus M_{11}$, $\Omega_{40}\triangleq M_{10}\oplus M_{10}$, $\Omega_{41}\triangleq M_{10}\oplus M_{12}$, $\Omega_{42}\triangleq M_{11}\oplus M_{11}$, $\Omega_{43}\triangleq M_{12}\oplus M_{12}$, $\Omega_{44}\triangleq M_{14}\oplus M_{14}$, $\Omega_{45}\triangleq M_{16}\oplus M_{16}$, $\Omega_{46}\triangleq M_{18}\oplus M_{18}$,
$\Omega_{47}\triangleq M_{19}\oplus M_{19}$,
$\Omega_{48}\triangleq M_{19}\oplus M_{20}$,
$\Omega_{49}\triangleq M_{20}\oplus M_{20}$.
\end{thm A}

\begin{rem}
We point out which of the Yetter-Drinfeld modules have principal realizations and which not, since the liftings are known when there exist principal realizations {\rm\cite [Subsection 2.2]{AAI}}.
Let $(h)$ and $(\delta_h)$ be dual bases of $H$ and $H^*$. Define
\begin{equation*}
\begin{split}
\chi:=&\delta_1+\delta_{x^2}+(-1)^i(\delta_x+ \delta_{x^3})+(-1)^j(\delta_y+\delta_{x^2y})+(-1)^{i+j}(\delta_{xy}+\delta_{x^3y})\\
&+(-1)^k (\delta_t+\delta_{x^2t})+(-1)^{i+k}(\delta_{xt}+\delta_{x^3t})+(-1)^{j+k}(\delta_{yt}+\delta_{x^2yt})\\
&+(-1)^{i+j+k} (\delta_{xyt}+\delta_{x^3yt})\in Alg(H,\mathbbm{k}),
\end{split}
\end{equation*}
then $(x^jy^l,\chi)$ is a YD-pair \cite{AI} if and only if $\mathbbm{k}_{x^jy^l}^{\chi}$ is a one-dimensional Yetter-Drinfeld module.
So only
$\Omega_1(n_1,..,n_8)$ has a principal realization.
\end{rem}

Based on the principle of the lifting method, we classify the finite-dimensional Hopf algebras over $H$ such that their infinitesimal braidings are those Yetter-Drinfeld modules listed in Theorem A. Here is the second main result.
\begin{thm B}
Let $A$ be a finite-dimensional Hopf algebra over $H$ such that its
infinitesimal braiding is isomorphic to one of the Yetter-Drinfeld modules in Theorem A, then $A$ is isomorphic either to

$(1)~\mathfrak{U}_1(n_1,n_2,\ldots,n_8;I_1)$, see Definition~$\ref{def 1}$;

$(2)~\mathfrak{U}_2(n_1,n_2,n_3,n_4;I_2)$, see Definition~$\ref{def 2}$;

$(3)~\mathfrak{U}_4(n_1,n_2,n_3,n_4;I_4)$, see Definition~$\ref{def 4}$;

$(4)~\mathfrak{U}_5(n_1,n_2,n_3,n_4;I_5)$, see Definition~$\ref{def 5}$;

$(5)~\mathfrak{U}_8(n_1,n_2,n_3,n_4;I_8)$, see Definition~$\ref{def 8}$;

$(6)~\mathfrak{U}_{14}(I_{14})$, see Definition~$\ref{def 14}$; $\mathfrak{U}_{15}(\lambda,\mu)$, see Definition~$\ref{def 15}$;

\hspace{1.5em}$\mathfrak{U}_{16}(\lambda,\mu)$, see Definition~$\ref{def 16}$;
$\mathfrak{U}_{17}(I_{17})$, see Definition~$\ref{def 17}$;

\hspace{1.5em}$\mathfrak{U}_{18}(I_{18})$, see Definition~$\ref{def 18}$;
$\mathfrak{U}_{19}(\lambda,\mu)$, see Definition~$\ref{def 19}$;

\hspace{1.5em}$\mathfrak{U}_{20}(I_{20})$, see Definition~$\ref{def 20}$;
$\mathfrak{U}_{21}(\lambda,\mu)$, see Definition~$\ref{def 21}$;

\hspace{1.5em}$\mathfrak{U}_{22}(I_{22})$, see Definition~$\ref{def 22}$;
$\mathfrak{U}_{23}(I_{23})$, see Definition~$\ref{def 23}$;

\hspace{1.5em}$\mathfrak{U}_{24}(\lambda)$, see Definition~$\ref{def 24}$;
$\mathfrak{U}_{26}(I_{26})$, see Definition~$\ref{def 26}$;

\hspace{1.5em}$\mathfrak{U}_{27}(\lambda,\mu)$, see Definition~$\ref{def 27}$;
$\mathfrak{U}_{28}(I_{28})$, see Definition~$\ref{def 28}$;

\hspace{1.5em}$\mathfrak{U}_{29}(\lambda,\mu)$, see Definition~$\ref{def 29}$.

$(7)$ $\mathcal{B}(\Omega_{25})\sharp H$, see Proposition $\ref{25}$.
\end{thm B}

Except for the case
$(7)$, the remaining  families of Hopf algebras contain non-trivial lifting relations.

\begin{rem} Note that
$(1)$ $\dim \mathfrak{U}_1(n_1,n_2,\ldots,n_8;I_1)=2^{4+\sum\limits _{i=1}^8 n_i};$

$(2)$ For $k\in\{2,4,5,8\}$, $\dim \mathfrak{U}_k(n_1,n_2,n_3,n_4;I_k)=
2^{6+\sum\limits _{i=1}^4 n_i};$

$(3)$ For $i\in\{14,17,18,20,22,23,26,28\}$,
$\dim \mathfrak{U}_{i}(I_{i})=256;$

$(4)$ For $i\in\{15,16,19,21,27,29\}$, $\dim \mathfrak{U}_{i}(\lambda,\mu)=\dim \mathfrak{U}_{24}(\lambda)=256$.

\end{rem}

The paper is organized as follows. In section 2, we recall some
basics and notations of Yetter-Drinfeld modules, Nichols algebras,
Radford biproduct and Drinfeld double. In section $3$, we describe
the structures of $H$ and present the Drinfeld double $D=D(H^{cop})$
by generators and relations. We also determine the simple
$D$-modules. In section 4, we describe the simple objects of
$_H^H\mathcal{YD}$ by using the monoidal isomorphism of braided
categories $_H^H\mathcal{YD}\cong {_D\mathcal{M}}$. In section 5, we
obtain all the possible finite-dimensional Nichols algebras of
semisimple modules satisfying $\mathcal{B}(V)\cong \bigotimes_{i\in
I}\mathcal{B}(V_i)$. In section 6, based on the principle of the
lifting method, we classify the finite-dimensional Hopf algebras
over $H$ such that their infinitesimal braidings are those
Yetter-Drinfeld modules listed in Theorem A. Then we get Theorem B.

\section{Preliminaries}

\noindent$\mathbf{Conventions}$.
Throughout the paper, the ground field $\mathbbm{k}$ is an algebraically closed field of characteristic zero and we denote by $\xi$ a primitive $4$-th root of unity. For the references of Hopf algebra theory, one can consult \cite{M}, \cite{CK}, \cite{R}, \cite{SW}, etc.

If $H$ is a Hopf algebra over $\mathbbm{k}$, then $\triangle,~\varepsilon, ~S$ denote the comultiplication, the counit and the antipode, respectively. We use Sweedler's notation for the comultiplication and coaction, $e.g.$
, $\triangle(h)=h_{(1)}\otimes h_{(2)}$ for $h\in H$. We denote by $H^{op}$ the Hopf algebra with the opposite multiplication, by $H^{cop}$ the Hopf algebra with the opposite comultiplication, and by $H^{bop}$ the Hopf algebra $H^{op~cop}$. Denote by $G(H)$ the set of group-like elements of $H$.
\begin{equation*}
\mathcal{P}_{g,h}(H)=\{x\in H\mid \triangle(x)=x\otimes g+h\otimes x\},\hspace{1em}\forall ~g,~h\in G(H).
\end{equation*}
In particular, the linear space $\mathcal{P}(H)=\mathcal{P}_{1,1}(H)$ is called the set of primitive elements.

If $V$ is a $\mathbbm{k}$-vector space, $v\in V,~ f\in V^\ast$, we use either $f(v),~\langle f,v\rangle$ or $\langle v,f\rangle$ to denote the evaluation. Given $n\geq 0,$ we denote $\mathbb{Z}_n=\mathbb{Z}/n\mathbb{Z}$ and $\mathbb{I}_{0,n}=\{0,1,\cdots,n\}$. It is emphasized that the operations $ij$ and $i\pm j$ are considered modulo $n+1$ for $i,j\in \mathbb{I}_{0,n}$ when not specified.

\subsection{Yetter-Drinfeld modules and Nichols algebras}

Let $H$ be a Hopf algebra. A left Yetter-Drinfeld module $V$ over $H$ is a left $H$-module $(V,\cdot)$ and a left $H$-comodule $(V,\delta)$ with $\delta(v)=v_{(-1)}\otimes v_{(0)}\in H\otimes V$ for all $v\in V$, satisfying
\begin{center}
$\delta(h\cdot v)=h_{(1)}v_{(-1)}S(h_{(3)})\otimes h_{(2)}\cdot v_{(0)},~~~~~~\forall ~v\in V, ~h\in H.$
\end{center}
We denote by $^H_H\mathcal{YD}$ the category of finite-dimensional left Yetter-Drinfeld modules over $H$. It is a braided monoidal category: for $V,~W \in {^H_H\mathcal{YD}}$, the braiding $c_{V,W}:V\otimes W\rightarrow W\otimes V$ is given by
\begin{equation}
c_{V,W}(v\otimes w)=v_{(-1)}\cdot w\otimes v_{(0)},\hspace{1em}\forall ~v\in V,~ w\in W.\label{braing}
\end{equation}
\vspace{-.8cm}
\\In particular, $(V,c_{V,V})$ is a braided vector space, that is, $c := c_{V,V}$ is a linear isomorphism satisfying the braid equation
\begin{equation*}
(c\otimes id)(id \otimes c)(c\otimes id)=(id \otimes c)(c\otimes id)(id \otimes c).
\end{equation*}
Moreover, $^H_H\mathcal{YD}$ is rigid. Denote by $V^\ast$ the left dual defined by
\begin{center}
$\langle h\cdot f,v\rangle=\langle f, S(h)v\rangle,~~~~~f_{(-1)}\langle f_{(0)},v\rangle=S^{-1}(v_{(-1)})\langle f,v_{(0)}\rangle.$
\end{center}

\begin{defi}{\rm\cite [Definition 2.1]{AS2}}
~Let $H$ be a Hopf algebra and $V$ a Yetter-Drinfeld module over H. A braided $\mathbb{N}$-graded Hopf algebra $R=\bigoplus_{n\geq 0}R(n)$ in $^H_H\mathcal{YD}$ is called a Nichols algebra of V if
\begin{center}
$\mathbbm{k}\simeq R(0), ~V\simeq R(1),~R(1)=\mathcal{P}(R),~R$ is generated as an algebra by $R(1)$.
\end{center}

\end{defi}

For any $V\in {^H_H\mathcal{YD}}$ there is a Nichols algebra $\mathcal{B}(V)$ associated to it. It is the quotient of the tensor algebra $T(V)$ by the largest homogeneous two-sided ideal $I$ satisfying:

~~$\bullet~I$ is generated by homogeneous elements of degree $\geq 2$.

~~$\bullet~\triangle(I)\subseteq I\otimes T(V)+T(V)\otimes I$, i.e., it is also a coideal.
\\ In such case, $\mathcal{B}(V)=T(V)/I$. See {\rm\cite [Section 2.1]{AS2}} for details.

\begin{rem}\label{rem}
As well known, the Nichols algebra as an algebra and a coalgebra is completely determined by the braided space.
If $W\subseteq V$ is a subspace such that $c(W\otimes W)\subseteq W\otimes W$, then $\dim \mathcal{B}(V)=\infty$ if $\dim \mathcal{B}(W)=\infty$.
In particular, if $V$ contains a non-zero element $v$ such that $c(v\otimes v)=v\otimes v$, then $\dim \mathcal{B}(V)=\infty$.
\end{rem}

\begin{lem}{\rm\cite [Theorem 2.2]{Gn00}}\label{direct sum}
Let $(V,c)$ be a braided vector space, $V=\bigoplus_{i=1}^{n}V_i$, such that each $\mathcal{B}(V_i)$ is finite-dimensional. Then $\dim \mathcal{B}(V)\geq\prod_{i=1}^{n}$ $\dim \mathcal{B}(V_i)$. Furthermore, the equality holds if and only if $b_{ij}=b_{ji}^{-1}$ for all $i\neq j$ $(b_{ij}=c\mid_{V_i\otimes V_j})$.
\end{lem}
\subsection{Bosonization and Hopf algebras with a projection}
Let $R$ be a Hopf algebra in $^H_H\mathcal{YD}$ and denote the coproduct by $\triangle_R(r)=r^{(1)}\otimes r^{(2)}$ for $r\in R$. We define the Radford biproduct or bosonization $R \sharp H$ as follows: as a vector space, $R \sharp H=R\otimes H$, and the multiplication and comultiplication are given by the smash product and smash-coproduct, respectively:
\begin{center}
$(r\sharp g)(s\sharp h)=r(g_{(1)}\cdot s)\sharp g_{(2)}h, ~~\triangle(r\sharp g)=r^{(1)}\sharp (r^{(2)})_{(-1)}g_{(1)}\otimes (r^{(2)})_{(0)}\sharp g_{(2)}$.
\end{center}
Clearly, the map
$\iota: H\rightarrow R \sharp H, ~h\mapsto 1\sharp h, ~\forall ~h\in H$
is injective and the map
\begin{equation*}
\pi: R \sharp H \rightarrow H, ~r\sharp h\mapsto \varepsilon_R(r)h, ~\forall ~r\in R, ~h\in H
\end{equation*}
is surjective such that $\pi \circ \iota=id_H$. Moreover, it holds that
\begin{equation*}
R=(R \sharp H)^{coH}=\{x\in R \sharp H\mid (id\otimes \pi)\triangle(x)=x\otimes1\}.
\end{equation*}

Conversely, if $A$ is a Hopf algebra with bijective antipode and $\pi:A\rightarrow H$ is a Hopf algebra epimorphism admitting a Hopf algebra section $\iota: H\rightarrow A$ such that $\pi \circ \iota=id_H$. Then $R=A^{co \pi}$ is a braided Hopf algebra in $^H_H\mathcal{YD}$ and $A\cong R \sharp H$ as Hopf algebras. See \cite{R85} and \cite{R} for more details.

\subsection{The Drinfeld double}

Let $H$ be a finite-dimensional Hopf algebra over $\mathbbm{k}$. The Drinfeld double $D(H)=H^{\ast cop}\otimes H$ is a Hopf algebra with the tensor product coalgebra structure and algebra structure defined by
\begin{center}
$(p\otimes a)(q\otimes b)=p\langle q_{(3)},a_{(1)}\rangle q_{(2)}\otimes a_{(2)}\langle q_{(1)}, S^{-1}(a_{(3)})\rangle b,~~~\forall~ p,~q\in H^{*},~a,~b\in H.$
\end{center}
\begin{pro}{\rm\cite [Proposition 10.6.16]{M}}
Let $H$ be a finite-dimensional Hopf algebra. Then the Yetter-Drinfeld category $^H_H\mathcal{YD}$ can be identified with the category $_{D(H^{cop})}\mathcal{M}$ of left modules over the Drinfeld double $D(H^{cop}).$
\end{pro}

\section{The Hopf algebra $H$ and its Drinfeld double}
In this section, we will recall the structure of the 16-dimensional non-trivial semisimple Hopf algebra $H=H_{b:1}$ in \cite{K00} and present the Drinfeld Double $D=D(H^{cop})$ by generators and relations. Then we will determine all simple left $D$-modules.

\begin{defi}\label{def3.1}
As an algebra $H$ is generated by $x,~y,~t$ satisfying the relations
\begin{eqnarray}
x^4=1,\hspace{1em}y^2=1,\hspace{1em}t^2=1,\hspace{1em}xy=yx,\hspace{1em}tx=x^{-1}t, \hspace{1em}ty=yt,\label{3.1}
\end{eqnarray}
and its coalgebra is given by
\begin{flalign}
&\triangle(x)=x\otimes x,~~~~\triangle(y)=y\otimes y,~~~~\varepsilon(x)=\varepsilon(y)=1,\label{3.2}\\
&\triangle(t)=\frac{1}{2}[(1+y)t\otimes t+(1-y)t\otimes x^2 t],~~~~~\varepsilon(t)=1,\label{3.3}
\end{flalign}
and its antipode is given by
\begin{equation}
S(x)=x^{-1},\hspace{1em}S(y)=y,\hspace{1em}S(t)=\frac{1}{2}[(1+y)t+(1-y)x^2 t].\nonumber
\end{equation}

\end{defi}

\begin{rem}
$1.~ G(H)=\langle x\rangle \times \langle y\rangle,~\mathcal{P}_{1,g}(H)=\mathbbm{k}\{1-g\}$ for $g\in G(H)$ and a linear basis of $H$ is given by $\{x^i, ~x^iy, ~x^it, ~x^iyt,~i\in \mathbb{I}_{0,3}\}$.

$2.~$ Denote by $\{(x^i)^*, ~(x^iy)^*, ~(x^it)^*, ~(x^iyt)^*,~i\in \mathbb{I}_{0,3}\}$ the the dual basis of Hopf algebra $H^*$. Let
\begin{flalign*}
&a=[(1-x+x^2-x^3)(1+y)(1+t)]^*,~~~b=[(1+x+x^2+x^3)(1-y)(1+t)]^*,\\
&c=[(1+x+x^2+x^3)(1-y)(1-t)]^*,~~~d=[(1+\xi x-x^2-\xi x^3)(1+y)(1+t)]^*.
\end{flalign*}
Then using the multiplication table induced by the relations of $H$, it follows that
\begin{flalign*}
&a^2=b^2=c^2=1,\hspace{2em}ab=ba, \hspace{2em}ac=ca,\hspace{2em}bc=cb,\\
&d^2=a,\hspace{3em}da=ad,\hspace{3em}db=cd,\hspace{3em}dc=bd,\\
&\triangle(a)=a\otimes a,\hspace{3em}\triangle(b)=b\otimes b,\hspace{2em}\triangle(c)=c\otimes c,\\
&\triangle(d)=\frac{1}{2}(d+bcd)\otimes d+\frac{1}{2}(d-bcd)\otimes ad.
\end{flalign*}

$3.~$ The automorphisms of $H$ are given in Table $1$.

\begin{center}
\begin{tabular}{|c|c|c|c|c|c|c|c|c|c|c|c|c|c|c|c|}
  \hline
   & $x$ & $y$ & $t$ &  & $x$ & $y$ & $t$ &  & $x$ & $y$ & $t$ &  & $x$ & $y$ & $t$  \\
  \hline
  $\tau_1$ & $x$ & $y$ & $t$ & $\tau_9$ & $x^3$ & $y$ & $t$ & $\tau_{17}$ & $xy$ & $y$ & $t$ & $\tau_{25}$ & $x^3y$ & $y$ & $t$ \\
  \hline
  $\tau_2$ & $x$ & $y$ & $xt$ & $\tau_{10}$ & $x^3$ & $y$ & $xt$ & $\tau_{18}$ & $xy$ & $y$ & $xt$ & $\tau_{26}$ & $x^3y$ & $y$ & $xt$ \\
  \hline
  $\tau_3$ & $x$ & $y$ & $x^2t$ & $\tau_{11}$ & $x^3$ & $y$ & $x^2t$ & $\tau_{19}$ & $xy$ & $y$ & $x^2t$ & $\tau_{27}$ & $x^3y$ & $y$ & $x^2t$ \\
  \hline
  $\tau_4$ & $x$ & $y$ & $x^3t$ & $\tau_{12}$ & $x^3$ & $y$ & $x^3t$ & $\tau_{20}$ & $xy$ & $y$ & $x^3t$ & $\tau_{28}$ & $x^3y$ & $y$ & $x^3t$ \\
  \hline
  $\tau_5$ & $x$ & $y$ & $yt$ & $\tau_{13}$ & $x^3$ & $y$ & $yt$ & $\tau_{21}$ & $xy$ & $y$ & $yt$ & $\tau_{29}$ & $x^3y$ & $y$ & $yt$ \\
  \hline
  $\tau_6$ & $x$ & $y$ & $xyt$ & $\tau_{14}$ & $x^3$ & $y$ & $xyt$ &  $\tau_{22}$& $xy$ & $y$ & $xyt$ & $\tau_{30}$ & $x^3y$ & $y$ & $xyt$ \\
  \hline
  $\tau_7$ & $x$ & $y$ & $x^2yt$ & $\tau_{15}$ & $x^3$ & $y$ & $x^2yt$ & $\tau_{23}$ & $xy$ & $y$ & $x^2yt$ & $\tau_{31}$ & $x^3y$ & $y$ & $x^2yt$ \\
  \hline
  $\tau_8$ & $x$ & $y$ & $x^3yt$ & $\tau_{16}$ & $x^3$ & $y$ & $x^3yt$ & $\tau_{24}$ & $xy$ & $y$ & $x^3yt$ & $\tau_{32}$ & $x^3y$ & $y$ & $x^3yt$ \\
  \hline
\end{tabular}
\end{center}
\begin{center}
Table $1.$ Automorphisms of $H$
\end{center}
\end{rem}

Now we describe the Drinfeld Double $D(H^{cop})$ of $H^{cop}$.
\begin{pro}
$D:=D(H^{cop})$ as a coalgebra is isomorphic to tensor coalgebra $H^{*bop}\otimes H^{cop}$, and as an algebra is generated by the elements $a,~b,~c,~d,~x,~y,~t$ such that $x,~y,~t$ satisfying the relations of $H^{cop}$, $a,~b,~c,~d$ satisfying the relations of $H^{*bop}$ and
\begin{flalign*}
&xa=ax, \hspace{2em}xb=bx, \hspace{2.7em}xc=cx, \hspace{2.5em}xd=bcdx,\\
&ya=ay, \hspace{2em}yb=by, \hspace{2.8em}yc=cy,\hspace{2.7em}yd=dy,\\
&ta=at,\hspace{2.4em}tb=bx^2t, \hspace{2em}tc=cx^2t, \hspace{2em}td=adyt.
\end{flalign*}
\end{pro}
\begin{proof}
After a direct computation, we have that
\begin{flalign*}
\triangle_{H^{cop}}^2(x)=&x\otimes x\otimes x, ~~~\triangle_{H^{cop}}^2(y)=y\otimes y\otimes y,\\
\triangle_{H^{*bop}}^2(a)=&a\otimes a\otimes a,
~~~\triangle_{H^{*bop}}^2(b)=b\otimes b\otimes b,
~~~\triangle_{H^{*bop}}^2(c)=c\otimes c\otimes c,\\
\triangle_{H^{cop}}^2(t)=&\frac{1}{4}[t\otimes (1+y)t\otimes (1+y)t+x^2t\otimes (1-y)t\otimes(1+y)t\\
&+x^2t\otimes x^2(1+y)t\otimes
(1-y)t +t\otimes x^2(1-y)t\otimes (1-y)t],\\
\triangle_{H^{*bop}}^2(d)=&\frac{1}{4}[d\otimes (d+bcd)\otimes (d+bcd)+ad\otimes (d-bcd)\otimes (d+bcd)\\
&+ad\otimes a(d+bcd)
\otimes (d-bcd)+d\otimes a(d-bcd)\otimes (d-bcd)].
\end{flalign*}
 It follows that
\begin{flalign*}
xa=&\langle a,x\rangle ax \langle a,S(x)\rangle=ax,
\hspace{1em}ya=\langle a,y\rangle ay \langle a,S(y)\rangle=ay,\\
xb=&\langle b,x\rangle bx \langle b,S(x)\rangle=bx,
\hspace{1.5em}yb=\langle b,y\rangle by \langle b,S(y)\rangle=by,\\
xc=&\langle c,x\rangle cx \langle c,S(x)\rangle=cx,
\hspace{1.5em}yc=\langle c,y\rangle cy \langle c,S(y)\rangle=cy,\\
xd=&\frac{1}{4}[\langle d,x\rangle (d+bcd)x \langle d+bcd,S(x)\rangle+\langle ad,x\rangle (d-bcd)x \langle d+bcd,S(x)\rangle]=bcdx,\\
yd=&\frac{1}{4}[\langle d,y\rangle (d+bcd)y \langle d+bcd,S(y)\rangle+\langle ad,y\rangle (d-bcd)y \langle d+bcd,S(y)\rangle]=dy,\\
ta=&\frac{1}{4}[\langle a,t\rangle a(1+y)t\langle a,S(t+yt)\rangle+\langle a,x^2t\rangle a(1-y)t\langle a,S(t+yt)\rangle]=at,\\
tb=&\frac{1}{4}[\langle b,x^2t\rangle bx^2(1+y)t\langle b,S(t-yt)\rangle+\langle b,t\rangle bx^2(1-y)t\langle b,S(t-yt)\rangle]=bx^2t,\\
tc=&\frac{1}{4}[\langle c,x^2t\rangle cx^2(1+y)t\langle c,S(t-yt)\rangle+\langle c,t\rangle cx^2(1-y)t\langle c,S(t-yt)\rangle]=cx^2t,\\
td=&\frac{1}{16}[\langle ad,t\rangle a(d+bcd)(1+y)t\langle d-bcd,S(t+yt)\rangle\\
&+\langle ad,x^2t\rangle a(d+bcd)(1-y)t\langle d-
bcd,S(t+yt)\rangle\\
&+\langle d,t\rangle a(d-bcd)(1+y)t\langle d-bcd,S(t+yt)\rangle\\
&+\langle d,x^2t\rangle a(d-bcd)
(1-y)t\langle d-bcd,S(t+yt)\rangle]=adyt.
\end{flalign*}

This completes the proof.
\end{proof}

We begin by describing the one-dimensional $D$-modules.
\begin{lem}\label{1simple}
There are $32$ non-isomorphic one-dimensional simple modules $\mathbbm{k}_{\chi_{i,j,k,l}}$ given by the characters $\chi_{i,j,k,l},~ 0\leq i,~k,~l<2, ~0\leq j<4$, where
\begin{flalign*}
&\chi_{i,j,k,l}(x)=(-1)^i,~~~\chi_{i,j,k,l}(y)=(-1)^j, ~~~\chi_{i,j,k,l}(t)=(-1)^k,\\
&\chi_{i,j,k,l}(a)=(-1)^j, ~~~\chi_{i,j,k,l}(b)=(-1)^l,~~~\chi_{i,j,k,l}(c)=(-1)^l,~~~\chi_{i,j,k,l}(d)=\xi^j.
\end{flalign*}
Moreover, any one-dimensional $D$-module is isomorphic to $\mathbbm{k}_{\chi_{i,j,k,l}}$ for some $0\leq i,k,l<2, ~0\leq j<4$.
\end{lem}
\begin{proof}
Let $\lambda: D\rightarrow \mathbbm{k}$ be a character and write
\begin{equation*}
\lambda(x)=\lambda_1,~~\lambda(y)=\lambda_2,~~\lambda(t)=\lambda_3, ~~\lambda(a)=\lambda_4,~~\lambda(b)=\lambda_5,~~\lambda(c)=\lambda_6, ~~\lambda(d)=\lambda_7.
\end{equation*}
Since $x^4=1,~y^2=t^2=a^2=b^2=c^2=1$, it follows that
\begin{equation*}
\lambda_1^4=1,\hspace{1em}\lambda_2^2=\lambda_3^2=\lambda_4^2=\lambda_5^2=\lambda_6^2=1.
\end{equation*}
By $tx=x^{-1}t,$ we have $\lambda_1=\lambda_1^3$, that is $\lambda_1^2=1$. As $db=cd,~dc=bd$, then $\lambda_5=\lambda_6$.
Because $td=adyt$, it follows that $\lambda_2\lambda_4=1$, i.e., $\lambda_2=\lambda_4$. From $d^2=a$, we get $\lambda_7^2=\lambda_4$.
Hence, $\lambda$ is completely determined by $\lambda(x),~\lambda(y),~\lambda(t),~\lambda(b).$ Let
\begin{equation*}
\lambda(x)=(-1)^i,\hspace{1em}\lambda(y)=(-1)^j,\hspace{1em}\lambda(t)=(-1)^k,\hspace{1em}\lambda(b)=(-1)^l.
\end{equation*}
It is clear that these modules are pairwise non-isomorphic and any one-dimensional
$D$-module is isomorphic to some $\mathbbm{k}_{\chi_{i,j,k,l}}$, where $0\leq i,~k,~l<2, ~0\leq j<4$.
\end{proof}

Next we describe the simple $D$-modules of dimension two. For convenience, let
\begin{flalign*}
&\Omega^1=\{(0,j,k,l,m,n)\mid j,k,l,m,n\in \mathbb{Z}_2, m+n=1\},\\
&\Omega^2=\{(i,j,0,l,m,n)
\mid i,j,l,m,n\in \mathbb{Z}_2, m+n=0, j+l=1\}, ~~~~\Omega=\Omega^1\cup \Omega^2,\\
&\Lambda^1=\{(1,j,k,l)\mid k\in \mathbb{Z}_4, j,l\in \mathbb{Z}_2\},
~~~\Lambda^2=\{(1,j,k,l)\mid j\in \mathbb{Z}_4, k,l\in \mathbb{Z}_2\}.
\end{flalign*}
Clearly, $|\Omega|=24,~|\Lambda^1|=|\Lambda^2|=16$.

\begin{lem}\label{2simple}
For any $6$-tuple $(i,j,k,l,m,n)\in \Omega$, there exists a simple left $D$-module $V_{i,j,k,l,m,n}$ of dimension $2$ with the action on a fixed basis given by
\begin{flalign*}
&[x]=
\left(
  \begin{array}{cc}
    (-1)^i & 0 \\
    0 & (-1)^{i+m+n} \\
  \end{array}
\right),
\hspace{1em}[y]=
\left(
  \begin{array}{cc}
    (-1)^j & 0 \\
    0 & (-1)^{j} \\
  \end{array}
\right),\\
&[t]=
\left(
  \begin{array}{cc}
    (-1)^k & 0 \\
    0 & (-1)^{j+l+k} \\
  \end{array}
\right),
\hspace{1em}[a]=
\left(
  \begin{array}{cc}
    (-1)^l & 0 \\
    0 & (-1)^{l} \\
  \end{array}
\right),\\
&[b]=
\left(
  \begin{array}{cc}
    (-1)^m & 0 \\
    0 & (-1)^{n} \\
  \end{array}
\right),
\hspace{1em}[c]=
\left(
  \begin{array}{cc}
    (-1)^n & 0 \\
    0 & (-1)^{m} \\
  \end{array}
\right),
\hspace{1em}[d]=
\left(
  \begin{array}{cc}
    0 & 1 \\
    (-1)^l & 0 \\
  \end{array}
\right).
\end{flalign*}
For any $4$-tuple $(i,j,k,l)\in \Lambda^1$, there exists a simple left $D$-module $W_{i,j,k,l}$ of dimension $2$ with the action on a fixed basis given by
\begin{flalign*}
&[x]=
\left(
  \begin{array}{cc}
    \xi & 0 \\
    0 & \xi^{-1} \\
  \end{array}
\right),
~~~~~[y]=
\left(
  \begin{array}{cc}
    (-1)^j & 0 \\
    0 & (-1)^{j} \\
  \end{array}
\right),
~~~[t]=
\left(
  \begin{array}{cc}
    0 & 1 \\
    1 & 0 \\
  \end{array}
\right),\\
&[a]=
\left(
  \begin{array}{cc}
    (-1)^k & 0 \\
    0 & (-1)^{k} \\
  \end{array}
\right),
~~~~~~~~~[b]=
\left(
  \begin{array}{cc}
    (-1)^l & 0 \\
    0 & (-1)^{l+1} \\
  \end{array}
\right),\\
&[c]=
\left(
  \begin{array}{cc}
    (-1)^l & 0 \\
    0 & (-1)^{l+1} \\
  \end{array}
\right),
~~~~~~~~[d]=
\left(
  \begin{array}{cc}
    \xi^k & 0 \\
    0 & (-1)^{j+k}\xi^k\\
  \end{array}
\right).
\end{flalign*}
For any $4$-tuple $(i,j,k,l)\in \Lambda^2$, there exists a simple left $D$-module $U_{i,j,k,l}$ of dimension $2$ with the action on a fixed basis given by
\begin{flalign*}
&[x]=
\left(
  \begin{array}{cc}
    \xi & 0 \\
    0 & \xi^{-1} \\
  \end{array}
\right),
~~~~~[y]=
\left(
  \begin{array}{cc}
    (-1)^j & 0 \\
    0 & (-1)^{j} \\
  \end{array}
\right),
~~~[t]=
\left(
  \begin{array}{cc}
    0 & 1 \\
    1 & 0 \\
  \end{array}
\right),\\
&[a]=
\left(
  \begin{array}{cc}
    (-1)^k & 0 \\
    0 & (-1)^{k} \\
  \end{array}
\right),
~~~~~~~~~[b]=
\left(
  \begin{array}{cc}
    (-1)^l & 0 \\
    0 & (-1)^{l+1} \\
  \end{array}
\right),\\
&[c]=
\left(
  \begin{array}{cc}
    (-1)^{l+1} & 0 \\
    0 & (-1)^{l} \\
  \end{array}
\right),
~~~~~~~~[d]=
\left(
  \begin{array}{cc}
    0 & \xi^j \\
    (-1)^{j+k}\xi^j & 0\\
  \end{array}
\right).
\end{flalign*}
Moreover, if $V$ is a simple $D$-module of dimension $2$, then $V$ is isomorphic to $V_{i,j,k,l,m,n}$ for some $(i,j,k,l,m,n)\in \Omega$ or $W_{i,j,k,l}$ for some $(i,j,k,l)\in \Lambda_1$ or $U_{i,j,k,l}$ for some $(i,j,k,l)\in \Lambda_2$ and $V_{i,j,k,l,m,n}\cong V_{p,q,r,\kappa,\mu,\nu}$ if and only if $(i,j,k,l,m,n)=(p,q,r,\kappa,\mu,\nu)$, $W_{i,j,k,l}\cong W_{p,q,r,\kappa}$ if and only if $(i,j,k,l)=(p,q,r,\kappa)$, $U_{i,j,k,l}\cong U_{p,q,r,\kappa}$ if and only if $(i,j,k,l)=(p,q,r,\kappa)$.
\end{lem}
\begin{proof} Since the elements $x,~y,~a,~b,~c$ commute each other and $x^4=1, ~y^2=a^2=b^2=c^2=1$, the matrices defining $D$-action on $V$ can be of the form
\begin{flalign*}
&[x]=
\left(
  \begin{array}{cc}
    x_1 & 0 \\
    0 & x_2 \\
  \end{array}
\right),
~~~~~[y]=
\left(
  \begin{array}{cc}
    y_1 & 0 \\
    0 & y_2 \\
  \end{array}
\right),
~~~[t]=
\left(
  \begin{array}{cc}
    t_1 & t_2 \\
    t_3 & t_4 \\
  \end{array}
\right),\\
&[a]=
\left(
  \begin{array}{cc}
    a_1 & 0 \\
    0 & a_2 \\
  \end{array}
\right),
~~~~~~[b]=
\left(
  \begin{array}{cc}
    b_1 & 0 \\
    0 & b_2 \\
  \end{array}
\right),
~~~[c]=
\left(
  \begin{array}{cc}
    c_1 & 0 \\
    0 & c_2 \\
  \end{array}
\right),
~~[d]=
\left(
  \begin{array}{cc}
    d_1 & d_2 \\
    d_3 & d_4\\
  \end{array}
\right),
\end{flalign*}
where $x_1^4=x_2^4=y_1^2=y_2^2=a_1^2=a_2^2=b_1^2=b_2^2=c_1^2=c_2^2=1$.
Since $ty=yt, ~dy=yd$, it follows that
\begin{flalign*}
&t_2(y_1-y_2)=0, ~~~~~t_3(y_1-y_2)=0, ~~~~~d_2(y_1-y_2)=0, ~~~~~d_3(y_1-y_2)=0.
\end{flalign*}
If $y_1\neq y_2$, then $t_2=t_3=d_2=d_3=0$ and whence $V$ is not simple $D$-module, a contradiction. Thus we have $y_1= y_2$. Similarly, the relations $ta=at, ~da=ad$ give $a_1=a_2$.

Since $t^2=1,~tx=x^{-1}t$, it follows that
\begin{flalign}
&t_1^2+t_2t_3=1,~~~~t_4^2+t_2t_3=1,~~~~t_2(t_1+t_4)=0,~~~~t_3(t_1+t_4)=0,\label{(1)}\\
&t_1x_1(1-x_1^2)=0,~~t_4x_2(1-x_2^2)=0,~~t_2(x_2-x_1^3)=0,~~t_3(x_1-x_2^3)=0.\label{(2)}
\end{flalign}
From the relations $tb=bx^{2}t,~tc=cx^2t, ~td=adyt$, we have the relations
\begin{flalign}
&t_1b_1(1-x_1^2)=0,~~~t_4b_2(1-x_2^2)=0,~~~t_2(b_2-b_1x_1^2)=0,~~~t_3(b_1-b_2x_2^2)=0,\label{(3)}\\
&t_1c_1(1-x_1^2)=0,~~~t_4c_2(1-x_2^2)=0,~~~t_2(c_2-c_1x_1^2)=0,~~~t_3(c_1-c_2x_2^2)=0,\label{(4)}\\
&t_1d_1+t_2d_3=a_1y_1(t_1d_1+t_3d_2),~~~~~~~~t_1d_2+t_2d_4=a_1y_1(t_2d_1+t_4d_2),\label{(5)}\\
&t_3d_1+t_4d_3=a_1y_1(t_1d_3+t_3d_4),~~~~~~~~t_3d_2+t_4d_4=a_1y_1(t_2d_3+t_4d_4).\label{(6)}
\end{flalign}
By the relations $d^2=a,~db=cd, ~dc=bd,~xd=bcdx$, it follows that
\begin{flalign}
&d_1^2+d_2d_3=a_1,~~~d_4^2+d_2d_3=a_1,~~~d_2(d_1+d_4)=0,~~~d_3(d_1+d_4)=0,\label{(7)}\\
&d_1(b_1-c_1)=0,~~~d_2(b_2-c_1)=0,~~~d_3(b_1-c_2)=0,~~~d_4(b_2-c_2)=0,\label{(8)}\\
&d_1(b_1-c_1)=0,~~~d_2(c_2-b_1)=0,~~~d_3(c_1-b_2)=0,~~~d_4(b_2-c_2)=0,\label{(9)}\\
&x_1d_1(1-b_1c_1)=0,~x_2d_4(1-b_2c_2)=0,~d_2(x_1-b_1c_1x_2)=0,~d_3(x_2-b_2c_2x_1)=0.\label{(10)}
\end{flalign}

If $x_1^2= 1,~ x_2^2\neq 1$, then by (\ref{(2)}) we have that
$t_4=0$ and whence from (\ref{(1)}) we get that $t_2t_3=1$. Thus by
relation (\ref{(2)}), it follows that $x_2=x_1$, contradiction.
Similarly, the case $x_1^2\neq 1, ~x_2^2=1$ doesn't appear.

If $x_1^2\neq 1, ~x_2^2\neq 1$, then by (\ref{(2)}) we have that $t_1=t_4=0$ and whence by (\ref{(1)})
we get $t_2t_3=1.$ From the relations $(\ref{(2)})-(\ref{(6)})$, it follows that
\begin{equation*}
x_2=x_1^3,~~~~b_2=-b_1,~~~~c_2=-c_1,~~~~d_4=a_1y_1d_1,~ ~~~t_2d_3=a_1y_1t_3d_2.
\end{equation*}
If $d_2=d_3=0$, then by relations $(\ref{(7)})-(\ref{(10)})$, we get that \begin{equation*}
d_1^2=a_1=d_4^2,\hspace{2em}b_1=c_1,\hspace{2em}b_2=c_2.
\end{equation*}
Whence the matrices defining the action on $V$ are of the form $W_{i,j,k,l}$, where $i\in\{1,3\},~0\leq k<4,~ 0\leq j,~l<2$. If $d_1=d_4=0$, then
\begin{equation*}
d_2d_3=a_1,\hspace{2em} b_2=c_1,\hspace{2em}b_1=c_2.
\end{equation*}
Whence the module is isomorphic to $U_{i,j,k,l}$, where $i\in\{1,3\},~0\leq j<4, ~0\leq k,~l<2$. If $d_1d_4\neq 0,~ d_2d_3\neq 0$, then $b_1=b_2$, a contradiction.
Other conditions are contained in the above conditions. A direct computation shows that

\hspace{2em}when $j+k=0$, $W_{i,j,k,l}\cong W_{-i,j,k,l+1}, ~U_{i,j,k,l}\cong U_{-i,j,k,l+1};$

\hspace{2em}when $j+k=1$, $W_{i,j,k,l}\cong W_{-i,j,k+2,l+1}, ~U_{i,j,k,l}\cong U_{-i,j+2,k,l+1}$.

If $x_1^2=1, ~x_2^2=1$, suppose that $t_1t_4\neq 0,~t_2=t_3=0$, then it is clear that $V$ is simple if and only if $d_2d_3\neq 0$,
whence by (\ref{(7)}) we have that $d_1+d_4=0$. By (\ref{(1)}), (\ref{(5)}), (\ref{(6)}), (\ref{(8)}), (\ref{(10)}) we have that
\begin{flalign*}
&t_1^2=t_4^2=1,~~~~~~~~(1-a_1y_1)t_1d_1=0,~~~~~~~~(1-a_1y_1)t_4d_4=0,\\
&t_4=a_1y_1t_1,~~~~~~~~b_2=c_1,~~~~~~b_1=c_2,~~~~~~x_2=b_2c_2x_1.
\end{flalign*}
If $d_1=d_4=0$, then the matrices defining the action on $V$ are of the form $V_{i,j,k,l,m,n}$, where $i,~j,~k,~l,~m,~n\in \{0,1\};$ If $d_1\neq 0, ~d_4\neq 0$, the module is not simple. Suppose that $t_1=t_4=0,$ then by (\ref{(1)})-(\ref{(10)}), we get that
\begin{equation*}
t_2t_3=1,~~~x_1=x_2,~~~b_1=b_2,~~~c_1=c_2£¬~~~t_2d_3=a_1y_1t_3d_2,~~~d_4=a_1y_1d_1. \end{equation*}
Whence the module is isomorphic to $V_{i,j,k,l,m,n}$.
Similarly, for the case
\begin{equation*}
t_1+t_4=0,\quad t_1\neq 0, \quad t_4\neq 0,\quad t_2\neq 0,\quad t_3\neq 0,
\end{equation*}
the module is isomorphic to $V_{i,j,k,l,m,n}$.
It is clear that $V_{i,j,k,l,m,n}$ is simple if and only if $m+n\neq 0$ or $j+l\neq 0$. A direct computation shows that

\hspace{2em}when $m+n=0,~j+l=1,~ V_{i,j,k,l,m,n}\cong V_{i,j,k+1,l,n,m}$;

\hspace{2em}when $m+n=1,~j+l=0, ~V_{i,j,k,l,m,n}\cong V_{-i,j,k,l,n,m}$;

\hspace{2em}when $m+n=1,~j+l=1, ~V_{i,j,k,l,m,n}\cong V_{-i,j,k+1,l,n,m}$.

Obviously, $V_{i,j,k,l,m,n}$ with $(i,j,k,l,m,n)\in \Omega$ is not isomorphic to the above two modules
and $W_{i,j,k,l}\ncong U_{i',j',k',l'}$ for $(i,j,k,l)\in \Lambda^1,~(i',j',k',l')\in \Lambda^2$.

We claim that $V_{i,j,k,l,m,n}\cong V_{i',j',k',l',m',n'}$ if and only if $(i,j,k,l,m,n)=(i',j',k',l',m',n')$ in $\Omega$.
Assume that $\Phi: V_{i,j,k,l,m,n}\rightarrow V_{i',j',k',l',m',n'}$ is an isomorphism of $D$-modules. Denote by $[\Phi]=(p_{i,j})_{i,j=1,2}$ the matrix of $\Phi$ with respect to the given basis. Since
\begin{equation*}
[y][\Phi]=[\Phi][y],\hspace{1em}[a][\Phi]=[\Phi][a],\hspace{1em}[d][\Phi]=[\Phi][d],
\end{equation*}
we have that $j=j', ~l=l', ~p_1=p_4,~ p_3=(-1)^lp_2$. Since
\begin{equation*}
 [x][\Phi]=[\Phi][x],\hspace{1em}[t][\Phi]=[\Phi][t], \hspace{1em}[b][\Phi]=[\Phi][b],
\end{equation*}
we have that
\begin{flalign}
&[(-1)^{i'}-(-1)^i]p_1=0,\hspace{4.8em}[(-1)^{i'}-(-1)^{i+m+n}]p_2=0, \label{(11)}\\
&[(-1)^{i'+m'+n'}-(-1)^i]p_2=0, \hspace{2em}[(-1)^{i'+m'+n'}-(-1)^{i+m+n}]p_1=0,\label{(12)}\\
&[(-1)^{k'}-(-1)^k]p_1=0,\hspace{4.5em}[(-1)^{k'}-(-1)^{j+l+k}]p_2=0,\label{(13)}\\
&[(-1)^{m'}-(-1)^m]p_1=0,\hspace{4em}[(-1)^{m'}-(-1)^n]p_2=0,\label{(14)}\\
&[(-1)^{n'}-(-1)^m]p_2=0,\hspace{4.2em}[(-1)^{n'}-(-1)^n]p_1=0.\label{(15)}
\end{flalign}
If $m+n=0$, then $k=k'=0$ and by (\ref{(11)}), (\ref{(12)}) we have that $i=i',~m'+n'=0$. Whence by (\ref{(14)}), (\ref{(15)}), it follows that $m'=m,~n'=n$. If $m+n=1$, suppose $m'+n'=0$, then
by (\ref{(11)}), (\ref{(12)}) we have that $i=-i$, contradiction. So $m'+n'=1$. Whence $i=i'=0$, $p_2=0$. Then by (\ref{(13)})-(\ref{(15)}), $k=k',~m=m',~n=n'$.
Thus the claim follows.

Similarly, we have that $W_{i,j,k,l}\cong W_{i',j',k',l'}$ if and
only if $(i,j,k,l)=(i',j',k',l')$ in $\Lambda^1$; $U_{i,j,k,l}\cong
U_{i',j',k',l'}$ if and only if $(i,j,k,l)=(i',j',k',l')$ in
$\Lambda^2$.
\end{proof}

\begin{thm}\label{simple}
There are $88$ simple left D-modules up to isomorphism, among which $32$ one-dimensional
objects given by Lemma $\ref{1simple}$ and $56$ two-dimensional objects given by Lemma $\ref{2simple}$.
\end{thm}
\begin{proof} Assume there is a simple module of dimension $d>2$ and let $n$ denote
the number of simple $d$-dimensional modules pairwise non-isomorphic. By Lemmas $\ref{1simple}$ \& $\ref{2simple}$, we have
$$
32\cdot1^2+56\cdot2^2+nd^2=256+nd^2\le \dim (D^*)_0=\dim D^*=256.
$$
Then $n=0$.
\end{proof}

\section{The category $^H_H\mathcal{YD}$}

In this section, by using the monoidal isomorphism
$^H_H\mathcal{YD}\cong {_D\mathcal M}$, we will describe the simple
objects in $^H_H\mathcal{YD}$ and determine their braidings.

\begin{pro}\label{1dim}
Let $\mathbbm{k}_{\chi_{i,j,k,l}}=\mathbbm{k}\{v\}$ be a one-dimensional $D$-module with $i,~k,~l\in \mathbb{I}_{0,1},~ j\in \mathbb{I}_{0,3}$. Then $\mathbbm{k}_{\chi_{i,j,k,l}}\in {^H_H\mathcal{YD}}$ with its module and comodule structure given by
\begin{equation*}
x\cdot v=(-1)^iv,~~~y\cdot v=(-1)^jv,~~~t\cdot v=(-1)^kv,~~~\delta(v)=x^jy^l\otimes v.
\end{equation*}
\end{pro}
\begin{proof} The action is given by the restriction of the action given in Lemma \ref{1simple}. Since $\mathbbm{k}_{\chi_{i,j,k,l}}$ is one-dimensional, we must have that $\delta(v)=h\otimes v$ with $h\in G(H)=\langle x\rangle\times\langle y\rangle$. As $f\cdot v=\langle f,h\rangle v$ for all $f\in H^*$, it follows that $\delta(v)=x^jy^l\otimes v.$
\end{proof}

\begin{pro}\label{c}
The braiding of the
one-dimensional Yetter-Drinfeld module $\mathbbm{k}_{\chi_{i,j,k,l}}$ is $c(v\otimes v)=(-1)^{(i+l)j}v\otimes v.$
\end{pro}
\begin{proof} By formula $(\ref{braing})$ and Proposition $\ref{1dim}$, we have that
\begin{equation*}
c(v\otimes v)=x^jy^l\cdot v\otimes v=(-1)^{(i+l)j}v\otimes v.
\end{equation*}
This completes the proof.
\end{proof}

\begin{pro}
Let $V_{i,j,k,l,m,n}=\mathbbm{k}\{v_1,v_2\}$ be a two-dimensional simple $D$-module with $(i,j,k,l,m,n)\in \Omega$,
then $V_{i,j,k,l,m,n}\in {^H_H\mathcal{YD}}$ with its action given by

\hspace{3em}$x\cdot v_1=(-1)^i v_1,~~~~~~~~~~y\cdot v_1=(-1)^j v_1, ~~~~~t\cdot v_1=(-1)^k v_1,$

\hspace{3em}$x\cdot v_2=(-1)^{i+m+n} v_2,~~~y\cdot v_2=(-1)^j v_2, ~~~~~t\cdot v_2=(-1)^{j+l+k } v_2,$
\\and its coaction by
\\$1.~$for $l=m=n=0$,

$\delta(v_1)=\frac{1}{2}(1+x^2)\otimes v_1+\frac{1}{2}(1-x^2)\otimes v_2,~~~\delta(v_2)=\frac{1}{2}(1+x^2)\otimes v_2+\frac{1}{2}(1-x^2)\otimes v_1;$
\\$2.~$for $l=0, m=n=1$,

$\delta(v_1)=\frac{1}{2}(1+x^2)y\otimes v_1+\frac{1}{2}(1-x^2)y\otimes v_2,~\delta(v_2)=\frac{1}{2}(1+x^2)y\otimes v_2+\frac{1}{2}(1-x^2)y\otimes v_1;$
\\$3.~$for $l=m=0, n=1$,

$\delta(v_1)=\frac{1}{2}(1+x^2)t\otimes v_1+\frac{1}{2}(1-x^2)yt\otimes v_2,~\delta(v_2)=\frac{1}{2}(1+x^2)yt\otimes v_2+\frac{1}{2}(1-x^2)t\otimes v_1;$
\\$4.~$for $l=0, m=1, n=0$,

$\delta(v_1)=\frac{1}{2}(1+x^2)yt\otimes v_1+\frac{1}{2}(1-x^2)t\otimes v_2,~\delta(v_2)=\frac{1}{2}(1+x^2)t\otimes v_2+\frac{1}{2}(1-x^2)yt\otimes v_1;$
\\$5.~$for $l=1, m=n=0$,

$\delta(v_1)=\frac{1}{2}x(1+x^2)\otimes v_1+\frac{1}{2}\xi x(1-x^2)\otimes v_2,\delta(v_2)=\frac{1}{2}x(1+x^2)\otimes v_2-\frac{1}{2}\xi x(1-x^2)\otimes v_1;$
\\$6.~$for $l=m=n=1$,

$\delta(v_1)=\frac{1}{2}xy(1+x^2)\otimes v_1+\frac{1}{2}\xi xy(1-x^2)\otimes v_2,\delta(v_2)=\frac{1}{2}xy(1+x^2)\otimes v_2-\frac{1}{2}\xi xy(1-x^2)\otimes v_1;$
\\$7.~$for $l=1, m=0, n=1$,

$\delta(v_1)=\frac{1}{2}xt(1+x^2)\otimes v_1+\frac{1}{2}\xi xyt(1-x^2)\otimes v_2,\delta(v_2)=\frac{1}{2}xyt(1+x^2)\otimes v_2-\frac{1}{2}\xi xt(1-x^2)\otimes v_1;$
\\$8.~$for $l=m=1, n=0$,

$\delta(v_1)=\frac{1}{2}xyt(1+x^2)\otimes v_1+\frac{1}{2}\xi xt(1-x^2)\otimes v_2,\delta(v_2)=\frac{1}{2}xt(1+x^2)\otimes v_2-\frac{1}{2}\xi xyt(1-x^2)\otimes v_1.$
\end{pro}
\begin{proof} We just need to describe the coaction. Denote by $\{h_i\}_{1\leq i\leq 16}$ and $\{h^i\}_{1\leq i\leq 16}$ a basis of $H$ and its dual basis,
respectively. Then the comodule structure is given by
$\delta(v)=\sum\limits_{i=1}^{16} h_i\otimes h^i\cdot v$, it follows that
\begin{flalign*}
\delta(v_1)=&\sum\limits_{i=0}^1\sum\limits_{j=0}^1\sum\limits_{k=0}^1\sum\limits_{l=0}^1(a^ib^jc^kd^l)^*\otimes a^ib^jc^kd^l\cdot v_1\\
=&[1^*+(-1)^la^*+(-1)^mb^*+(-1)^nc^*+(-1)^{m+l}(ab)^*+
(-1)^{l+n}(ac)^*+(-1)^{m+n}(bc)^*\\
&+(-1)^{l+m+n}(abc)^*]\otimes v_1\\
&+[(-1)^ld^*+(ad)^*+(-1)^{m+l}(bd)^*+(-1)^{n+l}(cd)^*+(-1)^m(abd)^*+(-1)^n(acd)^*
\\
&+(-1)^{m+n+l}(bcd)^*+(-1)^{m+n}(abcd)^*]\otimes v_2,
\end{flalign*}
\vspace{-3em}
\begin{flalign*}
\delta(v_2)=&\sum\limits_{i=0}^1\sum\limits_{j=0}^1\sum\limits_{k=0}^1\sum\limits_{l=0}^1(a^ib^jc^kd^l)^*\otimes a^ib^jc^kd^l\cdot v_2\\
=&[1^*+(-1)^la^*+(-1)^nb^*+(-1)^mc^*+(-1)^{n+l}(ab)^*+
(-1)^{l+m}(ac)^*+(-1)^{m+n}(bc)^*\\
&+(-1)^{l+m+n}(abc)^*]\otimes v_2\\
&+[d^*+(-1)^l(ad)^*+(-1)^{n}(bd)^*+(-1)^{m}(cd)^*+(-1)^{n+l}(abd)^*+(-1)^{m+l}(acd)^*\\
&+(-1)^{m+n}(bcd)^*+(-1)^{m+n+l}(abcd)^*]\otimes v_1.
\end{flalign*}
By discussing the possible situations, the claim follows.
\end{proof}

Using the same method, we have the following propositions.
\begin{pro}
Let $W_{i,j,k,l}=\mathbbm{k}\{v_1,v_2\}$ be a two-dimensional simple $D$-module with $(i,j,k,l)\in \Lambda^1$, then $W_{i,j,k,l}\in {^H_H\mathcal{YD}}$
with its action given by

\hspace{4em}$x\cdot v_1=\xi v_1,~~~~~~~~~~y\cdot v_1=(-1)^j v_1, ~~~~~~~t\cdot v_1=v_2,$

\hspace{4em}$x\cdot v_2=\xi^{-1} v_2,~~~~~~~y\cdot v_2=(-1)^j v_2, ~~~~~~~t\cdot v_2=v_1,$
\\and its coaction by $\delta(v_1)=x^ky^l\otimes v_1$,
\\$1.~$for $(k,j,l)=(0,0,0), ~~\delta(v_2)=y\otimes v_2$; $(k,j,l)=(2,0,0), ~~\delta(v_2)=x^2y\otimes v_2$;
\\$2.~$for $(k,j,l)=(0,1,0), ~~\delta(v_2)=x^2y\otimes v_2$; $(k,j,l)=(2,1,0), ~~\delta(v_2)=y\otimes v_2$;
\\$3.~$for $(k,j,l)=(0,0,1), ~~\delta(v_2)=1\otimes v_2$; $(k,j,l)=(2,0,1), ~~\delta(v_2)=x^2\otimes v_2$;
\\$4.~$for $(k,j,l)=(0,1,1), ~~\delta(v_2)=x^2\otimes v_2$; $(k,j,l)=(2,1,1), ~~\delta(v_2)=1\otimes v_2$;
\\$5.~$for $(k,j,l)=(1,0,0), ~~\delta(v_2)=x^3y\otimes v_2$; $(k,j,l)=(3,0,0), ~~\delta(v_2)=xy\otimes v_2$;
\\$6.~$for $(k,j,l)=(1,1,0), ~~\delta(v_2)=xy\otimes v_2$; $(k,j,l)=(3,1,0), ~~\delta(v_2)=x^3y\otimes v_2$;
\\$7.~$for $(k,j,l)=(1,0,1), ~~\delta(v_2)=x^3\otimes v_2$; $(k,j,l)=(3,0,1), ~~\delta(v_2)=x\otimes v_2$;
\\$8.~$for $(k,j,l)=(1,1,1), ~~\delta(v_2)=x\otimes v_2$; $(k,j,l)=(3,1,1), ~~\delta(v_2)=x^3\otimes v_2$.

\end{pro}

\begin{pro}
Let $U_{i,j,k,l}=\mathbbm{k}\{v_1,v_2\}$ be a two-dimensional simple $D$-module with $(i,j,k,l)\in \Lambda_2$, then $U_{i,j,k,l}\in {^H_H\mathcal{YD}}$
with its action same as $W_{i,j,k,l}$
and its coaction by
\\$1.$ for $(j,l,k)=(0,0,0)$,

$\delta(v_1)=\frac{1}{2}(1+x^2)t\otimes v_1+\frac{1}{2}(1-x^2)yt\otimes v_2,\delta(v_2)=\frac{1}{2}(1+x^2)yt\otimes v_2+\frac{1}{2}(1-x^2)t\otimes v_1;$
\\for $(j,l,k)=(2,0,0)$,

$\delta(v_1)=\frac{1}{2}(1+x^2)t\otimes v_1-\frac{1}{2}(1-x^2)yt\otimes v_2,~~\delta(v_2)=\frac{1}{2}(1+x^2)yt\otimes v_2-\frac{1}{2}(1-x^2)t\otimes v_1;$
\\$2.~$for $(j,l,k)=(0,1,0)$,

$\delta(v_1)=\frac{1}{2}(1+x^2)yt\otimes v_1+\frac{1}{2}(1-x^2)t\otimes v_2,~~\delta(v_2)=\frac{1}{2}(1+x^2)t\otimes v_2+\frac{1}{2}(1-x^2)yt\otimes v_1;$
\\for $(j,l,k)=(2,1,0)$,

$\delta(v_1)=\frac{1}{2}(1+x^2)yt\otimes v_1-\frac{1}{2}(1-x^2)t\otimes v_2,\delta(v_2)=\frac{1}{2}(1+x^2)t\otimes v_2-\frac{1}{2}(1-x^2)yt\otimes v_1;$
\\$3.~$for $(j,l,k)=(0,0,1)$,

$\delta(v_1)=\frac{1}{2}(1+x^2)xt\otimes v_1+\frac{\xi}{2}(1-x^2)xyt\otimes v_2,\delta(v_2)=\frac{1}{2}(1+x^2)xyt\otimes v_2-\frac{\xi}{2}(1-x^2)xt\otimes v_1;$
\\for $(j,l,k)=(2,0,1)$,

$\delta(v_1)=\frac{1}{2}(1+x^2)xt\otimes v_1-\frac{\xi}{2}(1-x^2)xyt\otimes v_2,\delta(v_2)=\frac{1}{2}(1+x^2)xyt\otimes v_2+\frac{\xi}{2}(1-x^2)xt\otimes v_1;$
\\$4.~$for $(j,l,k)=(0,1,1)$,

$\delta(v_1)=\frac{1}{2}(1+x^2)xyt\otimes v_1+\frac{\xi}{2}(1-x^2)xt\otimes v_2,\delta(v_2)=\frac{1}{2}(1+x^2)xt\otimes v_2-\frac{\xi}{2}(1-x^2)xyt\otimes v_1;$
\\for $(j,l,k)=(2,1,1)$,

$\delta(v_1)=\frac{1}{2}(1+x^2)xyt\otimes v_1-\frac{\xi}{2}(1-x^2)xt\otimes v_2,\delta(v_2)=\frac{1}{2}(1+x^2)xt\otimes v_2+\frac{\xi}{2}(1-x^2)xyt\otimes v_1;$
\\$5.~$for $(j,l,k)=(1,0,0)$,

$\delta(v_1)=\frac{1}{2}(1+x^2)t\otimes v_1-\frac{\xi}{2}(1-x^2)yt\otimes v_2,\delta(v_2)=\frac{1}{2}(1+x^2)yt\otimes v_2+\frac{\xi}{2}(1-x^2)t\otimes v_1;$
\\for $(j,l,k)=(3,0,0)$,

$\delta(v_1)=\frac{1}{2}(1+x^2)t\otimes v_1+\frac{\xi}{2}(1-x^2)yt\otimes v_2,\delta(v_2)=\frac{1}{2}(1+x^2)yt\otimes v_2-\frac{\xi}{2}(1-x^2)t\otimes v_1;$
\\$6.~$for $(j,l,k)=(1,1,0)$,

$\delta(v_1)=\frac{1}{2}(1+x^2)yt\otimes v_1-\frac{\xi}{2}(1-x^2)t\otimes v_2,\delta(v_2)=\frac{1}{2}(1+x^2)t\otimes v_2+\frac{\xi}{2}(1-x^2)yt\otimes v_1;$
\\for $(j,l,k)=(3,1,0)$,

$\delta(v_1)=\frac{1}{2}(1+x^2)yt\otimes v_1+\frac{\xi}{2}(1-x^2)t\otimes v_2,\delta(v_2)=\frac{1}{2}(1+x^2)t\otimes v_2-\frac{\xi}{2}(1-x^2)yt\otimes v_1;$
\\$7.~$for $(j,l,k)=(1,0,1)$,

$\delta(v_1)=\frac{1}{2}(1+x^2)xt\otimes v_1+\frac{1}{2}(1-x^2)xyt\otimes v_2,\delta(v_2)=\frac{1}{2}(1+x^2)xyt\otimes v_2+\frac{1}{2}(1-x^2)xt\otimes v_1;$
\\for $(j,l,k)=(3,0,1)$,

$\delta(v_1)=\frac{1}{2}(1+x^2)xt\otimes v_1-\frac{1}{2}(1-x^2)xyt\otimes v_2,\delta(v_2)=\frac{1}{2}(1+x^2)xyt\otimes v_2-\frac{1}{2}(1-x^2)xt\otimes v_1;$
\\$8.~$for $(j,l,k)=(1,1,1)$,

$\delta(v_1)=\frac{1}{2}(1+x^2)xyt\otimes v_1+\frac{1}{2}(1-x^2)xt\otimes v_2,\delta(v_2)=\frac{1}{2}(1+x^2)xt\otimes v_2+\frac{1}{2}(1-x^2)xyt\otimes v_1;$
\\for $(j,l,k)=(3,1,1)$,

$\delta(v_1)=\frac{1}{2}(1+x^2)xyt\otimes v_1-\frac{1}{2}(1-x^2)xt\otimes v_2,\delta(v_2)=\frac{1}{2}(1+x^2)xt\otimes v_2-\frac{1}{2}(1-x^2)xyt\otimes v_1.$

\end{pro}

\section{Nichols algebras in $^H_H\mathcal{YD}$}
Let $V=\bigoplus_{i\in I}V_i$, where $V_i$ is a simple object in $_H^H\mathcal{YD}$. In this section we will try to determine all the
finite-dimensional Nichols algebras satisfying $\mathcal{B}(V)\cong \bigotimes_{i\in I}\mathcal{B}(V_i)$.
We begin by studying the Nichols algebras of simple Yetter-Drinfeld modules.
\begin{lem}\label{1B}
Let $(i,j,k,l)\in \mathbb{I}_{0,1}\times \mathbb{I}_{0,3}\times \mathbb{I}_{0,1}\times \mathbb{I}_{0,1}$. The Nichols algebras
$\mathcal{B}(\mathbbm{k}_{\chi_{i,j,k,l}})$ associated to $\mathbbm{k}_{\chi_{i,j,k,l}}=\mathbbm{k}v$ are

~~~~~~~~~~~~$\mathcal{B}(\mathbbm{k}_{\chi_{i,j,k,l}})=\left\{\begin{array}{ll}
\mathbbm{k}[v], ~~~~~~~~~~~~~~~~~~~~~~~~if ~(i+l)j=0,~2,\\
\mathbbm{k}[v]/(v^2)=\bigwedge \mathbbm{k}_{\chi_{i,j,k,l}}, ~~~if~(i+l)j=1,~3.
                                                       \end{array}
                                                     \right.$
\end{lem}
\begin{proof}
The claim follows by Proposition \ref{c}.
\end{proof}

\begin{rem}\label{k}
$\mathcal{B}(\mathbbm{k}_{\chi_{i,j,k,l}})$ is finite-dimensional if and only if $\mathbbm{k}_{\chi_{i,j,k,l}}$ is isomorphic either to
\begin{flalign*}
&V_1:=\mathbbm{k}_{\chi_{1,1,0,0}},\quad V_2:=\mathbbm{k}_{\chi_{1,1,1,0}}, \quad V_3:=\mathbbm{k}_{\chi_{0,1,0,1}},\quad V_4:=\mathbbm{k}_{\chi_{0,1,1,1}},\\
&V_5:=\mathbbm{k}_{\chi_{1,3,0,0}},
\quad V_6:=\mathbbm{k}_{\chi_{1,3,1,0}},\quad V_7:=\mathbbm{k}_{\chi_{0,3,0,1}}, \quad V_8:=\mathbbm{k}_{\chi_{0,3,1,1}}.
\end{flalign*}
\end{rem}

For the convenience of our statements, we let
\begin{flalign*}
&M_1=V_{0,1,0,0,1,1},\quad M_2=V_{1,1,0,0,1,1,},\quad M_3=V_{0,0,1,0,0,1}, \quad M_4=V_{0,1,1,0,0,1},\\
&M_5=V_{0,0,1,0,1,0},\quad M_6=V_{0,1,0,0,1,0},\quad M_7=V_{1,0,0,1,0,0}, \quad M_{8}=V_{1,0,0,1,1,1}, ~\\
&M_{9}=V_{0,1,1,1,0,1},\quad M_{10}=V_{0,0,1,1,0,1},\ \ M_{11}=V_{0,1,0,1,1,0},\ \ M_{12}=V_{0,0,1,1,1,0},\\
&M_{13}=W_{1,1,0,1}, \hspace{1em}M_{14}=W_{1,1,2,0}, \hspace{1em}M_{15}=W_{1,0,2,0},\hspace{1em}M_{16}=W_{1,0,2,1},\\
&M_{17}=U_{1,2,0,0},\hspace{1.3em}M_{18}=U_{1,2,0,1},
\hspace{1.5em}M_{19}=U_{1,0,1,0},\hspace{1em}M_{20}=U_{1,0,1,1}.
\end{flalign*}

\begin{lem}\label{V}
Let $V_{i,j,k,l,m,n}\in {^H_H\mathcal{YD}}$ with $(i,j,k,l,m,n)\in \Omega$. Then

$(1)$ $\dim \mathcal{B}(V_{i,j,k,l,m,n})=\infty$, for $V_{i,j,k,l,m,n}\notin \{M_1, M_2,\ldots M_{12}\}$.

$(2)$ $\forall ~V\in \{M_1, M_2,\ldots M_{12}\},$ $\dim \mathcal{B}(V)=4$ and relations of $\mathcal{B}(M_4)$, $\mathcal{B}(M_6)$,
$\mathcal{B}(M_{9})$, $\mathcal{B}(M_{11})$ are given by
$v_1^2=0, ~v_2^2=0, ~v_1v_2-v_2v_1=0$,
others are given by $v_1^2=0, ~v_2^2=0, ~v_1v_2+v_2v_1=0$.
\end{lem}
\begin{proof} If $V_{i,j,k,l,m,n}\notin \{M_1, M_2,\ldots M_{12}\}$, then the braiding has an eigenvector $v_1\otimes v_1$ of eigenvalue $1$, hence $(1)$ follows by Remark \ref{rem}. For $(2)$,
the braiding is of type $A_1\times A_1$, then we can get the relations of Nichols algebras.
\end{proof}

Using the same method, we have the following Lemmas.
\begin{lem}\label{W}
Let $W_{i,j,k,l}\in {^H_H\mathcal{YD}}$ with $(i,j,k,l)\in \Lambda^1$. Then

$(1)$ $\dim \mathcal{B}(W_{i,j,k,l})=\infty$, for $W_{i,j,k,l,m,n}\notin \{M_{13}, M_{14}, M_{15}, M_{16}\}$.

$(2)$ $\forall ~V\in \{M_{13}, M_{14}, M_{15}, M_{16}\},$ $\dim \mathcal{B}(V)=4$ and the relations are given by
$v_1^2=0, ~v_2^2=0,~ v_1v_2+v_2v_1=0$.
\end{lem}

\begin{lem}\label{U}
Let $U_{i,j,k,l}\in {^H_H\mathcal{YD}}$ with $(i,j,k,l)\in \Lambda^2$. Then

$(1)$ $\dim \mathcal{B}(U_{i,j,k,l})=\infty$, for $U_{i,j,k,l}\notin \{M_{17}, M_{18}, M_{19}, M_{20}\}$.

$(2)$ $\forall ~V\in \{M_{17}, M_{18}, M_{19}, M_{20}\},$ $\dim \mathcal{B}(V)=4$ and the relations of $\mathcal{B}(M_{17})$, $\mathcal{B}(M_{18})$ are given by
$v_1v_2=0,~ v_2v_1=0, ~v_1^2+v_2^2=0$, the relations of $\mathcal{B}(M_{19})$, $\mathcal{B}(M_{20})$ are given by
$v_1v_2=0, ~v_2v_1=0, ~v_1^2-v_2^2=0$.
\end{lem}
\begin{proof} $(1)$ is similar to Proposition \ref{V}~(1).

$(2)$~ For $V\in \{M_{17}, M_{18}, M_{19}, M_{20}\},$ it belongs to the case $\mathfrak{R}_{1,4}$ in \cite{AGi17}.
Then the assertion follows by {\rm\cite [Proposition 3.15]{AGi17}}.
\end{proof}

\begin{pro}\label{sum}

$(1)$ Let $V,~W$ be simple objects in $_H^H\mathcal{YD}$. Finite-dimensional $\mathcal{B}(V\oplus W)\cong \mathcal{B}(V)\otimes\mathcal{B}(W)$ for the following cases

\hspace{1.5em}$(a)$ $V,~W\in\{V_1,V_2,\ldots V_8\};$

\hspace{1.5em}$(b)$ $V\in \{V_3,V_4,V_7,V_8\}$, $W\in\{M_1, M_{8}\};$

\hspace{1.5em}$(c)$ $V\in \{V_1,V_2,V_5,V_6\}$, $W\in\{M_2, M_{7}\};$

\hspace{1.5em}$(d)$ $V\in \{V_1,V_3,V_5,V_7\}$, $W\in\{M_3, M_{11}\};$

\hspace{1.5em}$(e)$ $V\in \{V_1,V_4,V_5,V_8\}$, $W\in\{M_4, M_{12}\};$

\hspace{1.5em}$(f)$ $V\in \{V_2,V_4,V_6,V_8\}$, $W\in\{M_5, M_{9}\};$

\hspace{1.5em}$(g)$ $V\in \{V_2,V_3,V_6,V_7\}$, $W\in\{M_6, M_{10}\};$

\hspace{1.5em}$(h)$ $V=M_1$, $W\in\{M_2,M_7\};$ $V=M_3$ or $M_{11}$, $W\in\{M_5,M_9\};$

\hspace{1.5em}$(i)$ $V=M_4$ or $M_{12}$, $W\in\{M_6,M_{10}\};$

\hspace{1.5em}$(j)$ $(V,W)\in\{(M_2,M_8),(M_7,M_{8}),(M_{13},M_{14}),(M_{15},M_{16}),(M_{17},M_{18}),(M_{19},M_{20})\};$

\hspace{1.5em}$(k)$ $V=W=M_i$ for $i\in\mathbb{I}_{1,20}$.

$(2)$ For $W\in \{M_{13}, M_{14}, M_{15}, M_{16}\},~ V\in\{V_1,\ldots,V_8\}$, $\dim \mathcal{B}(V\oplus W)=\infty$.

$(3)$ For $W\in \{M_{17}, M_{18}, M_{19}, M_{20}\},~ V\in\{V_1,\ldots,V_8\}$, $\dim \mathcal{B}(V\oplus W)=\infty$.

\end{pro}
\begin{proof} The part $(1)$ is a direct result of Lemma \ref{direct sum}.

$(2)$ Let $M_3=\mathbbm{k}\{ v_1, v_2 \}$, $\mathbbm{k}p=\mathbbm{k}_{\chi_{i,j,k,l}}\in\{V_1,\ldots,V_8\}$, then
\begin{flalign*}
&c(p\otimes v_1)=(-1)^l\xi^j v_1\otimes p,~~~~~~c(p\otimes v_2)=-(-1)^l\xi^j v_2\otimes p,\\
&c(v_1\otimes p)=- p\otimes v_1,~~~~~~~~~~~~~c(v_2\otimes p)=p\otimes v_2.
\end{flalign*}
The generalized Dynkin diagram associated to the braiding is given by
\begin{center}
\begin{tikzpicture}
\draw node{} node at(0.5,0)[below]{$v_1$} (0.5,0)circle[radius=0.05] (0.55,0)--(3.55,0);
\draw node{} node at(0.5,0.01)[above]{$-1$};
\draw node{} node at(2.2,0.01)[above]{$(-1)^{l+1}\xi^j$};
\draw node{} node at(3.5,0.01)[above]{$-1$};
\draw node{} node at(5.2,0.01)[above]{$(-1)^{l+1}\xi^j$};
\draw node{} node at(6.7,0.01)[above]{$-1$};
\draw node at(3.6,0)[below]{$p$} (3.6,0)circle[radius=0.05] (3.65,0)--(6.65,0);
\draw node at(6.7,0)[below]{$v_2$} (6.7,0)circle[radius=0.05] (6.75,0)--(6.75,0);
\end{tikzpicture} \\
\end{center}
 Then by \cite{H09} we know that the Nichols algebra $\mathcal{B}(\mathbbm{k}_{\chi_{i,j,k,l}}\oplus M_{13})$ is infinite-dimensional.
 The proofs for $M_{14}, ~M_{15}, ~M_{16}$ are completely analogous.

$(3)$ Let $M_{17}=\mathbbm{k}\{ v_1, v_2 \}$, $\mathbbm{k}p= \mathbbm{k}_{\chi_{i,j,k,l}}\in\{V_1,\ldots,V_8\}$, then
\begin{equation*}
(id-c^2)(p\otimes v_1)=[1-(-1)^k\xi^j]p\otimes v_1,\quad(id-c^2)(p\otimes v_2)=[1-(-1)^k\xi^j]p\otimes v_2.
\end{equation*}
Thus $ad (\mathbbm{k} _{\chi_{i,j,k,l}})(M_{17})=(id-c^2)(\mathbbm{k}_{\chi_{i,j,k,l}}\otimes M_{17})\cong \mathbbm{k}_{\chi_{i,j,k,l}}\otimes M_{17}$.
Let $u=p\otimes v_1$, $v=p\otimes v_2$, then
\begin{flalign*}
&x\cdot u=(-1)^i\xi u,\hspace{2em}y\cdot u=-u,\hspace{2em}t\cdot u=-(-1)^kv,\\
&x\cdot v=-(-1)^i\xi v,\hspace{2em}y\cdot v=-v,\hspace{2em}t\cdot v=-(-1)^ku,\\
&\delta(u)=\frac{1}{2}x^jy^l(1+x^2)t\otimes u-\frac{1}{2}x^jy^l(1-x^2)yt\otimes v,\\
&\delta(v)=\frac{1}{2}x^jy^l(1+x^2)yt\otimes v-\frac{1}{2}x^jy^l(1-x^2)t\otimes u.
\end{flalign*}
Whence by Lemma \ref{U}, we have that $\dim \mathcal{B}(ad (\mathbbm{k} _{\chi_{i,j,k,l}})(M_{17}))=\infty$.
As $ad (\mathbbm{k} _{\chi_{i,j,k,l}})(M_{17})$ is the Yetter-Drinfeld submodule of
$\mathcal{B}(\mathbbm{k}_{\chi_{i,j,k,l}}\oplus M_{17})$ \cite{HS}, then $\dim \mathcal{B}(\mathbbm{k}_{\chi_{i,j,k,l}}\oplus
M_{17})=\infty$. The proofs for $M_{18},~M_{19},~ M_{20}$ are completely analogous.
\end{proof}
\\$\mathbf{Proof~ of ~Theorem ~A}$. The claim follows by Remark \ref{k}, Lemmas \ref{V}, ~\ref{W}, ~\ref{U} and Proposition \ref{sum}.

\section{Hopf algebras over $H$}

In this section, based on the principle of the lifting method, we will determine finite-dimensional Hopf algebras over $H$ such that their infinitesimal braidings are
those Yetter-Drinfeld modules listed in Theorem A.  We first show that the diagrams of these Hopf algebras are Nichols algebras.

\begin{thm}\label{gr}
Let $A$ be a finite-dimensional Hopf algebra over $H$ such that its infinitesimal braiding is isomorphic to a Yetter-Drinfeld module $V$ in Theorem A.
Then gr$(A)\cong \mathcal{B}(V)\sharp H$.
\end{thm}
\begin{proof} Let $S$ be the graded dual of the diagram $R$ of $A$. By {\rm\cite [Lemma 2.4]{AS2}}, $S$ is generated by $W=S(1)$.
Since $R(0)=\mathbbm{k},~ V:=R(1)=\mathcal{P}(R)$, there exists an  epimorphism $S\twoheadrightarrow \mathcal{B}(W)$. We know that $R$
is a Nichols algebra if and only if $\mathcal{P}(S)=S(1)$, that is, $S$ is also a Nichols algebra. It is enough to show that the relations
of $\mathcal{B}(W)$ also hold in $S$. If $V$ is a simple object in $_H^H\mathcal{YD}$, then $W$ must be simple as an object in $_H^H\mathcal{YD}$.

Assume $W=\Omega_2(n_1,n_2,n_3,n_4)$. Then $W$ is generated by
\begin{equation*}
p_1,\quad p_2,\quad \{C_k\}_{k=1,.,n_1},\quad \{D_\ell\}_{\ell=1,.,n_2}, \quad \{G_s\}_{s=1,..,n_3},\quad \{H_r\}_{r=1,..,n_4}
\end{equation*}
with $M_{1}=\mathbbm{k}\{p_1,p_2\}$, $\mathbbm{k}C_k\simeq V_3$, $\mathbbm{k}D_\ell\simeq V_4$, $\mathbbm{k}G_s\simeq V_7$, $\mathbbm{k}H_r\simeq V_8$,
and the defining ideal of the Nichols algebra $\mathfrak B(W)$ is generated by the elements
\begin{flalign*}
&p_1^2,\ \ p_2^2, \ \ p_1p_2+p_2p_1, \ \ \{C_{k_1}C_{k_2}+C_{k_2}C_{k_1}\}_{1\leq k_1,k_2\leq n_1},
\ \ \{D_{\ell_1}D_{\ell_2}+D_{\ell_2}D_{\ell_1}\}_{1\leq \ell_1,\ell_2\leq n_2},\\
&\{G_{s_1}G_{s_2}+G_{s_2}G_{s_1}\}_{1\leq s_1,s_2\leq n_3}, \ \ \{H_{r_1}H_{r_2}+H_{r_2}H_{r_1}\}_{1\leq r_1,r_2\leq n_4},
\ \ C_kD_\ell+D_\ell C_k,\\
&C_kG_s+G_sC_k,
\ \ C_kH_r+H_rC_k, \ \ D_\ell G_s+G_sD_\ell, \ \ D_\ell H_r+H_rD_\ell, \ \ G_sH_r+H_rG_s,\\
&p_1C_k+C_kp_1, \hspace{1em}p_2C_k+C_kp_2,
\hspace{1em}p_1D_\ell+D_\ell p_1, \hspace{1em}p_2D_\ell+D_\ell p_2, \hspace{1em}p_1G_s+G_sp_1,\\
&p_2G_s+G_sp_2, \hspace{1em}p_1H_r+H_rp_1, \hspace{1em}p_2H_r+H_rp_2.
\end{flalign*}
Since $S$ is finite-dimensional, it is enough to show that $c(r\otimes r)=r\otimes r$ for all generators given in above for the defining ideal.
By {\rm\cite [Theorem 6]{X}}, all those generators are primitive elements.
As $\delta(C_k)=xy\otimes C_k$, $\delta(D_l)=xy\otimes D_l$,
\begin{flalign*}
&\delta(p_1)=\frac{1}{2}(1+x^2)y\otimes p_1+\frac{1}{2}(1-x^2)y\otimes p_2,\\
&\delta(p_2)=\frac{1}{2}(1+x^2)y\otimes p_2+\frac{1}{2}(1-x^2)y\otimes p_1,
\end{flalign*}
we have that
\begin{flalign*}
&\delta(p_1^2)=\frac{1}{2}(1+x^2)\otimes p_1^2+\frac{1}{2}(1-x^2)\otimes p_2^2, \\
&\delta(p_2^2)=\frac{1}{2}(1+x^2)\otimes p_2^2+\frac{1}{2}(1-x^2)\otimes p_1^2,\\
&\delta(C_kD_l+D_lC_k)=x^2\otimes(C_kD_l+D_lC_k),\\
&\delta(p_1p_2+p_2p_1)=\frac{1}{2}(1+x^2)\otimes (p_1p_2+p_2p_1)+\frac{1}{2}(1-x^2)\otimes (p_1p_2+p_2p_1),\\
&\delta(p_1C_k+C_kp_1)=\frac{1}{2}x(1+x^2)\otimes (p_1C_k+C_kp_1)+\frac{1}{2}x(1-x^2)\otimes(p_2C_k+C_kp_2).
\end{flalign*}
Then by the definition of the braiding in $_H^H\mathcal{YD}$, the claim follows. We leave the rest to the reader.
\end{proof}

\begin{lem}{\rm\cite [Lemma 6.1]{AS0}}
Let $H$ be a Hopf algebra, $\psi: H\rightarrow H$ an automorphism of
Hopf algebra, and $V, ~W$ Yetter-Drinfeld modules over H.

$(1)$ Let $V^\psi$ be the same space as that of the underlying $V$ but with action and
coaction
\begin{flalign*}
&h\cdot_\psi v=\psi(h)\cdot v, ~~~\delta^\psi(v)=(\psi^{-1}\otimes id)\delta(v),~~~h\in H,~v\in V.
\end{flalign*}
Then $V^\psi$ is also a Yetter-Drinfeld module over H. If $T: V\rightarrow W$ is a morphism in $_H^H\mathcal{YD}$, so is $T^\psi: V^\psi\rightarrow W^\psi$. Moreover, the braiding
$c: V^\psi\otimes W^\psi\rightarrow W^\psi\otimes V^\psi$ coincides with the braiding $c: V\otimes W\rightarrow W\otimes V$.

$(2)$ If $R$ is an algebra (resp., a coalgebra, a Hopf algebra) in
$_H^H\mathcal{YD}$, so is $R^\psi$, with the same structural maps.

$(3)$ Let $R$ be a Hopf algebra in $_H^H\mathcal{YD}$. Then the map
$\varphi: R^\psi\sharp H\rightarrow R\sharp H$ given by $\varphi(r\sharp h)=r\sharp \psi(h)$ is an isomorphism of Hopf algebras.
\end{lem}

\begin{cor}
$(1)~ V_1^{\tau_{17}}\simeq V_3, ~V_2^{\tau_{17}}\simeq V_4, ~V_5^{\tau_{17}}\simeq V_7, ~V_6^{\tau_{17}}\simeq V_8, ~M_2^{\tau_{17}}\simeq M_1, ~M_7^{\tau_{17}}\simeq M_8.$

$(2)~ V_1^{\tau_{5}}\simeq V_2$, $V_3^{\tau_{5}}\simeq V_4$, $V_5^{\tau_{5}}\simeq V_6$, $V_7^{\tau_{5}}\simeq V_8$, $M_3^{\tau_{5}}\simeq M_5$, $M_{11}^{\tau_{5}}\simeq M_9$, $M_{4}^{\tau_{5}}\simeq M_6$, $M_{12}^{\tau_{5}}\simeq M_{10},$ $M_{17}^{\tau_{5}}\simeq M_{18}$, $M_{13}^{\tau_{9}}\simeq M_{14}, ~M_{15}^{\tau_{9}}\simeq M_{16}$.

$(3)$ $M_3^{\tau_{2}}\simeq M_{10}, ~M_{4}^{\tau_{2}}\simeq M_9, ~M_{5}^{\tau_{2}}\simeq M_{12}, ~M_{6}^{\tau_{2}}\simeq M_{11},$
$M_{17}^{\tau_{2}}\simeq M_{19}, ~M_{18}^{\tau_{2}}\simeq M_{20}$,
$V_1^{\tau_{2}}\simeq V_{2}$, $V_5^{\tau_{2}}\simeq V_{6}$, $V_i^{\tau_{2}}\simeq V_{i}$ for $i\in\{3,4,7,8\}$.

$(4)~\mathcal{B}(\Omega_i(n_1,n_2,n_3,n_4))\sharp H\simeq \mathcal{B}(\Omega_j(n_1,n_2,n_3,n_4))\sharp H$ for
\begin{equation*}
(i,j)\in\{(2,3), (4,6), (5,7), (8,9), (10,12), (11,13), (4,11), (5,10)\}.
\end{equation*}

$(5)$ $\mathcal{B}(\Omega_i)\sharp H\simeq \mathcal{B}(\Omega_j)\sharp H$ for
\begin{flalign*}
(i,j)\in\{&(14,30), (16,31), (19,32), (17,33), (19,34),
(20,35), (32,36),\\
&(22,37), (20,38),(21,39), (17,40), (18,41), (38,42), (40,43),\\
&(24,44),(26,45), (28,46), (28,47), (29,48), (46,49)\}.
\end{flalign*}

\end{cor}

Let $i,~j,~k,~\ell,~m,~q,~s,~r\in \mathbb{N}^*$, denote
\begin{flalign*}
&\mathbbm{k}A_i\simeq V_1, \hspace{1em}\mathbbm{k}B_j\simeq V_2, \hspace{1em}\mathbbm{k}C_k\simeq V_3, \hspace{1em}\mathbbm{k}D_\ell\simeq V_4,\\
&\mathbbm{k}E_m\simeq V_5, \hspace{1em}\mathbbm{k}F_q\simeq V_6, \hspace{1em}\mathbbm{k}G_s\simeq V_7, \hspace{1em}\mathbbm{k}H_r\simeq V_8.
\end{flalign*}

\begin{defi}\label{def 1}
For $n_1,n_2,..,n_8\in \mathbb{N}^*$ with $\sum_{i=1}^8 n_i\geq 1$ and
\begin{flalign*}
I_1=&\{(\alpha_{i_1,i_2})_{n_1\times n_1},~(\beta_{j_1,j_2})_{n_2\times n_2},~(\gamma_{k_1,k_2})_{n_3\times n_3},~(\eta_{\ell_1,\ell_2})_{n_4\times n_4},\\
&(\zeta_{m_1,m_2})_{n_5\times n_5},~(\theta_{q_1,q_2})_{n_6\times n_6},~(\lambda_{s_1,s_2})_{n_7\times n_7},~(\mu_{r_1,r_2})_{n_8\times n_8}\}
\end{flalign*}
with entries in $\mathbbm{k}$, let us denote by $\mathfrak{U}_1(n_1,n_2,\ldots,n_8;I_1)$ the algebra that is generated by $x,y,t,$
$\{A_i\}_{i=1,.,n_1}$, $\{B_j\}_{j=1,.,n_2}$, $\{C_k\}_{k=1,..,n_3}$, $\{D_\ell\}_{\ell=1,..,n_4}, \{E_m\}_{m=1,..,n_5}$, $\{F_q\}_{q=1,..,n_6},
\{G_s\}_{s=1,..,n_7}$, $\{H_r\}_{r=1,..,n_8}$ satisfying the relations $(\ref{3.1})$, $(\ref{3.2})$, $(\ref{3.3})$ and
\begin{flalign}
&xA_i=-A_ix,~~yA_i=-A_iy,~~tA_i=A_ix^2t,&\label{A}\\
&xB_j=-B_jx,~~yB_j=-B_jy,~~tB_j=-B_jx^2t,&\label{B}\\
&xC_k=C_kx,~~yC_k=-C_ky,~~tC_k=C_kx^2t,&\label{C}\\
&xD_\ell=D_\ell x,~~yD_\ell=-D_\ell y,~~tD_\ell=-D_\ell x^2t,&\label{D}\\
&xE_m=-E_mx,~~yE_m=-E_my,~~tE_m=E_mx^2t,&\label{E}\\
&xF_q=-F_qx,~~yF_q=-F_qy,~~tF_q=-F_qx^2t,&\label{F}\\
&xG_s=G_sx,~~yG_s=-G_sy,~~tG_s=G_sx^2t,&\label{G}\\
&xH_r=H_rx,~~yH_r=-H_ry,~~tH_r=-H_rx^2t,&\label{H}\\
&A_{i_1}A_{i_2}+A_{i_2}A_{i_1}=\alpha_{i_1,i_2}(1-x^2), ~~~i_1,i_2\in \{1,\ldots,n_1\},&\label{Ai}\\
&B_{j_1}B_{j_2}+B_{j_2}B_{j_1}=\beta_{j_1,j_2}(1-x^2), ~~~j_1,j_2\in \{1,\ldots,n_2\},&\label{Bj}\\
&C_{k_1}C_{k_2}+C_{k_2}C_{k_1}=\gamma_{k_1,k_2}(1-x^2), ~~~k_1,k_2\in \{1,\ldots,n_3\},&\label{Ck}\\
&D_{\ell_1}D_{\ell_2}+D_{\ell_2}D_{\ell_1}=\eta_{\ell_1,\ell_2}(1-x^2), ~~~\ell_1,\ell_2\in \{1,\ldots,n_4\},&\label{Dl}\\
&E_{m_1}E_{m_2}+E_{m_2}E_{m_1}=\zeta_{m_1,m_2}(1-x^2), ~~~m_1,m_2\in \{1,\ldots,n_5\},&\label{Em}\\
&F_{q_1}F_{q_2}+F_{q_2}F_{q_1}=\theta_{q_1,q_2}(1-x^2), ~~~q_1,q_2\in \{1,\ldots,n_6\},&\label{Fq}\\
&G_{s_1}G_{s_2}+G_{s_2}G_{s_1}=\lambda_{s_1,s_2}(1-x^2), ~~~s_1,s_2\in \{1,\ldots,n_7\},&\label{Gs}\\
&H_{r_1}H_{r_2}+H_{r_2}H_{r_1}=\mu_{r_1,r_2}(1-x^2), ~~~r_1,r_2\in \{1,\ldots,n_8\},&\label{Hr}\\
&A_iB_j+B_jA_i=0,~~A_iC_k-C_kA_i=0,~~A_iD_\ell-D_\ell A_i=0,~~A_iE_m+E_mA_i=0,&\\
&A_iF_q+F_qA_i=0,~~A_iG_s-G_sA_i=0,~~A_iH_r-H_rA_i=0,~~B_jC_k-C_kB_j=0,&\\
&B_jD_\ell-D_\ell B_j=0,~~B_jE_m+E_mB_j=0,~~B_jF_q+F_qB_j=0,~~B_jG_s-G_sB_j=0,&\\
&B_jH_r-H_rB_j=0,~~C_kD_\ell+D_\ell C_k=0,~~C_kE_m-E_mC_k=0,~~C_kF_q-F_qC_k=0,&\\
&C_kG_s+G_sC_m=0,~~C_kH_r+H_rC_k=0,~~D_\ell E_m-E_mD_\ell=0,~~D_\ell F_q-F_qD_\ell=0,&\\
&D_\ell G_s+G_sD_\ell=0,~~D_\ell H_r+H_rD_\ell=0,~~E_mF_q+F_qE_m=0,~~E_mG_s-G_sE_m=0,&\\
&E_mH_r-H_rE_m=0,~~F_qG_s-G_sF_q=0,~~F_qH_r-H_rF_q=0,~~G_sH_r+H_rG_s=0.&\label{e1}
\end{flalign}
It is a Hopf algebra with its coalgebra structure determined by

$~~~~~~~~~~~\triangle(A_i)=A_i\otimes 1+x\otimes A_i,~~~~~~~\triangle(B_j)=B_j\otimes 1+x\otimes B_j,$

$~~~~~~~~~~~\triangle(C_k)=C_k\otimes 1+xy\otimes C_k,~~~~~\triangle(D_\ell)=D_\ell\otimes 1+xy\otimes D_\ell,$

$~~~~~~~~~~~\triangle(E_m)=E_m\otimes 1+x^3\otimes E_m,~~~\triangle(F_q)=F_q\otimes 1+x^3\otimes F_q,$

$~~~~~~~~~~~\triangle(G_s)=G_s\otimes 1+x^3y\otimes G_s,~~~~\triangle(H_r)=H_r\otimes 1+x^3y\otimes H_r.$
\end{defi}

\begin{rem}
It is clear that $\mathfrak{U}_1(n_1,n_2,\ldots,n_8;0)\cong \mathcal{B}(\Omega_1(n_1,n_2,\ldots,n_8))\sharp H$. Also note that  $\mathfrak{U}_1(n_1,n_2,\ldots,n_8;I_1)\cong T(\Omega_1(n_1,n_2,\ldots,n_8))/J(I_1)$, where $J(I_1)$ is a Hopf ideal generated by relations $(\ref{Ai})-(\ref{Hr})$.
\end{rem}

Now we determine the liftings and discuss the isomorphism classes of $\mathfrak{U}_1(n_1,n_2,\ldots,n_8;I_1)$.

\begin{pro}\label{1}
$(1)$ Let $A$ be a finite-dimensional Hopf algebra over $H$ such that its infinitesimal braiding is $\Omega_1(n_1,n_2,\ldots,n_8)$, then $A\cong \mathfrak{U}_1(n_1,n_2,\ldots,n_8;I_1)$.

$(2) ~\mathfrak{U}_1(n_1,\ldots,n_8;I_1)\cong \mathfrak{U}_1(n_1,\ldots,n_8;I_1^{'})$ if and only if there exist invertible matrices $(a_{i_1^{'}i_1})_{n_1\times n_1}$, $(a_{i_2^{'}i_2})_{n_1\times n_1}$, ~$(b_{j_1^{'}j_1})_{n_2\times n_2}$, ~$(b_{j_2^{'}j_2})_{n_2\times n_2}$, ~$(c_{k_1^{'}k_1})_{n_3\times n_3}$, ~$(c_{k_2^{'}k_2})_{n_3\times n_3}$, ~$(d_{\ell_1^{'}\ell_1})_{n_4\times n_4}$, $(d_{\ell_2^{'}\ell_2})_{n_4\times n_4}$,
$(e_{m_1^{'}m_1})_{n_5\times n_5}$, $(e_{m_2^{'}m_2})_{n_5\times n_5}$,
$(f_{q_1^{'}q_1})_{n_6\times n_6}$, $(f_{q_2^{'}q_2})_{n_6\times n_6}$,
$(g_{s_1^{'}s_1})_{n_7\times n_7}$, $(g_{s_2^{'}s_2})_{n_7\times n_7}$,
$(h_{r_1^{'}r_1})_{n_8\times n_8}$, $(h_{r_2^{'}r_2})_{n_8\times n_8}$
such that
\begin{flalign}
&\sum\limits_{i_1^{'}=1}^{n_1}\sum\limits_{i_2^{'}=1}^{n_1}
a_{i_1^{'}i_1}a_{i_2^{'}i_2}\alpha_{i_1^{'},i_2^{'}}^{'}=\alpha_{i_1,i_2},
\hspace{5em}\sum\limits_{j_1^{'}=1}^{n_2}\sum\limits_{j_2^{'}=1}^{n_2}
b_{j_1^{'}j_1}b_{j_2^{'}j_2}\beta_{j_1^{'},j_2^{'}}^{'}=\beta_{j_1,j_2},\label{Aut1.1.1}\\
&\sum\limits_{k_1^{'}=1}^{n_3}\sum\limits_{k_2^{'}=1}^{n_3}
c_{k_1^{'}k_1}c_{k_2^{'}k_2}\gamma_{k_1^{'},k_2^{'}}^{'}=\gamma_{k_1,k_2},
\hspace{4em}\sum\limits_{\ell_1^{'}=1}^{n_4}\sum\limits_{\ell_2^{'}=1}^{n_4}
d_{\ell_1^{'}\ell_1}d_{\ell_2^{'}\ell_2}\eta_{\ell_1^{'},\ell_2^{'}}^{'}=\eta_{\ell_1,\ell_2},\label{Aut1.1.2}\\
&\sum\limits_{m_1^{'}=1}^{n_5}\sum\limits_{m_2^{'}=1}^{n_5}
e_{m_1^{'}m_1}e_{m_2^{'}m_2}\zeta_{m_1^{'},m_2^{'}}^{'}=\zeta_{m_1,m_2},
\hspace{1.2em}\sum\limits_{q_1^{'}=1}^{n_6}\sum\limits_{q_2^{'}=1}^{n_6}
f_{q_1^{'}q_1}f_{q_2^{'}q_2}\theta_{q_1^{'},q_2^{'}}^{'}=\theta_{q_1,q_2},\label{Aut1.1.3}\\
&\sum\limits_{s_1^{'}=1}^{n_7}\sum\limits_{s_2^{'}=1}^{n_7}
g_{s_1^{'}s_1}g_{s_2^{'}s_2}\lambda_{s_1^{'},s_2^{'}}^{'}=\lambda_{s_1,s_2},
\hspace{4em}\sum\limits_{r_1^{'}=1}^{n_8}\sum\limits_{r_2^{'}=1}^{n_8}
h_{r_1^{'}r_1}h_{r_2^{'}r_2}\mu_{r_1^{'},r_2^{'}}^{'}=\mu_{r_1,r_2},\label{Aut1.1.4}
\end{flalign}
or $n_1=n_2,~n_5=n_6$
satisfying $(\ref{Aut1.1.2})$, $(\ref{Aut1.1.4})$ and
\begin{align}
\begin{split}
&\sum\limits_{j_1^{'}=1}^{n_1}\sum\limits_{j_2^{'}=1}^{n_1}
a_{j_1^{'}i_1}a_{j_2^{'}i_2}\beta_{j_1^{'},j_2^{'}}^{'}=\alpha_{i_1,i_2},
\hspace{2.8em}\sum\limits_{i_1^{'}=1}^{n_1}\sum\limits_{i_2^{'}=1}^{n_1}
b_{i_1^{'}j_1}b_{i_2^{'}j_2}\alpha_{i_1^{'},i_2^{'}}^{'}=\beta_{j_1,j_2},\\
&\sum\limits_{q_1^{'}=1}^{n_5}\sum\limits_{q_2^{'}=1}^{n_5}
e_{q_1^{'}m_1}e_{q_2^{'}m_2}\theta_{q_1^{'},q_2^{'}}^{'}=\zeta_{m_1,m_2},
\hspace{1em}\sum\limits_{m_1^{'}=1}^{n_5}\sum\limits_{m_2^{'}=1}^{n_5}
f_{m_1^{'}q_1}f_{m_2^{'}q_2}\zeta_{m_1^{'},m_2^{'}}^{'}=\theta_{q_1,q_2},\label{Aut1.2}
\end{split}
\end{align}
or $n_1=n_2, ~n_3=n_4,~ n_5=n_6, ~n_7=n_8$
satisfying $(\ref{Aut1.2})$ and
\begin{align}
\begin{split}
&\sum\limits_{\ell_1^{'}=1}^{n_3}\sum\limits_{\ell_2^{'}=1}^{n_3}
c_{\ell_1^{'}k_1}c_{\ell_2^{'}k_2}\eta_{\ell_1^{'},\ell_2^{'}}^{'}=\gamma_{k_1,k_2},
\hspace{2.5em}\sum\limits_{k_1^{'}=1}^{n_3}\sum\limits_{k_2^{'}=1}^{n_3}
d_{k_1^{'}\ell_1}d_{k_2^{'}\ell_2}\gamma_{k_1^{'},k_2^{'}}^{'}=\eta_{\ell_1,\ell_2},\\
&\sum\limits_{r_1^{'}=1}^{n_7}\sum\limits_{r_2^{'}=1}^{n_7}
g_{r_1^{'}s_1}g_{r_2^{'}s_2}\mu_{r_1^{'},r_2^{'}}^{'}=\lambda_{s_1,s_2},
\hspace{2em}\sum\limits_{s_1^{'}=1}^{n_7}\sum\limits_{s_2^{'}=1}^{n_7}
h_{s_1^{'}r_1}h_{s_2^{'}r_2}\lambda_{s_1^{'},s_2^{'}}^{'}=\mu_{r_1,r_2},\label{Aut1.3}
\end{split}
\end{align}
\\or $n_3=n_4, ~ n_7=n_8$
satisfying $(\ref{Aut1.1.1})$, $(\ref{Aut1.1.3})$ and $(\ref{Aut1.3})$
or $n_1=n_5,~ n_2=n_6,~n_3=n_7, ~n_4=n_8$ and there exist invertible matrices
$(a_{m_1^{'}i_1})_{n_1\times n_1}$, $(a_{m_2^{'}i_2})_{n_1\times n_1}$, $(b_{q_1^{'}j_1})_{n_2\times n_2}$, $(b_{q_2^{'}j_2})_{n_2\times n_2}$, $(c_{s_1^{'}k_1})_{n_3\times n_3}$, $(c_{s_2^{'}k_2})_{n_3\times n_3}$,
$(d_{r_1^{'}\ell_1})_{n_4\times n_4}$, $(d_{r_2^{'}\ell_2})_{n_4\times n_4}$,
$(e_{i_1^{'}m_1})_{n_1\times n_1}$, $(e_{i_2^{'}m_2})_{n_1\times n_1}$,
$(f_{j_1^{'}q_1})_{n_2\times n_2}$, $(f_{j_2^{'}q_2})_{n_2\times n_2}$,
$(g_{k_1^{'}s_1})_{n_3\times n_3}$, $(g_{k_2^{'}s_2})_{n_3\times n_3}$,
$(h_{\ell_1^{'}r_1})_{n_4\times n_4}$, $(h_{\ell_2^{'}r_2})_{n_4\times n_4}$
such that
\begin{flalign}
&\sum\limits_{m_1^{'}=1}^{n_1}\sum\limits_{m_2^{'}=1}^{n_1}
a_{m_1^{'}i_1}a_{m_2^{'}i_2}\zeta_{m_1^{'},m_2^{'}}^{'}=\alpha_{i_1,i_2},
\hspace{1.5em}\sum\limits_{q_1^{'}=1}^{n_2}\sum\limits_{q_2^{'}=1}^{n_2}
b_{q_1^{'}j_1}b_{q_2^{'}j_2}\theta_{q_1^{'},q_2^{'}}^{'}=\beta_{j_1,j_2},\label{Aut1.5.1}\\
&\sum\limits_{s_1^{'}=1}^{n_3}\sum\limits_{s_2^{'}=1}^{n_3}
c_{s_1^{'}k_1}c_{s_2^{'}k_2}\lambda_{s_1^{'},s_2^{'}}^{'}=\gamma_{k_1,k_2},
\hspace{3em}\sum\limits_{r_1^{'}=1}^{n_4}\sum\limits_{r_2^{'}=1}^{n_4}
d_{r_1^{'}\ell_1}d_{r_2^{'}\ell_2}\mu_{r_1^{'},r_2^{'}}^{'}=\eta_{\ell_1,\ell_2},\label{Aut1.5.2}\\
&\sum\limits_{i_1^{'}=1}^{n_1}\sum\limits_{i_2^{'}=1}^{n_1}
e_{i_1^{'}m_1}e_{i_2^{'}m_2}\alpha_{i_1^{'},i_2^{'}}^{'}=\zeta_{m_1,m_2},
\hspace{2.5em}\sum\limits_{j_1^{'}=1}^{n_2}\sum\limits_{j_2^{'}=1}^{n_2}
f_{j_1^{'}q_1}f_{j_2^{'}q_2}\beta_{j_1^{'},j_2^{'}}^{'}=\theta_{q_1,q_2},\label{Aut1.5.3}\\
&\sum\limits_{k_1^{'}=1}^{n_3}\sum\limits_{k_2^{'}=1}^{n_3}
g_{k_1^{'}s_1}g_{k_2^{'}s_2}\gamma_{k_1^{'},k_2^{'}}^{'}=\lambda_{s_1,s_2},
\hspace{3em}\sum\limits_{\ell_1^{'}=1}^{n_4}\sum\limits_{\ell_2^{'}=1}^{n_4}
h_{\ell_1^{'}r_1}h_{\ell_2^{'}r_2}\eta_{\ell_1^{'},\ell_2^{'}}^{'}=\mu_{r_1,r_2},\label{Aut1.5.4}
\end{flalign}
or $n_1=n_6, ~n_2=n_5,~n_3=n_7, ~n_4=n_8$
satisfying $(\ref{Aut1.5.2})$, $(\ref{Aut1.5.4})$ and
\begin{align}
\begin{split}
&\sum\limits_{q_1^{'}=1}^{n_1}\sum\limits_{q_2^{'}=1}^{n_1}
a_{q_1^{'}i_1}a_{q_2^{'}i_2}\theta_{q_1^{'},q_2^{'}}^{'}=\alpha_{i_1,i_2},
\hspace{3em}\sum\limits_{m_1^{'}=1}^{n_2}\sum\limits_{m_2^{'}=1}^{n_2}
b_{m_1^{'}j_1}b_{m_2^{'}j_2}\zeta_{m_1^{'},m_2^{'}}^{'}=\beta_{j_1,j_2},\\
&\sum\limits_{j_1^{'}=1}^{n_2}\sum\limits_{j_2^{'}=1}^{n_2}
e_{j_1^{'}m_1}e_{j_2^{'}m_2}\beta_{j_1^{'},j_2^{'}}^{'}=\zeta_{m_1,m_2},
\hspace{2em}\sum\limits_{i_1^{'}=1}^{n_1}\sum\limits_{i_2^{'}=1}^{n_1}
f_{i_1^{'}q_1}f_{i_2^{'}q_2}\alpha_{i_1^{'},i_2^{'}}^{'}=\theta_{q_1,q_2},\label{Aut1.6}
\end{split}
\end{align}
or $n_1=n_6,~ n_2=n_5,~n_3=n_8, ~n_4=n_7$
satisfying $(\ref{Aut1.6})$ and
\begin{align}
\begin{split}
&\sum\limits_{r_1^{'}=1}^{n_3}\sum\limits_{r_2^{'}=1}^{n_3}
c_{r_1^{'}k_1}c_{r_2^{'}k_2}\mu_{r_1^{'},r_2^{'}}^{'}=\gamma_{k_1,k_2},
\hspace{2em}\sum\limits_{s_1^{'}=1}^{n_4}\sum\limits_{s_2^{'}=1}^{n_4}
d_{s_1^{'}\ell_1}d_{s_2^{'}\ell_2}\lambda_{s_1^{'},s_2^{'}}^{'}=\eta_{\ell_1,\ell_2},\\
&\sum\limits_{\ell_1^{'}=1}^{n_4}\sum\limits_{\ell_2^{'}=1}^{n_4}
g_{\ell_1^{'}s_1}g_{\ell_2^{'}s_2}\eta_{\ell_1^{'},\ell_2^{'}}^{'}=\lambda_{s_1,s_2},
\hspace{2.5em}\sum\limits_{k_1^{'}=1}^{n_3}\sum\limits_{k_2^{'}=1}^{n_3}
h_{k_1^{'}r_1}h_{k_2^{'}r_2}\gamma_{k_1^{'},k_2^{'}}^{'}=\mu_{r_1,r_2},\label{Aut1.7}
\end{split}
\end{align}
\\or $n_1=n_5, ~n_2=n_6,~n_3=n_8,~ n_4=n_7$
satisfying $(\ref{Aut1.5.1})$, $(\ref{Aut1.5.3})$ and $(\ref{Aut1.7})$
or $n_1=n_3,~ n_2=n_4,~n_5=n_7, ~n_6=n_8$ and there exist invertible matrices
$(a_{k_1^{'}i_1})_{n_1\times n_1}$, $(a_{k_2^{'}i_2})_{n_1\times n_1}$,
$(b_{\ell_1^{'}j_1})_{n_2\times n_2}$, $(b_{\ell_2^{'}j_2})_{n_2\times n_2}$,
 $(c_{i_1^{'}k_1})_{n_1\times n_1}$, $(c_{i_2^{'}k_2})_{n_1\times n_1}$,
$(d_{j_1^{'}\ell_1})_{n_2\times n_2}$, $(d_{j_2^{'}\ell_2})_{n_2\times n_2}$,
$(e_{s_1^{'}m_1})_{n_5\times n_5}$, $(e_{s_2^{'}m_2})_{n_5\times n_5}$,
$(f_{r_1^{'}q_1})_{n_6\times n_6}$,$(f_{r_2^{'}q_2})_{n_6\times n_6}$,
$(g_{m_1^{'}s_1})_{n_5\times n_5}$,$(g_{m_2^{'}s_2})_{n_5\times n_5}$,
$(h_{q_1^{'}r_1})_{n_6\times n_6}$, $(h_{q_2^{'}r_2})_{n_6\times n_6}$
such that
\begin{flalign}
&\sum\limits_{k_1^{'}=1}^{n_1}\sum\limits_{k_2^{'}=1}^{n_1}
a_{k_1^{'}i_1}a_{k_2^{'}i_2}\gamma_{k_1^{'},k_2^{'}}^{'}=\alpha_{i_1,i_2},
\hspace{2em}\sum\limits_{\ell_1^{'}=1}^{n_2}\sum\limits_{\ell_2^{'}=1}^{n_2}
b_{\ell_1^{'}j_1}b_{\ell_2^{'}j_2}\eta_{\ell_1^{'},\ell_2^{'}}^{'}=\beta_{j_1,j_2},\label{Aut1.9.1}\\
&\sum\limits_{i_1^{'}=1}^{n_1}\sum\limits_{i_2^{'}=1}^{n_1}
c_{i_1^{'}k_1}c_{i_2^{'}k_2}\alpha_{i_1^{'},i_2^{'}}^{'}=\gamma_{k_1,k_2},
\hspace{2.5em}\sum\limits_{j_1^{'}=1}^{n_2}\sum\limits_{j_2^{'}=1}^{n_2}
d_{j_1^{'}\ell_1}d_{j_2^{'}\ell_2}\beta_{j_1^{'},j_2^{'}}^{'}=\eta_{\ell_1,\ell_2},\label{Aut1.9.2}\\
&\sum\limits_{s_1^{'}=1}^{n_5}\sum\limits_{s_2^{'}=1}^{n_5}
e_{s_1^{'}m_1}e_{s_2^{'}m_2}\lambda_{s_1^{'},s_2^{'}}^{'}=\zeta_{m_1,m_2},
\hspace{1em}\sum\limits_{r_1^{'}=1}^{n_6}\sum\limits_{r_2^{'}=1}^{n_6}
f_{r_1^{'}q_1}f_{r_2^{'}q_2}\mu_{r_1^{'},r_2^{'}}^{'}=\theta_{q_1,q_2},\label{Aut1.9.3}\\
&\sum\limits_{m_1^{'}=1}^{n_5}\sum\limits_{m_2^{'}=1}^{n_5}
g_{m_1^{'}s_1}g_{m_2^{'}s_2}\zeta_{m_1^{'},m_2^{'}}^{'}=\lambda_{s_1,s_2},
\hspace{.5em}\sum\limits_{q_1^{'}=1}^{n_6}\sum\limits_{q_2^{'}=1}^{n_6}
h_{q_1^{'}r_1}h_{q_2^{'}r_2}\theta_{q_1^{'},q_2^{'}}^{'}=\mu_{r_1,r_2},\label{Aut1.9.4}
\end{flalign}
or $n_1=n_2=n_3=n_4,~n_5=n_6=n_7=n_8$
satisfying $(\ref{Aut1.9.1})$, $(\ref{Aut1.9.3})$ and
\begin{align}
\begin{split}
&\sum\limits_{j_1^{'}=1}^{n_1}\sum\limits_{j_2^{'}=1}^{n_1}
c_{j_1^{'}k_1}c_{j_2^{'}k_2}\beta_{j_1^{'},j_2^{'}}^{'}=\gamma_{k_1,k_2},
\hspace{2em}\sum\limits_{i_1^{'}=1}^{n_1}\sum\limits_{i_2^{'}=1}^{n_1}
d_{i_1^{'}\ell_1}d_{i_2^{'}\ell_2}\alpha_{i_1^{'},i_2^{'}}^{'}=\eta_{\ell_1,\ell_2},\\
&\sum\limits_{q_1^{'}=1}^{n_5}\sum\limits_{q_2^{'}=1}^{n_5}
g_{q_1^{'}s_1}g_{q_2^{'}s_2}\theta_{q_1^{'},q_2^{'}}^{'}=\lambda_{s_1,s_2},
\hspace{2em}\sum\limits_{m_1^{'}=1}^{n_5}\sum\limits_{m_2^{'}=1}^{n_5}
h_{m_1^{'}r_1}h_{m_2^{'}r_2}\zeta_{m_1^{'},m_2^{'}}^{'}=\mu_{r_1,r_2},\label{Aut1.10}
\end{split}
\end{align}
or $n_1=n_4,n_2=n_3,~n_5=n_8,n_6=n_7$
satisfying $(\ref{Aut1.10})$ and
\begin{align}
\begin{split}
&\sum\limits_{\ell_1^{'}=1}^{n_1}\sum\limits_{\ell_2^{'}=1}^{n_1}
a_{\ell_1^{'}i_1}a_{\ell_2^{'}i_2}\eta_{\ell_1^{'},\ell_2^{'}}^{'}=\alpha_{i_1,i_2},
\hspace{2.5em}\sum\limits_{k_1^{'}=1}^{n_2}\sum\limits_{k_2^{'}=1}^{n_2}
b_{k_1^{'}j_1}b_{k_2^{'}j_2}\gamma_{k_1^{'},k_2^{'}}^{'}=\beta_{j_1,j_2},\\
&\sum\limits_{r_1^{'}=1}^{n_5}\sum\limits_{r_2^{'}=1}^{n_5}
e_{r_1^{'}m_1}e_{r_2^{'}m_2}\mu_{r_1^{'},r_2^{'}}^{'}=\zeta_{m_1,m_2},
\hspace{1em}\sum\limits_{s_1^{'}=1}^{n_6}\sum\limits_{s_2^{'}=1}^{n_6}
f_{s_1^{'}q_1}f_{s_2^{'}q_2}\lambda_{s_1^{'},s_2^{'}}^{'}=\theta_{q_1,q_2},\label{Aut1.11}
\end{split}
\end{align}
or $n_1=n_2=n_3=n_4,~n_5=n_6=n_7=n_8$
satisfying $(\ref{Aut1.9.2})$, $(\ref{Aut1.9.4})$ and $(\ref{Aut1.11})$
or $n_1=n_7,~n_2=n_8,~n_3=n_5,~n_4=n_6$ and there exist invertible matrices
$(a_{s_1^{'}i_1})_{n_1\times n_1}$, $(a_{s_2^{'}i_2})_{n_1\times n_1}$, $(b_{r_1^{'}j_1})_{n_2\times n_2}$, $(b_{r_2^{'}j_2})_{n_2\times n_2}$, $(c_{m_1^{'}k_1})_{n_3\times n_3}$, $(c_{m_2^{'}k_2})_{n_3\times n_3}$,
$(d_{q_1^{'}\ell_1})_{n_4\times n_4}$, $(d_{q_2^{'}\ell_2})_{n_4\times n_4}$,
$(e_{k_1^{'}m_1})_{n_3\times n_3}$, $(e_{k_2^{'}m_2})_{n_3\times n_3}$,
$(f_{\ell_1^{'}q_1})_{n_4\times n_4}$, $(f_{\ell_2^{'}q_2})_{n_4\times n_4}$,
$(g_{i_1^{'}s_1})_{n_1\times n_1}$, $(g_{i_2^{'}s_2})_{n_1\times n_1}$,
$(h_{j_1^{'}r_1})_{n_2\times n_2}$, $(h_{j_2^{'}r_2})_{n_2\times n_2}$
such that
\begin{flalign}
&\sum\limits_{s_1^{'}=1}^{n_1}\sum\limits_{s_2^{'}=1}^{n_1}
a_{s_1^{'}i_1}a_{s_2^{'}i_2}\lambda_{s_1^{'},s_2^{'}}^{'}=\alpha_{i_1,i_2},
\hspace{2.5em}\sum\limits_{r_1^{'}=1}^{n_2}\sum\limits_{r_2^{'}=1}^{n_2}
b_{r_1^{'}j_1}b_{r_2^{'}j_2}\mu_{r_1^{'},r_2^{'}}^{'}=\beta_{j_1,j_2},\label{Aut1.13.1}\\
&\sum\limits_{m_1^{'}=2}^{n_3}\sum\limits_{m_2^{'}=1}^{n_3}
c_{m_1^{'}k_1}c_{m_2^{'}k_2}\zeta_{m_1^{'},m_2^{'}}^{'}=\gamma_{k_1,k_2},
\hspace{.5em}\sum\limits_{q_1^{'}=1}^{n_4}\sum\limits_{q_2^{'}=1}^{n_4}
d_{q_1^{'}\ell_1}d_{q_2^{'}\ell_2}\theta_{q_1^{'},q_2^{'}}^{'}=\eta_{\ell_1,\ell_2},\label{Aut1.13.2}\\
&\sum\limits_{k_1^{'}=1}^{n_3}\sum\limits_{k_2^{'}=1}^{n_3}
e_{k_1^{'}m_1}e_{k_2^{'}m_2}\gamma_{k_1^{'},k_2^{'}}^{'}=\zeta_{m_1,m_2},
\hspace{1em}\sum\limits_{\ell_1^{'}=1}^{n_4}\sum\limits_{\ell_2^{'}=1}^{n_4}
f_{\ell_1^{'}q_1}f_{\ell_2^{'}q_2}\eta_{\ell_1^{'},\ell_2^{'}}^{'}=\theta_{q_1,q_2},\label{Aut1.13.3}\\
&\sum\limits_{i_1^{'}=1}^{n_1}\sum\limits_{i_2^{'}=1}^{n_1}
g_{i_1^{'}s_1}g_{i_2^{'}s_2}\alpha_{i_1^{'},i_2^{'}}^{'}=\lambda_{s_1,s_2},
\hspace{3em}\sum\limits_{j_1^{'}=1}^{n_2}\sum\limits_{j_2^{'}=1}^{n_2}
h_{j_1^{'}r_1}h_{j_2^{'}r_2}\beta_{j_1^{'},j_2^{'}}^{'}=\mu_{r_1,r_2},\label{Aut1.13.4}
\end{flalign}
or $n_1=n_2=n_7=n_8,~n_3=n_4=n_5=n_6$
satisfying $(\ref{Aut1.13.1})$, $(\ref{Aut1.13.3})$ and
\begin{align}
\begin{split}
&\sum\limits_{q_1^{'}=1}^{n_3}\sum\limits_{q_2^{'}=1}^{n_3}
c_{q_1^{'}k_1}c_{q_2^{'}k_2}\theta_{q_1^{'},q_2^{'}}^{'}=\gamma_{k_1,k_2},
\hspace{2em}\sum\limits_{m_1^{'}=1}^{n_3}\sum\limits_{m_2^{'}=1}^{n_3}
d_{m_1^{'}\ell_1}d_{m_2^{'}\ell_2}\zeta_{m_1^{'},m_2^{'}}^{'}=\eta_{\ell_1,\ell_2},\\
&\sum\limits_{j_1^{'}=1}^{n_1}\sum\limits_{j_2^{'}=1}^{n_1}
g_{j_1^{'}s_1}g_{j_2^{'}s_2}\beta_{j_1^{'},j_2^{'}}^{'}=\lambda_{s_1,s_2},
\hspace{2.5em}\sum\limits_{i_1^{'}=1}^{n_1}\sum\limits_{i_2^{'}=1}^{n_1}
h_{i_1^{'}r_1}h_{i_2^{'}r_2}\alpha_{i_1^{'},i_2^{'}}^{'}=\mu_{r_1,r_2},\label{Aut1.14}
\end{split}
\end{align}
or $n_1=n_8,~n_2=n_7,~n_3=n_6,~n_4=n_5$
satisfying $(\ref{Aut1.14})$ and
\begin{align}
\begin{split}
&\sum\limits_{r_1^{'}=1}^{n_1}\sum\limits_{r_2^{'}=1}^{n_1}
a_{r_1^{'}i_1}a_{r_2^{'}i_2}\mu_{r_1^{'},r_2^{'}}^{'}=\alpha_{i_1,i_2},
\hspace{2em}\sum\limits_{s_1^{'}=1}^{n_2}\sum\limits_{s_2^{'}=1}^{n_2}
b_{s_1^{'}j_1}b_{s_2^{'}j_2}\lambda_{s_1^{'},s_2^{'}}^{'}=\beta_{j_1,j_2},\\
&\sum\limits_{\ell_1^{'}=1}^{n_4}\sum\limits_{\ell_2^{'}=1}^{n_4}
e_{\ell_1^{'}m_1}e_{\ell_2^{'}m_2}\eta_{\ell_1^{'},\ell_2^{'}}^{'}=\zeta_{m_1,m_2},
\hspace{1.5em}\sum\limits_{k_1^{'}=1}^{n_3}\sum\limits_{k_2^{'}=1}^{n_3}
f_{k_1^{'}q_1}f_{k_2^{'}q_2}\gamma_{k_1^{'},k_2^{'}}^{'}=\theta_{q_1,q_2},\label{Aut1.15}
\end{split}
\end{align}
or $n_1=n_2=n_7=n_8,~n_3=n_4=n_5=n_6$
satisfying $(\ref{Aut1.13.2})$, $(\ref{Aut1.13.4})$ and $(\ref{Aut1.15})$.
\end{pro}
\begin{proof} $(1)$ By Theorem $\ref{gr}$, we have that
gr$(A)\cong \mathcal{B}(\Omega_1(n_1,n_2,\ldots,n_8))\sharp H$. We
can check the relations listed in Definition \ref{def 1}, except
$(\ref{Ai})-(\ref{e1})$, hold in $A$ from the bosonization
$\mathcal{B}(\Omega_1(n_1,n_2,\ldots,n_8))\sharp H$. As
$\triangle(A_i)=A_i\otimes 1+x\otimes A_i,
\triangle(B_j)=B_j\otimes 1+x\otimes B_j,$ a direct computation
shows that
\begin{flalign*}
&\triangle(A_{i_1}A_{i_2}+A_{i_2}A_{i_1})=(A_{i_1}A_{i_2}+A_{i_2}A_{i_1})\otimes 1+x^2\otimes (A_{i_1}A_{i_2}+A_{i_2}A_{i_1}),\\
&\triangle(A_iB_j+B_jA_i)=(A_iB_j+B_jA_i)\otimes 1+x^2\otimes (A_iB_j+B_jA_i).
\end{flalign*}
Then $A_{i_1}A_{i_2}+A_{i_2}A_{i_1}, A_iB_j+B_jA_i\in \mathcal{P}_{1,x^2}(\mathcal{B}(\Omega_1(n_1,n_2,\ldots,n_8))\sharp H)
=\mathcal{P}_{1,x^2}(H)=\mathbbm{k}\{1-x^2\}$. Since $t(A_iB_j+B_jA_i)=-(A_iB_j+B_jA_i)t$, it follows that
\begin{equation*}
A_{i_1}A_{i_2}+A_{i_2}A_{i_1}=\alpha_{i_1,i_2}(1-x^2),\hspace{1em}A_iB_j+B_jA_i=0
\end{equation*}
for some $\alpha_{i_1,i_2}\in\mathbbm{k}$. Similarly, the relations
$(\ref{Bj})-(\ref{e1})$ hold in $A$. Then there is a surjective Hopf
morphism from $\mathfrak{U}_1(n_1,n_2,\ldots,n_8;I_1)$ to $A$. We
can observe that each element of
$\mathfrak{U}_1(n_1,n_2,\ldots,n_8;I_1)$ can be expressed by a
linear combination of
\begin{equation*}
\{A_i^{\alpha_i}B_j^{\beta_j}C_k^{\gamma_k}
D_\ell^{\eta_\ell}E_m^{\zeta_m}F_q^{\theta_q}G_s^{\lambda_s}H_r^{\mu_r}x^ey^ft^g;
~\alpha_i,\beta_j,\gamma_k,\eta_\ell,\zeta_m,\theta_q,\lambda_s,\mu_r,f,g\in\mathbb{I}_{0,1},
e\in\mathbb{I}_{0,3}\}.
\end{equation*}
In fact, according to the Diamond Lemma \cite{G}, the set is a basis of
$\mathfrak{U}_1(n_1,n_2,\ldots,n_8;I_1)$.
Then $\dim A$ = $\dim \mathfrak{U}_1(n_1,n_2,\ldots,n_8;I_1)$ and
whence $A\cong \mathfrak{U}_1(n_1,n_2,\ldots,n_8;I_1)$.

$(2)$ Suppose that $\Phi : \mathfrak{U}_1(n_1,\ldots,n_8;I_1)\rightarrow \mathfrak{U}_1(n_1,\ldots,n_8;I_1^{'})$ is an isomorphism of
Hopf algebras, where
\begin{flalign*}
I_1^{'}=\{&(\alpha_{i_1,i_2}^{'})_{n_1\times n_1}, ~(\beta_{j_1,j_2}^{'})_{n_2\times n_2},~ (\gamma_{k_1,k_2}^{'})_{n_3\times n_3},~ (\eta_{\ell_1,\ell_2}^{'})_{n_4\times n_4}, \\
&(\zeta_{m_1,m_2}^{'})_{n_5\times n_5},  ~(\theta_{q_1,q_2}^{'})_{n_6\times n_6}, ~(\lambda_{s_1,s_2}^{'})_{n_7\times n_7}, ~(\mu_{r_1,r_2}^{'})_{n_8\times n_8} \}
\end{flalign*}
and $\mathfrak{U}_1(n_1,n_2,\ldots,n_8;I_1^{'})$ is generated by $x, ~y, ~t,$ $A_i^{'}$, $B_j^{'}$, $C_k^{'}$, $D_\ell^{'}$, $E_m^{'}$, $F_q^{'}$, $G_s^{'}$, $H_r^{'}$.

When $\Phi|_H=\tau_1$ or $\tau_3$, then $\Phi(A_i)$ is $(1,x)$-skew primitive, so
\begin{equation*}
\Phi(A_i)\in \bigoplus_{i^{'}=1}^{n_1}\mathbbm{k}A_{i^{'}}^{'} \bigoplus_{j^{'}=1}^{n_2}\mathbbm{k}B_{j^{'}}^{'} \bigoplus \mathbbm{k}(1-x).
\end{equation*}
$tA_i=A_ix^2t, ~tB_{j^{'}}^{'}=-B_{j^{'}}^{'}x^2t$ imply that $\Phi(A_i)$ doesn't contain the term of $\bigoplus_{j^{'}=1}^{n_2}\mathbbm{k}B_{j^{'}}^{'}$. $xA_i = -A_ix$ implies that $\Phi(A_i)$ doesn't contain the term of $1-x$. So there exists an invertible matrix $(a_{i^{'}i})_{n_1\times n_1}$ such that
$\Phi(A_i)=\sum\limits_{i^{'}=1}^{n_1}a_{i^{'}i}A_{i^{'}}^{'}$. Similarly, $\Phi(B_j)=\sum\limits_{j^{'}=1}^{n_2}b_{j^{'}j}B_{j^{'}}^{'}$,
\begin{flalign*}
&\Phi(C_k)=\sum\limits_{k^{'}=1}^{n_3}c_{k^{'}k}C_{k^{'}}^{'},
\hspace{1em}\Phi(D_\ell)=\sum\limits_{\ell^{'}=1}^{n_4}d_{\ell^{'}\ell}D_{\ell^{'}}^{'},
\hspace{1em}\Phi(E_m)=\sum\limits_{m^{'}=1}^{n_5}e_{m^{'}m}E_{m^{'}}^{'},\\
&\Phi(F_q)=\sum\limits_{q^{'}=1}^{n_6}f_{q^{'}q}F_{q^{'}}^{'},
\hspace{1em}\Phi(G_s)=\sum\limits_{s^{'}=1}^{n_7}g_{s^{'}s}G_{s^{'}}^{'},
\hspace{1em}\Phi(H_r)=\sum\limits_{r^{'}=1}^{n_8}h_{r^{'}r}H_{r^{'}}^{'}.
\end{flalign*}
In this case, $\Phi$ is an isomorphism of Hopf algebras if and
only if the relations $(\ref{Aut1.1.1})-(\ref{Aut1.1.4})$ hold.
Similarly, $\Phi$ is an isomorphism of Hopf algebras if and only if

when $\Phi|_H=\tau_2$ or $\tau_4$, the relations $(\ref{Aut1.1.2})$, $(\ref{Aut1.1.4})$, $(\ref{Aut1.2})$ hold;

when $\Phi|_H=\tau_5$ or $\tau_7$, the relations $(\ref{Aut1.3})$, $(\ref{Aut1.2})$ hold;

when $\Phi|_H=\tau_6$ or $\tau_8$, the relations $(\ref{Aut1.1.1})$, $(\ref{Aut1.1.3})$, $(\ref{Aut1.3})$ hold;

when $\Phi|_H=\tau_9$ or $\tau_{11}$, the relations $(\ref{Aut1.5.1})-(\ref{Aut1.5.4})$ hold;

when $\Phi|_H=\tau_{10}$ or $\tau_{12}$, the relations $(\ref{Aut1.5.2})$, $(\ref{Aut1.5.4})$, $(\ref{Aut1.6})$ hold;

when $\Phi|_H=\tau_{13}$ or $\tau_{15}$, the relations $(\ref{Aut1.6})$ , $(\ref{Aut1.7})$ hold;

when $\Phi|_H=\tau_{14}$ or $\tau_{16}$, the relations $(\ref{Aut1.5.1})$, $(\ref{Aut1.5.3})$, $(\ref{Aut1.7})$ hold;

when $\Phi|_H=\tau_{17}$ or $\tau_{19}$, the relations $(\ref{Aut1.9.1})-(\ref{Aut1.9.4})$ hold;

when $\Phi|_H=\tau_{18}$ or $\tau_{20}$, the relations $(\ref{Aut1.9.1})$, $(\ref{Aut1.9.3})$, $(\ref{Aut1.10})$ hold;

when $\Phi|_H=\tau_{21}$ or $\tau_{23}$, the relations $(\ref{Aut1.10})$, $(\ref{Aut1.11})$ hold;

when $\Phi|_H=\tau_{22}$ or $\tau_{24}$, the relations $(\ref{Aut1.9.2})$, $(\ref{Aut1.9.4})$, $(\ref{Aut1.11})$ hold;

when $\Phi|_H=\tau_{25}$ or $\tau_{27}$, the relations $(\ref{Aut1.13.1})-(\ref{Aut1.13.4})$ hold;

when $\Phi|_H=\tau_{26}$ or $\tau_{28}$, the relations $(\ref{Aut1.13.1})$, $(\ref{Aut1.13.3})$, $(\ref{Aut1.14})$ hold;

when $\Phi|_H=\tau_{29}$ or $\tau_{31}$, the relations $(\ref{Aut1.14})$, $(\ref{Aut1.15})$ hold;

when $\Phi|_H=\tau_{30}$ or $\tau_{32}$, the relations $(\ref{Aut1.13.2})$, $(\ref{Aut1.13.4})$, $(\ref{Aut1.15})$ hold.
\end{proof}

\begin{defi}\label{def 2}
For $n_1,~n_2,~n_3~n_4\in \mathbb{N}^*$ with $\sum_{i=1}^4 n_i\geq 0$ and
\begin{equation*}
I_2=\{\gamma_{k_1,k_2},~ \eta_{\ell_1,\ell_2}, ~\lambda_{s_1,s_2}, ~\mu_{r_1,r_2}, ~\lambda_k, ~\zeta_\ell, ~\iota_s, ~\theta_r, ~\nu\},
\end{equation*}
where $k,~k_1,~k_2=1,..,n_1,$ $\ell,~\ell_1,~\ell_2=1,..,n_2, ~s,~s_1,~s_2=1,..,n_3, ~ r,~r_1,~r_2 =1,..,n_4$. Let us denote by $\mathfrak{U}_2(n_1,n_2,n_3,n_4;I_2)$ the algebra that is generated by $x,~y,~t,~p_1,~p_2$,
$\{C_k\},$ $\{D_\ell\},$
$\{G_s\}, \{H_r\}$
satisfying the relations $(\ref{3.1})-(\ref{3.3})$, $(\ref{C})$, $(\ref{D})$, $(\ref{G})$, $(\ref{H})$, $(\ref{Ck})$, $(\ref{Dl})$, $(\ref{Gs})$,  $(\ref{Hr})$ and
\begin{flalign}
&xp_1=p_1x,~xp_2=p_2x,~yp_1=-p_1y,~yp_2=-p_2y,~tp_1=p_1x^2t,~tp_2=-p_2x^2t,&\nonumber\\
&C_kD_\ell+D_\ell C_k=0,~~C_kG_s+G_sC_m=0,~~C_kH_r+H_rC_k=0,&\nonumber\\
&D_\ell G_s+G_sD_\ell=0,~~D_\ell H_r+H_rD_\ell=0,~~G_sH_r+H_rG_s=0,&\nonumber\\
&p_1^2=\nu(1-x^2),~~~~p_2^2=-\nu(1-x^2),~~~~p_1p_2+p_2p_1=0,&\label{e2.1}\\
&p_1C_k+C_kp_1=\lambda_k(x+x^3-2), ~~p_2C_k+C_kp_2=\lambda_k(x-x^3),&\label{e2.2}\\
&p_1D_\ell+D_\ell p_1=\zeta_\ell(x-x^3), ~~~~~~~~p_2D_\ell+D_\ell p_2=\zeta_\ell(x+x^3-2),&\label{e2.3}\\
&p_1G_s+G_sp_1=\iota_s(x+x^3-2), ~~~p_2G_s+G_sp_2=\iota_s(x^3-x),&\label{e2.4}\\
&p_1H_r+H_rp_1=\theta_r(x^3-x), ~~~~~~~p_2H_r+H_rp_2=\theta_r(x+x^3-2).&\label{e2.5}
\end{flalign}
It is a Hopf algebra with its coalgebra structure determined by

$~~~~~~~~~~~~\triangle(p_1)=p_1\otimes 1+\frac{1}{2}(1+x^2)y\otimes p_1+\frac{1}{2}(1-x^2)y\otimes p_2$,

$~~~~~~~~~~~~\triangle(p_2)=p_2\otimes 1+\frac{1}{2}(1+x^2)y\otimes p_2+\frac{1}{2}(1-x^2)y\otimes p_1$,

$~~~~~~~~~~~~\triangle(C_k)=C_k\otimes 1+xy\otimes C_k,~~~~~\triangle(D_\ell)=D_\ell\otimes 1+xy\otimes D_\ell,$

$~~~~~~~~~~~~\triangle(G_s)=G_s\otimes 1+x^3y\otimes G_s,~~~~\triangle(H_r)=H_r\otimes 1+x^3y\otimes H_r.$
\end{defi}

\begin{rem}
 It is clear that $\mathfrak{U}_2(n_1,n_2,n_3,n_4;0)\cong \mathcal{B}(\Omega_2(n_1,n_2,n_3,n_4))\sharp H$. Also note that  $\mathfrak{U}_2(n_1,n_2,n_3,n_4;I_2)\cong T(\Omega_2(n_1,n_2,n_3,n_4))/J(I_2)$, where $J(I_2)$ is a Hopf ideal generated by relations $(\ref{Ck})$, $(\ref{Dl})$, $(\ref{Gs})$,  $(\ref{Hr})$, $(\ref{e2.1})-(\ref{e2.5})$.

\end{rem}

\begin{defi}\label{def 8}
For $n_1,~n_2,~n_3,~n_4\in \mathbb{N}^*$ with $\sum_{i=1}^4 n_i\geq 0$ and a set
\begin{equation*}
I_{8}=\{(\alpha_{i_1,i_2})_{n_1\times n_1}, ~(\beta_{j_1,j_2})_{n_2\times n_2},~(\zeta_{m_1,m_2})_{n_3\times n_3}, ~(\theta_{q_1,q_2})_{n_4\times n_4},~\nu\},
\end{equation*}
we denote by $\mathfrak{U}_{8}(n_1,n_2,n_3,n_4;I_{8})$ the algebra that is generated by $x,y,t,p_1,p_2$, $\{A_i\}_{i=1,.,n_1}$, $\{B_j\}_{j=1,.,n_2}$, $\{E_m\}_{m=1,\ldots,n_3},$
$\{F_q\}_{q=1,\ldots,n_4},$
satisfying the relations $(\ref{3.1})-(\ref{3.3})$, $(\ref{A})$, $(\ref{B})$, $(\ref{E})$, $(\ref{F})$, $(\ref{Ai})$, $(\ref{Bj})$, $(\ref{Em})$, $(\ref{Fq})$, $(\ref{e2.1})$ and
\begin{flalign}
&xp_1=-p_1x,~xp_2=-p_2x,~yp_1=p_1y,~yp_2=p_2y,~tp_1=p_1t,~tp_2=-p_2t,&\nonumber\\
&A_iB_j+B_jA_i=0,~~A_iE_m+E_mA_i=0,~~A_iF_q+F_qA_i=0,&\nonumber\\
&B_jE_m+E_mB_j=0,~~B_jF_q+F_qB_j=0,~~E_mF_q+F_qE_m=0,&\nonumber\\
&p_1A_i+A_ip_1=0,~~p_2A_i+A_ip_2=0,~~p_1B_j+B_jp_1=0,~~p_2B_j+B_jp_2=0,&\label{e8.2}\\
&p_1E_m+E_mp_1=0,~~p_2E_m+E_mp_2=0,~~p_1F_q+F_qp_1=0,~~p_2F_q+F_qp_2=0.&\label{e8.3}
\end{flalign}
It is a Hopf algebra with its coalgebra structure determined by

$~~~~~~~~~~~\triangle(p_1)=p_1\otimes 1+\frac{1}{2}(1+x^2)x\otimes p_1+\frac{\xi}{2}(1-x^2)x\otimes p_2$,

$~~~~~~~~~~~\triangle(p_2)=p_2\otimes 1+\frac{1}{2}(1+x^2)x\otimes p_2-\frac{\xi}{2}(1-x^2)x\otimes p_1$,

$~~~~~~~~~~~\triangle(A_i)=A_i\otimes 1+x\otimes A_i,~~~~~~~~\triangle(B_j)=B_j\otimes 1+x\otimes B_j,$

$~~~~~~~~~~~\triangle(E_m)=E_m\otimes 1+x^3\otimes E_m,~~~~\triangle(F_q)=F_q\otimes 1+x^3\otimes F_q.$
\end{defi}

\begin{pro}\label{2}
$(1)$ Let $A$ be a finite-dimensional Hopf algebra with coradical $H$ such that its infinitesimal braiding is $\Omega_i(n_1,n_2,n_3,n_4)$ for $i\in\{2,8\}$, then $A\cong \mathfrak{U}_i(n_1,n_2,n_3,n_4;I_i)$.

$(2) ~\mathfrak{U}_2(n_1,n_2,n_3,n_4;I_2)\cong \mathfrak{U}_2(n_1,n_2,n_3,n_4;I_2^{'})$ if and only if there exist nonzero parameter $z$ and invertible matrices
$(c_{k^{'}k})_{n_1\times n_1}$, $(c_{k_1^{'}k_1})_{n_1\times n_1}$, $(c_{k_2^{'}k_2})_{n_1\times n_1}$, $(d_{\ell^{'}\ell})_{n_2\times n_2}$,
$(d_{\ell_1^{'}\ell_1})_{n_2\times n_2}$, $(d_{\ell_2^{'}\ell_2})_{n_2\times n_2}$,
$(g_{s^{'}s})_{n_3\times n_3}$, $(g_{s_1^{'}s_1})_{n_3\times n_3}$, $(g_{s_2^{'}s_2})_{n_3\times n_3}$, $(h_{r^{'}r})_{n_4\times n_4}$,
$(h_{r_1^{'}r_1})_{n_4\times n_4}$, $(h_{r_2^{'}r_2})_{n_4\times n_4}$
satisfying $z^2v^{'}=v$, $(\ref{Aut1.1.2})$, $(\ref{Aut1.1.4})$ and
\begin{align}
\begin{split}
&z\sum\limits_{k^{'}=1}^{n_1}c_{k^{'}k}\lambda_{k^{'}}^{'}=\lambda_{k},
~~z\sum\limits_{\ell^{'}=1}^{n_2}d_{\ell^{'}\ell}\zeta_{\ell^{'}}^{'}=\zeta_{\ell},
~~z\sum\limits_{s^{'}=1}^{n_3}g_{s^{'}s}\iota_{s^{'}}^{'}=\iota_{s},
~~z\sum\limits_{r^{'}=1}^{n_4}h_{r^{'}r}\theta_{r^{'}}^{'}=\theta_{r}.\label{Aut2.1}
\end{split}
\end{align}
or $n_1=n_2,~n_3=n_4$
satisfying $z^2v^{'}=-v$, $(\ref{Aut1.3})$ and
\begin{align}
\begin{split}
&z\sum\limits_{\ell^{'}=1}^{n_1}c_{\ell^{'}k}\zeta_{\ell^{'}}^{'}=\lambda_{k},
~~z\sum\limits_{k^{'}=1}^{n_1}d_{k^{'}\ell}\lambda_{k^{'}}^{'}=\zeta_{\ell},
~~z\sum\limits_{r^{'}=1}^{n_3}g_{r^{'}s}\theta_{r^{'}}^{'}=\iota_{s},
~~z\sum\limits_{s^{'}=1}^{n_3}h_{s^{'}r}\iota_{s^{'}}^{'}=\theta_{r}.\label{Aut2.2}
\end{split}
\end{align}
or $n_1=n_3,~n_2=n_4$
satisfying $z^2v^{'}=v$, $(\ref{Aut1.5.2})$, $(\ref{Aut1.5.4})$ and
\begin{align}
\begin{split}
&z\sum\limits_{s^{'}=1}^{n_1}c_{s^{'}k}\iota_{s^{'}}^{'}=\lambda_{k},
~~z\sum\limits_{r^{'}=1}^{n_2}d_{r^{'}\ell}\theta_{r^{'}}^{'}=\zeta_{\ell},
~~z\sum\limits_{k^{'}=1}^{n_1}g_{k^{'}s}\lambda_{k^{'}}^{'}=\iota_{s},
~~z\sum\limits_{\ell^{'}=1}^{n_2}h_{\ell^{'}r}\zeta_{\ell^{'}}^{'}=\theta_{r}.\label{Aut2.3}
\end{split}
\end{align}
or $n_1=n_4,~n_2=n_3$
satisfying $z^2v^{'}=-v$, $(\ref{Aut1.7})$ and
\begin{align}
\begin{split}
&z\sum\limits_{r^{'}=1}^{n_1}c_{r^{'}k}\theta_{r^{'}}^{'}=\lambda_{k},
~~z\sum\limits_{s^{'}=1}^{n_2}d_{s^{'}l}\iota_{s^{'}}^{'}=\zeta_{\ell},
~~z\sum\limits_{\ell^{'}=1}^{n_2}g_{\ell^{'}s}\zeta_{\ell^{'}}^{'}=\iota_{s},
~~z\sum\limits_{k^{'}=1}^{n_1}h_{k^{'}r}\lambda_{k^{'}}^{'}=\theta_{r}.\label{Aut2.4}
\end{split}
\end{align}

$(3)$ $\mathfrak{U}_8(n_1,n_2,n_3,n_4;I_8)\cong \mathfrak{U}_8(n_1,n_2,n_3,n_4;I_8^{'})$ if and only if there exist nonzero parameter $z$ and invertible matrices
$(a_{i_1^{'}i_1})_{n_1\times n_1}$, $(a_{i_2^{'}i_2})_{n_1\times n_1}$, $(b_{j_1^{'}j_1})_{n_2\times n_2}$, $(b_{j_2^{'}j_2})_{n_2\times n_2}$,
$(e_{m_1^{'}m_1})_{n_3\times n_3}$, $(e_{m_2^{'}m_2})_{n_3\times n_3}$,
$(f_{q_1^{'}q_1})_{n_4\times n_4}$, $(f_{q_2^{'}q_2})_{n_4\times n_4}$
satisfying $z^2v^{'}=v$, $(\ref{Aut1.1.1})$, $(\ref{Aut1.1.3})$
or $z^2v^{'}=-v$ or $n_1=n_2,~ n_3=n_4$ satisfying $z^2v^{'}=v$,
$(\ref{Aut1.2})$ or $z^2v^{'}=-v$
or $n_1=n_3, ~n_2=n_4$ satisfying $z^2v^{'}=v$, $(\ref{Aut1.5.1})$, $(\ref{Aut1.5.3})$ or $z^2v^{'}=-v$
or $n_1=n_4, ~n_2=n_3$ satisfying $z^2v^{'}=v$,
$(\ref{Aut1.6})$ or $z^2v^{'}=-v$.
\end{pro}
\begin{proof} $(1)$ We prove the claim only for $\Omega_2(n_1,n_2,n_3,n_4)$, since the proof for $\Omega_8(n_1,n_2,n_3,n_4)$ follows the same lines. By Theorem $\ref{gr}$, we have that gr$(A)\cong \mathcal{B}(\Omega_2(n_1,n_2,n_3,n_4))\sharp H$. Let $M_1=\mathbbm{k}\{ p_1,p_2\}$. From Proposition \ref{1} and a direct computation, we get that the relations except $(\ref{e2.1})-(\ref{e2.5})$ hold in $A$. As
\begin{flalign*}
&\triangle(p_1)=p_1\otimes 1+\frac{1}{2}(1+x^2)y\otimes p_1+\frac{1}{2}(1-x^2)y\otimes p_2,\\
&\triangle(p_2)=p_2\otimes 1+\frac{1}{2}(1+x^2)y\otimes p_2+\frac{1}{2}(1-x^2)y\otimes p_1,
\end{flalign*}
it follows that
\begin{flalign*}
&\triangle(p_1^2+p_2^2)=(p_1^2+p_2^2)\otimes 1+1\otimes (p_1^2+p_2^2),\\
&\triangle(p_1^2-p_2^2)=(p_1^2-p_2^2)\otimes 1+x^2\otimes (p_1^2-p_2^2),\\
&\triangle(p_1p_2+p_2p_1)=(p_1p_2+p_2p_1)\otimes 1+1\otimes (p_1p_2+p_2p_1).
\end{flalign*}
So we have the relation $(\ref{e2.1})$ holds in $A$. Since
$\triangle(C_k)=C_k\otimes 1+xy\otimes C_k$,
we get that
\begin{flalign*}
\triangle(p_1C_k+C_kp_1)=&(p_1C_k+C_kp_1)\otimes 1+\frac{1}{2}(x+x^3)\otimes(p_1C_k+C_kp_1)\\
&+\frac{1}{2}(x-x^3)\otimes(p_2C_k+C_kp_2),\\
\triangle(p_2C_k+C_kp_2)=&(p_2C_k+C_kp_2)\otimes 1+\frac{1}{2}(x+x^3)\otimes(p_2C_k+C_kp_2)\\
&+\frac{1}{2}(x-x^3)\otimes(p_1C_k+C_kp_1).
\end{flalign*}
Since
\begin{flalign*}
&t(p_1C_k+C_kp_1)=(p_1C_k+C_kp_1)t, ~~~~~~~t(p_2C_k+C_kp_2)=-(p_2C_k+C_kp_2)t,
\end{flalign*}
we have the relation $(\ref{e2.2})$ holds in $A$.

Similarly, the relations $(\ref{e2.3}),~(\ref{e2.4}),~(\ref{e2.5})$ hold in $A$. Then there is a surjective Hopf homomorphism
from $\mathfrak{U}_2(n_1,n_2,n_3,n_4;I_2)$ to $A$.
By the Diamond Lemma, we have that the linear basis of
$\mathfrak{U}_2(n_1,n_2,n_3,n_4;I_2)$ is
\begin{equation*}
\{C_k^{\gamma_k}
D_\ell^{\eta_\ell}G_s^{\lambda_s}H_r^{\mu_r}p_1^ip_2^jx^ey^ft^g;~\gamma_k,\eta_\ell,\lambda_s,\mu_r,i,j,f,g\in\mathbb{I}_{0,1},
e\in\mathbb{I}_{0,3}\}.
\end{equation*}
Then $\dim A$ = $\dim \mathfrak{U}_2(n_1,n_2,n_3,n_4;I_2)$ and
whence $A\cong \mathfrak{U}_2(n_1,n_2,n_3,n_4;I_2)$.

$(2)$ Suppose that $\Phi: \mathfrak{U}_2(n_1,n_2,n_3,n_4;I_2)\rightarrow \mathfrak{U}_2(n_1,n_2,n_3,n_4;I_2^{'})$ is an isomorphism of
Hopf algebras. By Proposition \ref{1}, we know that $\Phi|_H\in \{\tau_1,\cdots,\tau_{16}\}$. If $\Phi|_H\in\{\tau_1,\tau_2,\tau_3,\tau_4\}$, then
\begin{flalign*}
&\Phi(p_1)=zp_1^{'},
\hspace{1em}\Phi(p_2)=zp_2^{'}, \hspace{1em}\Phi(C_k)=\sum\limits_{k^{'}=1}^{n_1}c_{k^{'}k}C_{k^{'}}^{'},\\
&\Phi(D_\ell)=\sum\limits_{\ell^{'}=1}^{n_2}d_{\ell^{'}\ell}D_{\ell^{'}}^{'},
\hspace{1em}\Phi(G_s)=\sum\limits_{s^{'}=1}^{n_3}g_{s^{'}s}G_{s^{'}}^{'},
\hspace{1em}\Phi(H_r)=\sum\limits_{r^{'}=1}^{n_4}h_{r^{'}r}H_{r^{'}}^{'}.
\end{flalign*}
Whence the relations $z^2v^{'}=v$, (\ref{Aut1.1.2}), (\ref{Aut1.1.4}), (\ref{Aut2.1}) hold.
Similarly, $\Phi$ is an isomorphism of Hopf algebras if and only if

\hspace{1em}when $\Phi|_H\in\{\tau_5,\tau_6,\tau_7,\tau_8\}$, the relations $z^2v^{'}=-v$, $(\ref{Aut1.3})$, $(\ref{Aut2.2})$ hold;

\hspace{1em}when $\Phi|_H\in\{\tau_9,\tau_{10},\tau_{11},\tau_{12}\}$, the relations $z^2v^{'}=v$, $(\ref{Aut1.5.2})$, $(\ref{Aut1.5.4})$, $(\ref{Aut2.3})$ hold;

\hspace{1em}when $\Phi|_H\in\{\tau_{13},\tau_{14},\tau_{15},\tau_{16}\}$, the relations $z^2v^{'}=-v$, $(\ref{Aut1.7})$, $(\ref{Aut2.4})$ hold.

The proof of $(3)$ is completely analogous.
\end{proof}

\begin{defi}\label{def 4}
For $n_1,~n_2,~n_3,~n_4\in \mathbb{N}^*$ with $\sum_{i=1}^4 n_i\geq 0$ and a set
\begin{equation*}
I_4=\{(\alpha_{i_1,i_2})_{n_1\times n_1},~ (\gamma_{k_1,k_2})_{n_2\times n_2}, ~(\zeta_{m_1,m_2})_{n_3\times n_3}, ~ (\lambda_{s_1,s_2})_{n_4\times n_4},~ \nu\},
\end{equation*}
we denote by $\mathfrak{U}_4(n_1,n_2,n_3,n_4;I_4)$ the algebra that is generated by $x,y,t,$
$p_1,p_2,$ $\{A_i\}_{i=1,.,n_1},$ $\{C_k\}_{k=1,.,n_2},$  $\{E_m\}_{m=1,..,n_3},$
$\{G_s\}_{s=1,..,n_4}$ satisfying the relations $(\ref{3.1})-(\ref{3.3})$, $(\ref{A})$, $(\ref{C})$, $(\ref{E})$, $(\ref{G})$, $(\ref{Ai})$, $(\ref{Ck})$, $(\ref{Em})$, $(\ref{Gs})$, $(\ref{e2.1})$ and
\begin{flalign}
&xp_1=p_1x,~xp_2=-p_2x,~yp_1=p_1y,~yp_2=p_2y,~tp_1=-p_1t,~tp_2=-p_2t,&\nonumber\\
&A_iC_k-C_kA_i=0,~~A_iE_m+E_mA_i=0,~~A_iG_s-G_sA_i=0,&\nonumber\\
&C_kE_m-E_mC_m=0,~~C_kG_s+G_sC_m=0,~~E_mG_s-G_sE_m=0,&\nonumber\\
&p_1A_i-A_ip_1=0, ~~p_2A_i+A_ip_2=0,~~p_1C_k-C_kp_1=0, ~~p_2C_k+C_kp_2=0,&\label{e4.2}\\
&p_1E_m-E_mp_1=0, ~~p_2E_m+E_mp_2=0,~~p_1G_s-G_sp_1=0, ~~p_2G_s+G_sp_2=0.&\label{e4.3}
\end{flalign}
It is a Hopf algebra with its coalgebra structure determined by

$~~~~~~~~~~~\triangle(p_1)=p_1\otimes 1+\frac{1}{2}(1+x^2)t\otimes p_1+\frac{1}{2}(1-x^2)yt\otimes p_2$,

$~~~~~~~~~~~\triangle(p_2)=p_2\otimes 1+\frac{1}{2}(1+x^2)yt\otimes p_2+\frac{1}{2}(1-x^2)t\otimes p_1$,

$~~~~~~~~~~~\triangle(A_i)=A_i\otimes 1+x\otimes A_i,~~~~~~~~\triangle(C_k)=C_k\otimes 1+xy\otimes C_k,$

$~~~~~~~~~~~\triangle(E_m)=E_m\otimes 1+x^3\otimes E_m,~~~~\triangle(G_s)=G_s\otimes 1+x^3y\otimes G_s.$
\end{defi}

\begin{defi}\label{def 5}
For $n_1,~n_2,~n_3,~n_4\in \mathbb{N}^*$ with $\sum_{i=1}^4 n_i\geq 0$ and a set
\begin{equation*}
I_5=\{\alpha_{i_1,i_2}, ~\eta_{\ell_1,\ell_2}, ~\zeta_{m_1,m_2},\mu_{r_1,r_2},~\nu,~\lambda_i, ~\kappa_\ell,~ \iota_m, ~\theta_r\}
\end{equation*}
with $i,~i_1,~i_2\in \{1,..,n_1\},$ $\ell,~\ell_1,~\ell_2\in \{1,..,n_2\}, ~m,~m_1,~m_2\in \{1,..,n_3\}, ~ r,~r_1,~r_2\in \{1,..,n_4\},$ denote by $\mathfrak{U}_5(n_1,n_2,n_3,n_4;I_5)$ the algebra that is generated by $x,~y,~t,~p_1,~p_2$, $\{A_i\},$ $\{D_\ell\},$ $\{E_m\},$
$\{H_r\}$ satisfying the relations $(\ref{3.1})-(\ref{3.3})$,
$(\ref{A})$, $(\ref{D})$, $(\ref{E})$, $(\ref{H})$, $(\ref{Ai})$, $(\ref{Dl})$, $(\ref{Em})$, $(\ref{Hr})$ and
\begin{flalign}
&xp_1=p_1x,~xp_2=-p_2x,~yp_1=-p_1y,~yp_2=-p_2y,~tp_1=-p_1x^2t,~tp_2=p_2x^2t,&\nonumber\\
&A_iD_\ell-D_\ell A_i=0,~~A_iE_m+E_mA_i=0,~~A_iH_r-H_rA_i=0,&\nonumber\\
&D_\ell E_m-E_mD_\ell=0,~~D_\ell H_r+H_rD_\ell=0,~~E_mH_r-H_rE_m=0,&\nonumber\\
&p_1^2=\nu(1-x^2),~~~~p_2^2=-\nu(1-x^2),~~~~p_1p_2-p_2p_1=0,&\label{e5.1}\\
&p_1A_i-A_ip_1=\lambda_i(xyt-x^3yt), ~~p_2A_i+A_ip_2=\lambda_i(2-xyt-x^3yt),&\label{e5.2}\\
&p_1D_\ell+D_\ell p_1=\kappa_\ell(2-xyt-x^3yt), ~~p_2D_\ell-D_\ell p_2=\kappa_\ell(xyt-x^3yt),&\label{e5.3}\\
&p_1E_m-E_mp_1=\iota_m(xyt-x^3yt), ~~p_2E_m+E_mp_2=\iota_m(-2+xyt+x^3yt),&\label{e5.4}\\
&p_1H_r+H_rp_1=\theta_r(-2+xyt+x^3yt), ~~p_2H_r-H_rp_2=\theta_r(xyt-x^3yt).&\label{e5.5}
\end{flalign}
It is a Hopf algebra with its coalgebra structure determined by

$~~~~~~~~~~~\triangle(p_1)=p_1\otimes 1+\frac{1}{2}(1+x^2)t\otimes p_1+\frac{1}{2}(1-x^2)yt\otimes p_2$,

$~~~~~~~~~~~\triangle(p_2)=p_2\otimes 1+\frac{1}{2}(1+x^2)yt\otimes p_2+\frac{1}{2}(1-x^2)t\otimes p_1$,

$~~~~~~~~~~~\triangle(A_i)=A_i\otimes 1+x\otimes A_i,~~~~~~~~\triangle(D_\ell)=D_\ell\otimes 1+xy\otimes D_\ell,$

$~~~~~~~~~~~\triangle(E_m)=E_m\otimes 1+x^3\otimes E_m,~~~~\triangle(H_r)=H_r\otimes 1+x^3y\otimes H_r.$
\end{defi}

\begin{pro}\label{4}
$(1)$ Let $A$ be a finite-dimensional Hopf algebra with coradical $H$ such that its infinitesimal braiding is $\Omega_i(n_1,n_2,n_3,n_4)$ for $i\in\{4,5\}$, then $A\cong \mathfrak{U}_i(n_1,n_2,n_3,n_4;I_i)$.

$(2)$ $\mathfrak{U}_4(n_1,n_2,n_3,n_4;I_4)\cong \mathfrak{U}_4(n_1,n_2,n_3,n_4;I_4^{'})$ if and only if there exist nonzero parameter $z$ such that $z^2v^{'}=v$ and invertible matrices
$(a_{i_1^{'}i_1})_{n_1\times n_1}$, $(a_{i_2^{'}i_2})_{n_1\times n_1}$,
$(c_{k_1^{'}k_1})_{n_2\times n_2}$, $(c_{k_2^{'}k_2})_{n_2\times n_2}$,
$(e_{m_1^{'}m_1})_{n_3\times n_3}$, $(e_{m_2^{'}m_2})_{n_3\times n_3}$,
$(g_{s_1^{'}s_1})_{n_4\times n_4}$, $(g_{s_2^{'}s_2})_{n_4\times n_4}$
satisfying the first equations of $(\ref{Aut1.1.1})-(\ref{Aut1.1.4})$
or $n_1=n_3,~n_2=n_4$ satisfying the first equations of $(\ref{Aut1.5.1})-(\ref{Aut1.5.4})$
or $n_1=n_2, ~n_3=n_4$ satisfying the first equations of $(\ref{Aut1.9.1})-(\ref{Aut1.9.4})$
or $n_1=n_4,~n_2=n_3$ satisfying the first equations of $(\ref{Aut1.13.1})-(\ref{Aut1.13.4})$.

$(3)$ $\mathfrak{U}_5(n_1,n_2,n_3,n_4;I_5)\cong \mathfrak{U}_5(n_1,n_2,n_3,n_4;I_5^{'})$ if and only if there exist nonzero parameter $z$ and invertible matrices
$(a_{i^{'}i})_{n_1\times n_1}$, $(a_{i_1^{'}i_1})_{n_1\times n_1}$, $(a_{i_2^{'}i_2})_{n_1\times n_1}$, $(d_{\ell^{'}\ell})_{n_2\times n_2}$,
$(d_{\ell_1^{'}\ell_1})_{n_2\times n_2}$, $(d_{\ell_2^{'}\ell_2})_{n_2\times n_2}$,
$(e_{m^{'}m})_{n_3\times n_3}$, ~$(e_{m_1^{'}m_1})_{n_3\times n_3}$,~$(e_{m_2^{'}m_2})_{n_3\times n_3}$, $(h_{r^{'}r})_{n_4\times n_4}$,
$(h_{r_1^{'}r_1})_{n_4\times n_4}$, $(h_{r_2^{'}r_2})_{n_4\times n_4}$
satisfying $z^2v^{'}=v$, the first equations of $(\ref{Aut1.1.1}), ~(\ref{Aut1.1.3})$, the second equations of $(\ref{Aut1.1.2}), ~(\ref{Aut1.1.4})$ and
\begin{align}\label{Aut5.1}
\begin{split}
&z\sum\limits_{i^{'}=1}^{n_1}a_{i^{'}i}\lambda_{i^{'}}^{'}=\lambda_{i},
~~z\sum\limits_{\ell^{'}=1}^{n_2}d_{\ell^{'}\ell}\kappa_{\ell^{'}}^{'}=\kappa_{\ell},
~~z\sum\limits_{m^{'}=1}^{n_3}e_{m^{'}m}\iota_{m^{'}}^{'}=\iota_{m},
~~z\sum\limits_{r^{'}=1}^{n_4}h_{r^{'}r}\theta_{r^{'}}^{'}=\theta_{r}.
\end{split}
\end{align}
or
\begin{align}
\begin{split}
&z\sum\limits_{i^{'}=1}^{n_1}a_{i^{'}i}\lambda_{i^{'}}^{'}=-\lambda_{i},
~~z\sum\limits_{\ell^{'}=1}^{n_2}d_{\ell^{'}\ell}\kappa_{\ell^{'}}^{'}=\kappa_{\ell},
~~z\sum\limits_{m^{'}=1}^{n_3}e_{m^{'}m}\iota_{m^{'}}^{'}=-\iota_{m},
~~z\sum\limits_{r^{'}=1}^{n_4}h_{r^{'}r}\theta_{r^{'}}^{'}=\theta_{r}.
\end{split}
\end{align}
or $n_1=n_3,~ n_2=n_4$ satisfying $z^2v^{'}=v$, the first equations of $(\ref{Aut1.5.1}), ~(\ref{Aut1.5.3})$, the second equations of $(\ref{Aut1.5.2}), ~(\ref{Aut1.5.4})$ and
\begin{align}
\begin{split}
&z\sum\limits_{m^{'}=1}^{n_1}a_{m^{'}i}\iota_{m^{'}}^{'}=-\lambda_{i},
~z\sum\limits_{r^{'}=1}^{n_2}d_{r^{'}\ell}\theta_{r^{'}}^{'}=-\kappa_{\ell},
~z\sum\limits_{i^{'}=1}^{n_1}e_{i^{'}m}\lambda_{i^{'}}^{'}=-\iota_{m},
~z\sum\limits_{\ell^{'}=1}^{n_2}h_{\ell^{'}r}\kappa_{\ell^{'}}^{'}=-\theta_{r}.\label{Aut5.2}
\end{split}
\end{align}
or
\begin{align}
\begin{split}
&z\sum\limits_{m^{'}=1}^{n_1}a_{m^{'}i}\iota_{m^{'}}^{'}=\lambda_{i},
~z\sum\limits_{r^{'}=1}^{n_2}d_{r^{'}\ell}\theta_{r^{'}}^{'}=-\kappa_{\ell},
~z\sum\limits_{i^{'}=1}^{n_1}e_{i^{'}m}\lambda_{i^{'}}^{'}=\iota_{m},
~z\sum\limits_{\ell^{'}=1}^{n_2}h_{\ell^{'}r}\kappa_{\ell^{'}}^{'}=-\theta_{r}.
\end{split}
\end{align}
or $n_1=n_2,~n_3=n_4$ satisfying $z^2v^{'}=-v$, the first equation of $(\ref{Aut1.11})$, the second equation of $(\ref{Aut1.10})$ and
\begin{align}
\begin{split}
&z\sum\limits_{\ell^{'}=1}^{n_1}a_{\ell^{'}i}\kappa_{\ell^{'}}^{'}=\lambda_{i},
~z\sum\limits_{i^{'}=1}^{n_1}d_{i^{'}\ell}\lambda_{i^{'}}^{'}=\kappa_{\ell},
~z\sum\limits_{r^{'}=1}^{n_3}e_{r^{'}m}\theta_{r^{'}}^{'}=\iota_{m},
~z\sum\limits_{m^{'}=1}^{n_3}h_{m^{'}r}\iota_{m^{'}}^{'}=\theta_{r}.\label{Aut5.3}
\end{split}
\end{align}
or
\begin{align}
\begin{split}
&z\sum\limits_{\ell^{'}=1}^{n_1}a_{\ell^{'}i}\kappa_{\ell^{'}}^{'}=-\lambda_{i},
~z\sum\limits_{i^{'}=1}^{n_1}d_{i^{'}\ell}\lambda_{i^{'}}^{'}=\kappa_{\ell},
~z\sum\limits_{r^{'}=1}^{n_3}e_{r^{'}m}\theta_{r^{'}}^{'}=-\iota_{m},
~z\sum\limits_{m^{'}=1}^{n_3}h_{m^{'}r}\iota_{m^{'}}^{'}=\theta_{r}.
\end{split}
\end{align}
or $n_1=n_4,~n_2=n_3$ satisfying $z^2v^{'}=-v$, the first equation of $(\ref{Aut1.15})$, the second equation of $(\ref{Aut1.14})$ and
\begin{align}
\begin{split}
&z\sum\limits_{r^{'}=1}^{n_1}a_{r^{'}i}\theta_{r^{'}}^{'}=-\lambda_{i},
~z\sum\limits_{m^{'}=1}^{n_2}d_{m^{'}\ell}\iota_{m^{'}}^{'}=-\kappa_{\ell},
~z\sum\limits_{\ell^{'}=1}^{n_2}e_{\ell^{'}m}\kappa_{\ell^{'}}^{'}=-\iota_{m},
~z\sum\limits_{i^{'}=1}^{n_1}h_{i^{'}r}\lambda_{i^{'}}^{'}=-\theta_{r}.\label{Aut5.4}
\end{split}
\end{align}
or
\begin{align}
\begin{split}
&z\sum\limits_{r^{'}=1}^{n_1}a_{r^{'}i}\theta_{r^{'}}^{'}=\lambda_{i},
~z\sum\limits_{m^{'}=1}^{n_2}d_{m^{'}\ell}\iota_{m^{'}}^{'}=-\kappa_{\ell},
~z\sum\limits_{\ell^{'}=1}^{n_2}e_{\ell^{'}m}\kappa_{\ell^{'}}^{'}=\iota_{m},
~z\sum\limits_{i^{'}=1}^{n_1}h_{i^{'}r}\lambda_{i^{'}}^{'}=-\theta_{r}.
\end{split}
\end{align}
\end{pro}
\begin{proof} $(1)$ We prove the claim only for $\Omega_4(n_1,n_2,n_3,n_4)$, since the proof for $\Omega_5(n_1,n_2,n_3,n_4)$
follows the same lines. By Theorem $\ref{gr}$, we have that gr$(A)\cong \mathcal{B}(\Omega_4(n_1,n_2,n_3,n_4))\sharp H$.
Let $M_3=\mathbbm{k}\{ p_1,p_2\}$. Similar to Proposition \ref{2}, we only need to prove that the relations $(\ref{e4.2})$, $(\ref{e4.3})$ hold in $A$.
Since $\triangle(A_i)=A_i\otimes 1+x\otimes A_i,$ we have that
\begin{flalign*}
\triangle(p_1A_i-A_ip_1)=&(p_1A_i-A_ip_1)\otimes1+
\frac{1}{2}(xt+x^3t)\otimes(p_1A_i-A_ip_1)\\
&+\frac{1}{2}(x^3yt-xyt)\otimes(p_2A_i+A_ip_2),\\
\triangle(p_2A_i+A_ip_2)=&(p_2A_i+A_ip_2)\otimes1+
\frac{1}{2}(xyt+x^3yt)\otimes(p_2A_i+A_ip_2)\\
&+\frac{1}{2}(x^3t-xt)\otimes(p_1A_i-A_ip_1).
\end{flalign*}
If the relations $p_1A_i-A_ip_1=0$ admit non-trivial deformations, then $p_1A_i-A_ip_1\in A_{[1]}$ is a linear combination of

$\{x^jy^lt^r$, $A_ix^jy^lt^r,$ $C_kx^jy^lt^r,$ $E_mx^jy^lt^r,$ $G_sx^jy^lt^r$, $p_1x^jy^lt^r$, $p_2x^jy^lt^r,j\in \mathbb{I}_{0,3},l,r\in \mathbb{I}_{0,1}\}$.
\\Since
\begin{equation*}
y(p_1A_i-A_ip_1)=-(p_1A_i-A_ip_1)y,\hspace{1em} t(p_1A_i-A_ip_1)=-(p_1A_i-A_ip_1)x^2t,
\end{equation*}
then $p_1A_i-A_ip_1=0$ hold in $A$. Whence by the  above two equations we have that the relations $p_2A_i+A_ip_2=0$ hold in $A$.

Similarly, other relations  hold in $A$.
Then there is a surjective Hopf morphism from $\mathfrak{U}_4(n_1,n_2,n_3,n_4;I_4)$ to $A$.
By the Diamond Lemma, the linear basis of
$\mathfrak{U}_4(n_1,n_2,n_3,n_4;I_4)$ is $\{A_i^{\alpha_i}C_k^{\gamma_k}
E_m^{\zeta_m}G_s^{\lambda_s}p_1^jp_2^lx^ey^ft^g\}$
for all parameters $\alpha_i,~\gamma_k,~\zeta_m,~\lambda_s,~j,~l,~f,~g\in\mathbb{I}_{0,1},
~e\in\mathbb{I}_{0,3}$. Then $\dim A$ = $\dim \mathfrak{U}_4(n_1,n_2,n_3,n_4;I_4)$ and
whence $A\cong \mathfrak{U}_4(n_1,n_2,n_3,n_4;I_4)$.

$(2)$ Suppose that $\Phi: \mathfrak{U}_4(n_1,n_2,n_3,n_4;I_4)\rightarrow \mathfrak{U}_4(n_1,n_2,n_3,n_4;I_4^{'})$ is an isomorphism of
Hopf algebras. Similar to the proof of Proposition~\ref{2}, $\Phi|_H\in \{\tau_1,\tau_3,\tau_9,\tau_{11},\tau_{17},\tau_{19},\tau_{25},\tau_{27}\}$. If $\Phi|_H=\tau_1$, then
\begin{flalign*}
&\Phi(p_1)=zp_1^{'},\hspace{1em}\Phi(p_2)=zp_2^{'}, \hspace{1em}\Phi(A_i)=\sum\limits_{i^{'}=1}^{n_1}a_{i^{'}i}A_{i^{'}}^{'},\\
&\Phi(C_k)=\sum\limits_{k^{'}=1}^{n_2}c_{k^{'}k}C_{k^{'}}^{'},
~\Phi(E_m)=\sum\limits_{m^{'}=1}^{n_3}e_{m^{'}m}E_{m^{'}}^{'},
~\Phi(G_s)=\sum\limits_{s^{'}=1}^{n_4}g_{s^{'}s}G_{s^{'}}^{'}.
\end{flalign*}
Whence the first equations of $(\ref{Aut1.1.1})-(\ref{Aut1.1.4})$ hold.
Similarly, $\Phi$ is an isomorphism of Hopf algebras if and only if

\hspace{2em}when $\Phi|_H=\tau_3$, the first equations of $(\ref{Aut1.1.1})-(\ref{Aut1.1.4})$ hold;

\hspace{2em}when $\Phi|_H\in\{\tau_9,\tau_{11}\}$, the first equations of $(\ref{Aut1.5.1})-(\ref{Aut1.5.4})$ hold;

\hspace{2em}when $\Phi|_H\in\{\tau_{17},\tau_{19}\}$, the first equations of $(\ref{Aut1.9.1})-(\ref{Aut1.9.4})$ hold;

\hspace{2em}when $\Phi|_H\in\{\tau_{25},\tau_{27}\}$, the first equations of $(\ref{Aut1.13.1})-(\ref{Aut1.13.4})$ hold.

The proof of $(3)$ is completely analogous.
\end{proof}

\begin{defi}\label{def 14}
For a set of parameters $I_{14}=\{\lambda,\mu,\alpha \}$, denote by $\mathfrak{U}_{14}(I_{14})$ the algebra that is generated by $x,~y,~t$, $p_1,~p_2$, $q_1,~q_2$ satisfying the relations $(\ref{3.1})-(\ref{3.3})$ and
\begin{flalign}
&xp_1=p_1x,~~yp_1=-p_1y,~~tp_1=p_1x^2t,
~~xp_2=p_2x,~~yp_2=-p_2y,~~tp_2=-p_2x^2t,&\label{e14.1}\\
&xq_1=q_1x,~~yq_1=-q_1y,~~tq_1=q_1x^2t,
~~xq_2=q_2x, ~~yq_2=-q_2y,~~tq_2=-q_2x^2t,&\\
&p_1^2=\lambda(1-x^2),\hspace{2em}p_2^2=-\lambda(1-x^2),\hspace{2em}p_1p_2+p_2p_1=0,&\label{e14.3}\\
&q_1^2=\mu(1-x^2),\hspace{2em}q_2^2=-\mu(1-x^2),\hspace{2em}q_1q_2+q_2q_1=0,&\label{e14.4}\\
&p_1q_1+q_1p_1=\alpha(1-x^2),~p_2q_2+q_2p_2=-\alpha(1-x^2),~p_1q_2+q_2p_1=p_2q_1+q_1p_2=0.&\label{e14.5}
\end{flalign}
It is a Hopf algebra with its coalgebra structure determined by

$~~~~~~~~~~~~\triangle(p_1)=p_1\otimes 1+\frac{1}{2}(1+x^2)y\otimes p_1+\frac{1}{2}(1-x^2)y\otimes p_2$,

$~~~~~~~~~~~~\triangle(p_2)=p_2\otimes 1+\frac{1}{2}(1+x^2)y\otimes p_2+\frac{1}{2}(1-x^2)y\otimes p_1$,

$~~~~~~~~~~~~\triangle(q_1)=q_1\otimes 1+\frac{1}{2}(1+x^2)y\otimes q_1+\frac{1}{2}(1-x^2)y\otimes q_2$,

$~~~~~~~~~~~~\triangle(q_2)=q_2\otimes 1+\frac{1}{2}(1+x^2)y\otimes q_2+\frac{1}{2}(1-x^2)y\otimes q_1$.

\end{defi}

\begin{defi}\label{def 15}
Denote by $\mathfrak{U}_{15}(\lambda,\mu)$ the algebra that is generated by $x,~y,~t$, $p_1,~p_2$, $q_1,~q_2$ satisfying the relations $(\ref{3.1})-(\ref{3.3})$, $(\ref{e14.1})$, $(\ref{e14.3})$, $(\ref{e14.4})$ and
\begin{flalign}
&xq_1=-q_1x,~yq_1=-q_1y,~tq_1=q_1x^2t,
~xq_2=-q_2x,~yq_2=-q_2y,~tq_2=-q_2x^2t,&\\
&p_1q_1+q_1p_1=0,\hspace{1em}p_2q_2+q_2p_2=0,\hspace{1em}p_1q_2+q_2p_1=0,\hspace{1em}p_2q_1+q_1p_2=0.&
\end{flalign}
It is a Hopf algebra with the same coalgebra structure as $\mathfrak{U}_{14}(I_{14})$.
\end{defi}

\begin{defi}\label{def 16}
Denote by $\mathfrak{U}_{16}(\lambda,\mu)$ the algebra that is generated by $x,~y,~t$, $p_1,~p_2$, $q_1,~q_2$ satisfying the relations $(\ref{3.1})-(\ref{3.3})$, $(\ref{e14.1})$, $(\ref{e14.3})$, $(\ref{e14.4})$ and
\begin{flalign}
&xq_1=-q_1x, \hspace{1em}yq_1=q_1y,\hspace{1em}tq_1=q_1t,
\hspace{1em}xq_2=-q_2x, \hspace{1em}yq_2=q_2y,\hspace{1em}tq_2=-q_2t,&\label{e16.1}\\
&p_1q_1-q_1p_1=0,\hspace{1em}p_2q_2-q_2p_2=0,\hspace{1em}p_1q_2-q_2p_1=0,\hspace{1em}p_2q_1-q_1p_2=0.&
\end{flalign}
It is a Hopf algebra with its coalgebra structure determined by

$~~~~~~~~~~~~\triangle(q_1)=q_1\otimes 1+\frac{1}{2}(1+x^2)x\otimes q_1+\frac{1}{2}(1-x^2)x\otimes q_2$,

$~~~~~~~~~~~~\triangle(q_2)=q_2\otimes 1+\frac{1}{2}(1+x^2)x\otimes q_2+\frac{1}{2}(1-x^2)x\otimes q_1$.

\end{defi}

\begin{defi}\label{def 17}
For a set of parameters $I_{17}=\{\lambda,\mu,\alpha \}$, denote by $\mathfrak{U}_{17}(I_{17})$ the algebra that is generated by $x,~y,~t$, $p_1,~p_2$, $q_1,~q_2$ satisfying the relations $(\ref{3.1})-(\ref{3.3})$, $(\ref{e14.3})-(\ref{e14.5})$ and
\begin{flalign}
&xp_1=p_1x, \hspace{1em}yp_1=p_1y,\hspace{1em}tp_1=-p_1t,
\hspace{1em}xp_2=-p_2x,\hspace{1em}yp_2=p_2y,\hspace{1em}tp_2=-p_2t,&\label{e17.1}\\
&xq_1=q_1x, \hspace{1em}yq_1=q_1y,\hspace{1em}tq_1=-q_1t,
\hspace{1em}xq_2=-q_2x, \hspace{1em}yq_2=q_2y,\hspace{1em}tq_2=-q_2t.&\label{e17.2}
\end{flalign}
It is a Hopf algebra with its coalgebra structure determined by

$~~~~~~~~~~~~\triangle(p_1)=p_1\otimes 1+\frac{1}{2}(1+x^2)t\otimes p_1+\frac{1}{2}(1-x^2)yt\otimes p_2$,

$~~~~~~~~~~~~\triangle(p_2)=p_2\otimes 1+\frac{1}{2}(1+x^2)yt\otimes p_2+\frac{1}{2}(1-x^2)t\otimes p_1$,

$~~~~~~~~~~~~\triangle(q_1)=q_1\otimes 1+\frac{1}{2}(1+x^2)t\otimes q_1+\frac{1}{2}(1-x^2)yt\otimes q_2$,

$~~~~~~~~~~~~\triangle(q_2)=q_2\otimes 1+\frac{1}{2}(1+x^2)yt\otimes q_2+\frac{1}{2}(1-x^2)t\otimes q_1$.

\end{defi}

\begin{defi}\label{def 18}
For a set of parameters $I_{18}=\{\lambda,\mu,\alpha \}$, denote by $\mathfrak{U}_{18}(I_{18})$ the algebra that is  generated by $x,~y,~t$, $p_1,~p_2$, $q_1,~q_2$ satisfying the relations $(\ref{3.1})-(\ref{3.3})$, $(\ref{e14.3})$, $(\ref{e14.4})$, $(\ref{e17.1})$, $(\ref{e17.2})$ and
\begin{flalign}
&p_1q_1+q_1p_1=\alpha(y+x^2y-2),\hspace{2em}p_2q_2+q_2p_2=\alpha(y-x^2y),\\
&p_1q_2+q_2p_1=0,\hspace{2em}p_2q_1+q_1p_2=0.&
\end{flalign}
It is a Hopf algebra with its coalgebra structure determined by

$~~~~~~~~~~~~\triangle(q_1)=q_1\otimes 1+\frac{1}{2}(1+x^2)yt\otimes q_1+\frac{1}{2}(1-x^2)t\otimes q_2$,

$~~~~~~~~~~~~\triangle(q_2)=q_2\otimes 1+\frac{1}{2}(1+x^2)t\otimes q_2+\frac{1}{2}(1-x^2)yt\otimes q_1$.
\end{defi}

\begin{defi}\label{def 19}
Denote by $\mathfrak{U}_{19}(\lambda,\mu)$ the algebra that is generated by $x,~y,~t$, $p_1,~p_2$, $q_1,~q_2$ satisfying the relations $(\ref{3.1})-(\ref{3.3})$, $(\ref{e14.3})$, $(\ref{e17.1})$ and
\begin{flalign}
&xq_1=q_1x, ~yq_1=-q_1y, ~tq_1=-q_1x^2t,
xq_2=-q_2x, ~yq_2=-q_2y, ~tq_2=-q_2x^2t,&\\
&q_1^2=\mu(1-x^2),\hspace{2em}q_2^2=-\mu(1-x^2),\hspace{2em}q_1q_2-q_2q_1=0,&\\
&p_1q_1+q_1p_1=0,\hspace{1em}p_2q_2-q_2p_2=0,
\hspace{1em}p_1q_2+q_2p_1=0,\hspace{1em}p_2q_1-q_1p_2=0.&
\end{flalign}
It is a Hopf algebra with its coalgebra structure determined by

$~~~~~~~~~~~~\triangle(q_1)=q_1\otimes 1+\frac{1}{2}(1+x^2)xt\otimes q_1+\frac{1}{2}(1-x^2)xyt\otimes q_2$,

$~~~~~~~~~~~~\triangle(q_2)=q_2\otimes 1+\frac{1}{2}(1+x^2)xyt\otimes q_2+\frac{1}{2}(1-x^2)xt\otimes q_1$.
\end{defi}

\begin{pro}\label{14}
$(1)$ Suppose $A$ is a finite-dimensional Hopf algebra with coradical $H$ such
that its infinitesimal braiding is isomorphic to $\Omega_{i}$ for $i\in\{14,15,16,17,18,19\}$, then $A\cong \mathfrak{U}_{14}(I_{14})$ or $\mathfrak{U}_{15}(\lambda,\mu)$ or $\mathfrak{U}_{16}(\lambda,\mu)$ or $\mathfrak{U}_{17}(I_{17})$, $\mathfrak{U}_{18}(I_{18})$, $\mathfrak{U}_{19}(\lambda,\mu)$.

$(2)$ $\mathfrak{U}_{17}(I_{17})\cong \mathfrak{U}_{17}(I_{17}^{'})$ if and only if there exist nonzero parameters $a_1,~a_2$, $b_1,~b_2$
satisfying $(\ref{Aut14.1})$. $\mathfrak{U}_{14}(I_{14})\cong \mathfrak{U}_{14}(I_{14}^{'})$ if and only if there exist nonzero parameters $a_1,~a_2$, $b_1,~b_2$
such that
\begin{align}
\begin{split}
&a_1^2\lambda^{'}+a_1a_2\alpha^{'}+a_2^2\mu^{'}=\lambda,
\hspace{2em}b_1^2\lambda^{'}+b_1b_2\alpha^{'}+b_2^2\mu^{'}=\mu,\\
&2a_1b_1\lambda^{'}+2a_2b_2\mu^{'}+(a_1b_2+a_2b_1)\alpha^{'}=\alpha,
\label{Aut14.1}
\end{split}
\end{align}
or satisfying
\begin{align}
\begin{split}
&a_1^2\lambda^{'}+a_1a_2\alpha^{'}+a_2^2\mu^{'}=-\lambda,
\hspace{2em}b_1^2\lambda^{'}+b_1b_2\alpha^{'}+b_2^2\mu^{'}=-\mu,\\
&2a_1b_1\lambda^{'}+2a_2b_2\mu^{'}+(a_1b_2+a_2b_1)\alpha^{'}=-\alpha.
\label{Aut14.2}
\end{split}
\end{align}

$(3) ~\mathfrak{U}_{18}(I_{18})\cong \mathfrak{U}_{18}(I_{18}^{'})$ if and only if there exist nonzero parameters $a_1$, $a_2$ such that
\begin{align}
\begin{split}
&{a_1}^2\lambda^{'}=\lambda,
\hspace{2em}{a_2}^2\mu^{'}=\mu,
\hspace{2em}a_1a_2\alpha^{'}=\alpha,
\label{Aut18.1}
\end{split}
\end{align}
or satisfying
\vspace{-2em}
\begin{align}
\begin{split}
&{a_1}^2\mu^{'}=\lambda,
\hspace{2em}{a_2}^2\lambda^{'}=\mu,
\hspace{2em}a_1a_2\alpha^{'}=\alpha.
\label{Aut18.2}
\end{split}
\end{align}

$(4)$ Let $i\in\{15,16,19\}$. $\mathfrak{U}_{i}(\lambda,\mu)\cong \mathfrak{U}_{i}(1,1)$ for $\lambda,\mu\neq 0$,  and $\mathfrak{U}_{i}(1,1)\ncong \mathfrak{U}_{i}(0,0)$.
\end{pro}
\begin{proof} $(1)$ We prove the claim for $\Omega_{14}$. The proofs for others follow the same lines. By Theorem $\ref{gr}$, we have that gr$(A)\cong \mathcal{B}(\Omega_{14})\sharp H$. Let $M_1=\mathbbm{k}\{ p_1,p_2\}=\mathbbm{k}\{ q_1,q_2\}$. By Proposition \ref{2}, we only need to prove that $(\ref{e14.5})$ holds in $A$. After a direct computation, we have that
\begin{flalign*}
\triangle(p_1q_1+q_1p_1)=&(p_1q_1+q_1p_1)\otimes 1+\frac{1}{2}(1+x^2)\otimes(p_1q_1+q_1p_1)\\
&+\frac{1}{2}(1-x^2)\otimes(p_2q_2+q_2p_2),\\
\triangle(p_2q_2+q_2p_2)=&(p_2q_2+q_2p_2)\otimes 1+\frac{1}{2}(1+x^2)\otimes(p_2q_2+q_2p_2)\\
&+\frac{1}{2}(1-x^2)\otimes(p_1q_1+q_1p_1).
\end{flalign*}
So $p_1q_1+q_1p_1=\alpha(1-x^2), ~p_2q_2+q_2p_2=-\alpha(1-x^2)$ for
some $\alpha\in\mathbbm{k}$. Similarly, other relations hold in $A$.
Then there is a surjective Hopf homomorphism from
$\mathfrak{U}_{14}(I_{14})$ to $A$. We can observe that all elements
of $\mathfrak{U}_{14}(I_{14})$ can be expressed by linear
combinations of
\begin{equation*}
\{p_1^ip_2^jq_1^kq_2^l x^ey^ft^g;~ i,j,k,l,f,g\in\mathbb{I}_{0,1}, ~e\in\mathbb{I}_{0,3}\}.
\end{equation*}
In fact, according to the Diamond Lemma, the set is a basis of
$\mathfrak{U}_{14}(I_{14})$. Then $\dim A$ =
$\dim \mathfrak{U}_{14}(I_{14})$,
whence $A\cong \mathfrak{U}_{14}(I_{14})$.

$(2)$ Suppose $\Phi : \mathfrak{U}_{14}(I_{14})\rightarrow \mathfrak{U}_{14}(I_{14}^{'})$ is an isomorphism of
Hopf algebras. Similar to the proof of Proposition \ref{2}, $\Phi|_H\in \{\tau_1-\tau_{16}\}$: when $\Phi|_H\in\{\tau_1-\tau_4,\tau_{9}-\tau_{12}\}$,
\begin{equation*}
\Phi(p_1)=a_1p_1^{'}+a_2q_1^{'},~ \Phi(p_2)=a_1p_2^{'}+a_2q_2^{'}, ~\Phi(q_1)=b_1p_1^{'}+b_2q_1^{'},~ \Phi(q_2)=b_1p_2^{'}+b_2q_2^{'},
\end{equation*}
then the relation $(\ref{Aut14.1})$ holds;  when $\Phi|_H\in\{\tau_5-\tau_8,\tau_{13}-\tau_{16}\}$,
\begin{equation*}
\Phi(p_1)=a_1p_2^{'}+a_2q_2^{'}, ~\Phi(p_2)=a_1p_1^{'}+a_2q_1^{'}, ~\Phi(q_1)=b_1p_2^{'}+b_2q_2^{'},~ \Phi(q_2)=b_1p_1^{'}+b_2q_1^{'},
\end{equation*}
the relation $(\ref{Aut14.2})$ holds.

For $\mathfrak{U}_{17}(I_{17})$, the proof is completely analogous.

$(3)$ Suppose $\Phi : \mathfrak{U}_{18}(I_{18})\rightarrow \mathfrak{U}_{18}(I_{18}^{'})$ is an isomorphism of
Hopf algebras.

When
$\Phi|_H\in\{\tau_1,\tau_{9},\tau_{17},\tau_{25}\}$,
\begin{equation*}
\Phi(p_1)=a_1p_1^{'}, \hspace{1em}\Phi(p_2)=a_1p_2^{'}, \hspace{1em}\Phi(q_1)=a_2q_1^{'}, \hspace{1em}\Phi(q_2)=a_2q_2^{'};
\end{equation*}
when $\Phi|_H\in\{\tau_3,\tau_{11},\tau_{19},\tau_{27},\}$,
\begin{equation*}
\Phi(p_1)=a_1p_1^{'}, \hspace{1em}\Phi(p_2)=-a_1p_2^{'}, \hspace{1em}\Phi(q_1)=a_2q_1^{'},\hspace{1em} \Phi(q_2)=-a_2q_2^{'},
\end{equation*}
then the relation $(\ref{Aut18.1})$ holds. When $\Phi|_H\in\{\tau_5,\tau_{13},\tau_{21},\tau_{29}\}$,
\begin{equation*}
\Phi(p_1)=a_1q_1^{'},\hspace{1em} \Phi(p_2)=a_1q_2^{'}, \hspace{1em}\Phi(q_1)=a_2p_1^{'}, \hspace{1em}\Phi(q_2)=a_2p_2^{'};
\end{equation*}
when $\Phi|_H\in\{\tau_7,\tau_{15},\tau_{23},\tau_{31},\}$,
\begin{equation*}
\Phi(p_1)=a_1q_1^{'}, \hspace{1em}\Phi(p_2)=-a_1q_2^{'}, \hspace{1em}\Phi(q_1)=a_2p_1^{'},\hspace{1em} \Phi(q_2)=-a_2p_2^{'},
\end{equation*}
then the relation $(\ref{Aut18.2})$ holds.

$(4)$ When $\lambda,~\mu\neq 0$, $\Phi:\mathfrak{U}_{15}(1,1)\rightarrow \mathfrak{U}_{15}(\lambda,\mu)$
by $\Phi|_{H}=id$, $p_i\mapsto \sqrt{\lambda}p_i$, $q_i\mapsto \sqrt{\mu}q_i$ for $i=1,2$.
For $\mathfrak{U}_{16}(\lambda,\mu)$, $\mathfrak{U}_{19}(\lambda,\mu)$, the proofs are completely analogous.
\end{proof}

\begin{defi}\label{def 20}
For a set of parameters $I_{20}=\{\lambda,\mu,\alpha \}$, denote by $\mathfrak{U}_{20}(I_{20})$ the algebra that is
generated by $x,~y,~t$, $p_1,~p_2$, $q_1,~q_2$ satisfying the relations $(\ref{3.1})-(\ref{3.3})$ and
\begin{flalign}
&xp_1=p_1x,~yp_1=-p_1y,~tp_1=-p_1x^2t,
~xp_2=-p_2x,~yp_2=-p_2y,~tp_2=p_2x^2t,&\label{e20.1}\\
&xq_1=q_1x,~yq_1=-q_1y,~tq_1=-q_1x^2t,
~xq_2=-q_2x,~yq_2=-q_2y,~tq_2=q_2x^2t,&\\
&p_1^2=\lambda(1-x^2),\hspace{2em}p_2^2=-\lambda(1-x^2),\hspace{2em}p_1p_2-p_2p_1=0,&\label{e20.2}\\
&q_1^2=\mu(1-x^2),\hspace{2em}q_2^2=-\mu(1-x^2),\hspace{2em}q_1q_2-q_2q_1=0,&\label{e20.3}\\
&p_1q_1+q_1p_1=\alpha(1-x^2),\hspace{2em}p_2q_2+q_2p_2=-\alpha(1-x^2),\label{e20.5}\\
&p_1q_2-q_2p_1=0,\hspace{2em}p_2q_1-q_1p_2=0.&\label{e20.4}
\end{flalign}
It is a Hopf algebra with the same coalgebra structure as $\mathfrak{U}_{17}(I_{17})$.
\end{defi}

\begin{defi}\label{def 21}
Denote by $\mathfrak{U}_{21}(\lambda,\mu)$ the algebra that is generated by $x,~y,~t$, $p_1,~p_2$, $q_1,~q_2$
satisfying the relations $(\ref{3.1})-(\ref{3.3})$, $(\ref{e20.1})$, $(\ref{e20.2})$, $(\ref{e20.3})$ and
\begin{flalign}
&xq_1=q_1x,~yq_1=-q_1y,~tq_1=q_1x^2t,
~xq_2=-q_2x,~yq_2=-q_2y,~tq_2=-q_2x^2t,&\\
&p_1q_1-q_1p_1=0,\hspace{1em}p_2q_2-q_2p_2=0,
\hspace{1em}p_1q_2+q_2p_1=0,\hspace{1em}p_2q_1+q_1p_2=0.&
\end{flalign}
It is a Hopf algebra with the same coalgebra structure as $\mathfrak{U}_{18}(I_{18})$.
\end{defi}

\begin{defi}\label{def 22}
For a set of parameters $I_{22}=\{\lambda,\mu,\alpha \}$, denote by $\mathfrak{U}_{22}(I_{22})$ the algebra that
is generated by $x,~y,~t$, $p_1,~p_2$, $q_1,~q_2$ satisfying the relations $(\ref{3.1})-(\ref{3.3})$, $(\ref{e14.3})$, $(\ref{e14.4})$, $(\ref{e16.1})$ and
\begin{flalign}
&xp_1=-p_1x, \hspace{1em}yp_1=p_1y,\hspace{1em}tp_1=p_1t,
\hspace{1em}xp_2=-p_2x,\hspace{1em}yp_2=p_2y,\hspace{1em}tp_2=-p_2t,&\label{e22.1}\\
&p_1q_1+q_1p_1=p_2q_2+q_2p_2=\alpha(1-x^2),~p_1q_2+q_2p_1=0,~p_2q_1+q_1p_2=0.&
\end{flalign}
It is a Hopf algebra with its coalgebra structure determined by

$~~~~~~~~~~~~\triangle(p_1)=p_1\otimes 1+\frac{1}{2}(1+x^2)x\otimes p_1+\frac{1}{2}(1-x^2)x\otimes p_2$,

$~~~~~~~~~~~~\triangle(p_2)=p_2\otimes 1+\frac{1}{2}(1+x^2)x\otimes p_2+\frac{1}{2}(1-x^2)x\otimes p_1$,

$~~~~~~~~~~~~\triangle(q_1)=q_1\otimes 1+\frac{1}{2}(1+x^2)x\otimes q_1+\frac{1}{2}(1-x^2)x\otimes q_2$,

$~~~~~~~~~~~~\triangle(q_2)=q_2\otimes 1+\frac{1}{2}(1+x^2)x\otimes q_2+\frac{1}{2}(1-x^2)x\otimes q_1$.

\end{defi}

\begin{defi}\label{def 23}
For a set of parameters $I_{23}=\{\lambda,\mu,\alpha \}$, denote by $\mathfrak{U}_{23}(I_{23})$ the algebra
that is generated by $x,~y,~t$, $p_1,~p_2$, $q_1,~q_2$ satisfying the relations $(\ref{3.1})-(\ref{3.3})$,
$(\ref{e14.3})$, $(\ref{e14.4})$, $(\ref{e16.1})$, $(\ref{e22.1})$ and
\begin{flalign}
&p_1q_1+q_1p_1=\alpha(y+x^2y-2),\hspace{2em}p_2q_2+q_2p_2=\alpha(x^2y-y),\\
&p_1q_2+q_2p_1=0,\hspace{2em}p_2q_1+q_1p_2=0.&
\end{flalign}
It is a Hopf algebra with its coalgebra structure determined by

$~~~~~~~~~~~~\triangle(p_1)=p_1\otimes 1+\frac{1}{2}(1+x^2)x\otimes p_1+\frac{1}{2}(1-x^2)x\otimes p_2$,

$~~~~~~~~~~~~\triangle(p_2)=p_2\otimes 1+\frac{1}{2}(1+x^2)x\otimes p_2+\frac{1}{2}(1-x^2)x\otimes p_1$,

$~~~~~~~~~~~~\triangle(q_1)=q_1\otimes 1+\frac{1}{2}(1+x^2)xy\otimes q_1+\frac{1}{2}(1-x^2)xy\otimes q_2$,

$~~~~~~~~~~~~\triangle(q_2)=q_2\otimes 1+\frac{1}{2}(1+x^2)xy\otimes q_2+\frac{1}{2}(1-x^2)xy\otimes q_1$.

\end{defi}

Similar to Proposition \ref{14}, we have the following
\begin{pro}\label{20}
$(1)$ Suppose $A$ is a finite-dimensional Hopf algebra with coradical $H$ such
that its infinitesimal braiding is isomorphic to $\Omega_{i}$ for $i\in\{20,21,22,23\}$, then $A\cong \mathfrak{U}_{i}(I_{i})$.

$(2) ~\mathfrak{U}_{20}(I_{20})\cong \mathfrak{U}_{20}(I_{20}^{'})$, $\mathfrak{U}_{22}(I_{22})\cong \mathfrak{U}_{22}(I_{22}^{'})$ if and only if there
exist nonzero parameters $a_1,~a_2$, $b_1,~ b_2$
satisfying $(\ref{Aut14.1})$ or $(\ref{Aut14.2})$.

$(3) ~\mathfrak{U}_{23}(I_{23})\cong \mathfrak{U}_{23}(I_{23}^{'})$ if and only if there exist nonzero parameters $a_1$, $a_2$ satisfying
$(\ref{Aut18.1})$ or $(\ref{Aut18.2})$ or satisfying
\begin{align}
\begin{split}
&{a_1}^2\lambda^{'}=-\lambda,
\hspace{2em}{a_2}^2\mu^{'}=-\mu,
\hspace{2em}\alpha=0,
\label{Aut23.1}
\end{split}
\end{align}
or satisfying
\begin{align}
\begin{split}
&{a_1}^2\mu^{'}=-\lambda,
\hspace{2em}{a_2}^2\lambda^{'}=-\mu,
\hspace{2em}\alpha=0.
\label{Aut23.2}
\end{split}
\end{align}

$(4)$ $\mathfrak{U}_{21}(\lambda,\mu)\cong \mathfrak{U}_{21}(1,1)$ for $\lambda,\mu\neq 0$,  and $\mathfrak{U}_{21}(1,1)\ncong \mathfrak{U}_{21}(0,0)$.
\end{pro}

\begin{defi}\label{def 24}
Denote by $\mathfrak{U}_{24}(\lambda)$ the algebra that is generated by $x,~y,~t$, $p_1,~p_2$, $q_1,~q_2$ satisfying the relations $(\ref{3.1})-(\ref{3.3})$ and
\begin{flalign}
&xp_1=\xi p_1x,~yp_1=-p_1y,~tp_1=p_2x^2t,
~xp_2=-\xi p_2x,~yp_2=-p_2y,~tp_2=p_1x^2t,&\\
&xq_1=\xi q_1x,~yq_1=-q_1y,~tq_1=q_2x^2t,
~xq_2=-\xi q_2x,~yq_2=-q_2y,~tq_2=q_1x^2t,&\\
&p_1^2=0,\hspace{1em}p_2^2=0,\hspace{1em}p_1p_2+p_2p_1=0,
\hspace{1em}q_1^2=0,\hspace{1em}q_2^2=0,\hspace{1em}q_1q_2+q_2q_1=0,&\label{e24.1}\\
&p_1q_1+q_1p_1=0,\hspace{2em}p_2q_2+q_2p_2=0,&\label{e24.2}\\
&p_1q_2+q_2p_1=\lambda(1-x^2y),\hspace{2em}p_2q_1+q_1p_2=-\lambda(1-x^2y).&\label{e24.3}
\end{flalign}
It is a Hopf algebra with its coalgebra structure determined by

$~~~~~~~~~~~~\triangle(p_1)=p_1\otimes 1+y\otimes p_1$,
\hspace{2em}$\triangle(p_2)=p_2\otimes 1+x^2\otimes p_2$,

$~~~~~~~~~~~~\triangle(q_1)=q_1\otimes 1+y\otimes q_1$,
\hspace{2em}$\triangle(q_2)=q_2\otimes 1+x^2\otimes q_2$.

\end{defi}

\begin{defi}\label{def 26}
For a set of parameters $I_{26}=\{\lambda,\mu,\alpha \}$, denote by $\mathfrak{U}_{26}(I_{26})$ the algebra that is generated by
$x,~y,~t$, $p_1,~p_2$, $q_1,~q_2$ satisfying the relations $(\ref{3.1})-(\ref{3.3})$ and
\begin{flalign}
&xp_1=\xi p_1x,~yp_1=p_1y,~tp_1=p_2t,
~xp_2=-\xi p_2x,~yp_2=p_2y,~tp_2=p_1t,&\label{e26.1}\\
&xq_1=\xi q_1x,~yq_1=q_1y,~tq_1=q_2t,
~xq_2=-\xi q_2x,~yq_2=q_2y,~tq_2=q_1t,&\label{e26.4}\\
&p_1^2=0,\hspace{2em}p_2^2=0,\hspace{2em}p_1p_2+p_2p_1=\lambda(1-y),&\label{e26.2}\\
&q_1^2=0,\hspace{2em}q_2^2=0,\hspace{2em}q_1q_2+q_2q_1=\mu(1-y),&\label{e26.3}\\
&p_1q_1+q_1p_1=0,\hspace{1em}p_2q_2+q_2p_2=0,
\hspace{1em}p_1q_2+q_2p_1=p_2q_1+q_1p_2=\alpha(1-y).&
\end{flalign}
It is a Hopf algebra with its coalgebra structure determined by

$~~~~~~~~~~~~\triangle(p_1)=p_1\otimes 1+x^2\otimes p_1$,
\hspace{2em}$\triangle(p_2)=p_2\otimes 1+x^2y\otimes p_2$,

$~~~~~~~~~~~~\triangle(q_1)=q_1\otimes 1+x^2\otimes q_1$,
\hspace{2em}$\triangle(q_2)=q_2\otimes 1+x^2y\otimes q_2$.

\end{defi}

\begin{defi}\label{def 27}
Denote by $\mathfrak{U}_{27}(\lambda,\mu)$ the algebra that is generated by $x,~y,~t$, $p_1,~p_2$, $q_1,~q_2$ satisfying the
relations $(\ref{3.1})-(\ref{3.3})$, $(\ref{e26.1})$, $(\ref{e26.2})$, $(\ref{e26.3})$ and
\begin{flalign}
&xq_1=-\xi q_1x,~yq_1=q_1y,~tq_1=q_2t,
~xq_2=\xi q_2x,~yq_2=q_2y,~tq_2=q_1t,&\label{e27.1}\\
&p_1q_1+q_1p_1=0,\hspace{1em}p_2q_2+q_2p_2=0,
\hspace{1em}p_1q_2+q_2p_1=0,\hspace{1em}p_2q_1+q_1p_2=0.&
\end{flalign}
It is a Hopf algebra with the same coalgebra structure as $\mathfrak{U}_{26}(I_{26})$.

\end{defi}

\begin{pro}\label{24}
$(1)$ Suppose $A$ is a finite-dimensional Hopf algebra with coradical $H$ such
that its infinitesimal braiding is isomorphic to $\Omega_{i}$ for $i\in\{24,26,27\}$, then $A\cong \mathfrak{U}_{24}(\lambda)$ or
$\mathfrak{U}_{26}(I_{26})$ or $\mathfrak{U}_{27}(\lambda,\mu)$.

$(2)$ $\mathfrak{U}_{24}(\lambda)\cong \mathfrak{U}_{24}(1)$, $\mathfrak{U}_{27}(\lambda,\mu)\cong \mathfrak{U}_{27}(1,1)$ for
$\lambda,~\mu\neq 0$, $\mathfrak{U}_{24}(1)\ncong \mathfrak{U}_{24}(0)$ and $\mathfrak{U}_{27}(1,1)\ncong \mathfrak{U}_{27}(0,0)$.

$(3) ~\mathfrak{U}_{26}(I_{26})\cong \mathfrak{U}_{26}(I_{26}^{'})$ if and only if there exist nonzero parameters $a_1,~a_2$, $b_1,~b_2$
such that
\begin{align}
\begin{split}
&a_1^2\lambda^{'}+2a_1a_2\alpha^{'}+a_2^2\mu^{'}=\lambda,
\hspace{1em}b_1^2\lambda^{'}+2b_1b_2\alpha^{'}+b_2^2\mu^{'}=\mu,\\
&a_1b_1\lambda^{'}+(a_1b_2+a_2b_1)\alpha^{'}+a_2b_2\mu^{'}=\alpha,
\label{Aut26.1}
\end{split}
\end{align}
or satisfying
\begin{align}
\begin{split}
&a_1^2\lambda^{'}+2a_1a_2\alpha^{'}+a_2^2\mu^{'}=-\xi\lambda,
\hspace{1em}b_1^2\lambda^{'}+2b_1b_2\alpha^{'}+b_2^2\mu^{'}=-\xi\mu,\\
&a_1b_1\lambda^{'}+(a_1b_2+a_2b_1)\alpha^{'}+a_2b_2\mu^{'}=-\xi\alpha,
\label{Aut26.2}
\end{split}
\end{align}
or satisfying
\begin{align}
\begin{split}
&a_1^2\lambda^{'}+2a_1a_2\alpha^{'}+a_2^2\mu^{'}=-\lambda,
\hspace{1em}b_1^2\lambda^{'}+2b_1b_2\alpha^{'}+b_2^2\mu^{'}=-\mu,\\
&a_1b_1\lambda^{'}+(a_1b_2+a_2b_1)\alpha^{'}+a_2b_2\mu^{'}=-\alpha,
\label{Aut26.3}
\end{split}
\end{align}
or satisfying
\begin{align}
\begin{split}
&a_1^2\lambda^{'}+2a_1a_2\alpha^{'}+a_2^2\mu^{'}=\xi\lambda,
\hspace{1em}b_1^2\lambda^{'}+2b_1b_2\alpha^{'}+b_2^2\mu^{'}=\xi\mu,\\
&a_1b_1\lambda^{'}+(a_1b_2+a_2b_1)\alpha^{'}+a_2b_2\mu^{'}=\xi\alpha.
\label{Aut26.4}
\end{split}
\end{align}

\end{pro}

\begin{proof} $(1)$ We prove the claim for $\Omega_{24}$. The proofs for $\Omega_{26}$, $\Omega_{27}$ follow the same lines.
By Theorem $\ref{gr}$, we have that gr$(A)\cong \mathcal{B}(\Omega_{24})\sharp H$. Let $M_{13}=\mathbbm{k}\{ p_1,p_2\}=\mathbbm{k}\{ q_1,q_2\}$.
Similar to Proposition \ref{14}, we only need to prove that $(\ref{e24.1})-(\ref{e24.3})$ in Definition \ref{def 24} hold in $A$. Since
\begin{flalign*}
&\triangle(p_1)=p_1\otimes 1+y\otimes p_1,
\hspace{1em}\triangle(p_2)=p_2\otimes 1+x^2\otimes p_2,\\
&\triangle(q_1)=q_1\otimes 1+y\otimes q_1,
\hspace{1em}\triangle(q_2)=q_2\otimes 1+x^2\otimes q_2,
\end{flalign*}
then $p_1^2$,
$p_2^2$, $p_1q_1+q_1p_1$, $p_2q_2+q_2p_2$ are primitive elements and
\begin{flalign*}
&\triangle(p_1p_2+p_2p_1)=(p_1p_2+p_2p_1)\otimes 1+x^2y\otimes (p_1p_2+p_2p_1),\\
&\triangle(p_1q_2+q_2p_1)=(p_1q_2+q_2p_1)\otimes 1+x^2y\otimes (p_1q_2+q_2p_1),\\
&\triangle(p_2q_1+q_1p_2)=(p_2q_1+q_1p_2)\otimes 1+x^2y\otimes (p_2q_1+q_1p_2).
\end{flalign*}
As
\begin{equation*}
t(p_1p_2+p_2p_1)=-(p_1p_2+p_2p_1)t,\hspace{1em} t(p_1q_2+q_2p_1)=-(p_2q_1+q_1p_2)t,
\end{equation*}
the relations $(\ref{e24.1})-(\ref{e24.3})$ hold in $A$. Then there
is a surjective Hopf morphism from $\mathfrak{U}_{24}(I_{24})$ to
$A$. We can observe that each element of $\mathfrak{U}_{24}(I_{24})$
can be expressed by a linear combination of $\{p_1^ip_2^jq_1^kq_2^l
x^ey^ft^g\}$ for all parameters $i,~j,~k,~l$,
$f,~g\in\mathbb{I}_{0,1}, ~e\in\mathbb{I}_{0,3}$. In fact, according
to the Diamond Lemma, the set is a basis of
$\mathfrak{U}_{24}(I_{24})$. Then $\dim A$ = $\dim \mathfrak{U}_{24}(I_{24})$, whence $A\cong
\mathfrak{U}_{24}(I_{24})$.

$(2)$ The proof is the same as that of Proposition $\ref{14}~(4)$.

$(3)$ Suppose $\Phi : \mathfrak{U}_{26}(I_{26})\rightarrow \mathfrak{U}_{26}(I_{26}^{'})$ is an isomorphism of
Hopf algebras. Similar to Proposition $\ref{14}$,  $\Phi|_H\in\{\tau_1-\tau_8,\tau_{17}-\tau_{24}\}$.
When $\Phi|_H\in\{\tau_1,\tau_5,\tau_{17},\tau_{21}\}$,
\begin{equation*}
\Phi(p_1)=a_1p_1^{'}+a_2q_1^{'},\ \ \Phi(p_2)=a_1p_2^{'}+a_2q_2^{'}, \ \ \Phi(q_1)=b_1p_1^{'}+b_2q_1^{'},\ \ \Phi(q_2)=b_1p_2^{'}+b_2q_2^{'},
\end{equation*}
then the relation $(\ref{Aut26.1})$ holds. When
$\Phi|_H\in\{\tau_2,\tau_6,\tau_{18},\tau_{22}\}$,
\begin{equation*}
\Phi(p_1)=a_1p_1^{'}+a_2q_1^{'},~ \Phi(p_2)=\xi(a_1p_2^{'}+a_2q_2^{'}), ~\Phi(q_1)=b_1p_1^{'}+b_2q_1^{'}, ~\Phi(q_2)=\xi(b_1p_2^{'}+b_2q_2^{'}),
\end{equation*}
then the relation $(\ref{Aut26.2})$ holds. When $\Phi|_H\in\{\tau_3,\tau_7,\tau_{19},\tau_{23}\}$,
\begin{equation*}
\Phi(p_1)=a_1p_1^{'}+a_2q_1^{'}, ~\Phi(p_2)=-(a_1p_2^{'}+a_2q_2^{'}), ~\Phi(q_1)=b_1p_1^{'}+b_2q_1^{'},~ \Phi(q_2)=-(b_1p_2^{'}+b_2q_2^{'}),
\end{equation*}
then the relation $(\ref{Aut26.3})$ holds. When $\Phi|_H\in\{\tau_4,\tau_8,\tau_{20},\tau_{24}\}$,
\begin{equation*}
\Phi(p_1)=a_1p_1^{'}+a_2q_1^{'},~ \Phi(p_2)=-\xi(a_1p_2^{'}+a_2q_2^{'}), ~\Phi(q_1)=b_1p_1^{'}+b_2q_1^{'}, ~\Phi(q_2)=-\xi(b_1p_2^{'}+b_2q_2^{'}),
\end{equation*}
then the relation $(\ref{Aut26.4})$ holds.

\end{proof}

\begin{pro}\label{25}
Suppose $A$ is a finite-dimensional Hopf algebra with coradical $H$ such
that its infinitesimal braiding is $\Omega_{25}$, then $A\cong \mathcal{B}(\Omega_{25})\sharp H$.
\end{pro}
\begin{proof} Let $M_{13}=\mathbbm{k}\{ p_1,p_2\}$, $M_{14}=\mathbbm{k}\{ q_1,q_2\}$. By Theorem $\ref{gr}$, we have that gr$(A)\cong \mathcal{B}(\Omega_{25})\sharp H$. Recall that $\mathcal{B}(\Omega_{25})\sharp H$ is the algebra generated by $p_1,~p_2$, $q_1,~q_2$, $x,~y,~t$ with
$p_1,~p_2$, $q_1,~q_2$ satisfying the relation of $\mathcal{B}(\Omega_{25})$, $x,~y,~t$ satisfying the relations of $H$, and all together satisfying the cross relations:
\begin{flalign*}
&xp_1=\xi p_1x, \hspace{1em}yp_1=-p_1y,\hspace{1em}tp_1=p_2x^2t
\hspace{1em}xp_2=-\xi p_2x, \hspace{1em}yp_2=-p_2y,\hspace{1em}tp_2=p_1x^2t,\\
&xq_1=-\xi q_1x, \hspace{1em}yq_1=-q_1y,\hspace{1em}tq_1=q_2x^2t
\hspace{1em}xq_2=\xi q_2x, \hspace{1em}yq_2=-q_2y,\hspace{1em}tq_2=q_1x^2t.
\end{flalign*}

By Proposition $\ref{24}~(1)$, we know that the relations $p_1^2=0$, $p_2^2=0$, $p_1p_2+p_2p_1=0$
hold in $A$. Similarly, the relations $q_1^2=0$, $q_2^2=0$, $q_1q_2+q_2q_1=0$ hold in $A$. After a direct computation,
$p_1q_1+q_1p_1$, $p_2q_2+q_2p_2$ are primitive elements and
$p_1q_2+q_2p_1$, $p_2q_1+q_1p_2$ are $(1,x^2y)$-skew primitive. As
\begin{equation*}
x(p_1q_2+q_2p_1)=-(p_1q_2+q_2p_1)x,\hspace{1em} x(p_2q_1+q_1p_2)=-(p_2q_1+q_1p_2)x,
\end{equation*}
then the relations
\begin{equation*}
p_1q_1+q_1p_1=0, \hspace{1em}p_2q_2+q_2p_2=0,\hspace{1em} p_1q_2+q_2p_1=0, \hspace{1em}p_2q_1+q_1p_2=0
\end{equation*}
hold in $A$.
Therefore, $A\cong gr(A)$.
\end{proof}

\begin{defi}\label{def 28}
For a set of parameters $I_{28}=\{\lambda,\mu,\alpha \}$, denote by $\mathfrak{U}_{28}(I_{28})$ the algebra that is generated by $x,~y,~t$, $p_1,~p_2$, $q_1,~q_2$ satisfying the relations $(\ref{3.1})-(\ref{3.3})$, $(\ref{e26.1})$, $(\ref{e26.4})$ and
\begin{flalign}
&p_1^2+p_2^2=0,\hspace{1em}p_1p_2=p_2p_1=\lambda(1-y),
\hspace{1em}q_1^2+q_2^2=0,\hspace{1em}q_1q_2=q_2q_1=\mu(1-y),&\label{e28.1}\\
&p_1q_1+q_1p_1+p_2q_2+q_2p_2=0,
\hspace{1em}p_1q_2+q_2p_1+p_2q_1+q_1p_2=\alpha(1-y),\label{e28.2}&\\
&p_1q_1-q_1p_1-p_2q_2+q_2p_2=0,
\hspace{1em}p_1q_2-q_2p_1-p_2q_1+q_1p_2=0.\label{e28.3}&
\end{flalign}
It is a Hopf algebra with its coalgebra structure determined by

$~~~~~~~~~~~~\triangle(p_1)=p_1\otimes 1+\frac{1}{2}(1+x^2)t\otimes p_1-\frac{1}{2}(1-x^2)yt\otimes p_2$,

$~~~~~~~~~~~~\triangle(p_2)=p_2\otimes 1+\frac{1}{2}(1+x^2)yt\otimes p_2-\frac{1}{2}(1-x^2)t\otimes p_1$,

$~~~~~~~~~~~~\triangle(q_1)=q_1\otimes 1+\frac{1}{2}(1+x^2)t\otimes q_1-\frac{1}{2}(1-x^2)yt\otimes q_2$,

$~~~~~~~~~~~~\triangle(q_2)=q_2\otimes 1+\frac{1}{2}(1+x^2)yt\otimes q_2-\frac{1}{2}(1-x^2)t\otimes q_1$.
\end{defi}

\begin{defi}\label{def 29}
Denote by $\mathfrak{U}_{29}(\lambda,\mu)$ the algebra that is generated by $x,~y,~t$, $p_1,~p_2$, $q_1,~q_2$ satisfying the relations $(\ref{3.1})-(\ref{3.3})$, $(\ref{e26.1})$, $(\ref{e27.1})$, $(\ref{e28.1})$, $(\ref{e28.3})$ and
\begin{flalign}
&p_1q_1+q_1p_1+p_2q_2+q_2p_2=0,
\hspace{1em}p_1q_2+q_2p_1+p_2q_1+q_1p_2=0.\label{e29.1}&
\end{flalign}
It is a Hopf algebra with the same coalgebra structure as $\mathfrak{U}_{28}(I_{28})$.

\end{defi}

\begin{pro}\label{28}
$(1)$ Suppose $A$ is a finite-dimensional Hopf algebra with coradical $H$ such
that its infinitesimal braiding is isomorphic to $\Omega_{28}$ or $\Omega_{29}$, then $A\cong \mathfrak{U}_{28}(I_{28})$ or $\mathfrak{U}_{29}(\lambda,\mu)$.

$(2) ~\mathfrak{U}_{28}(I_{28})\cong \mathfrak{U}_{28}(I_{28}^{'})$ if and only if there exist nonzero parameters $a_1,~a_2$, $b_1,~b_2$
such that
\begin{align}
\begin{split}
&a_1^2\lambda^{'}+\frac{1}{2}a_1a_2\alpha^{'}+a_2^2\mu^{'}=\lambda,
\hspace{1em}b_1^2\lambda^{'}+\frac{1}{2}b_1b_2\alpha^{'}+b_2^2\mu^{'}=\mu,\\
&4a_1b_1\lambda^{'}+(a_1b_2+a_2b_1)\alpha^{'}+4a_2b_2\mu^{'}=\alpha,
\label{Aut28.1}
\end{split}
\end{align}
or satisfying
\begin{align}
\begin{split}
&a_1^2\lambda^{'}+\frac{1}{2}a_1a_2\alpha^{'}+a_2^2\mu^{'}=-\lambda,
\hspace{1em}b_1^2\lambda^{'}+\frac{1}{2}b_1b_2\alpha^{'}+b_2^2\mu^{'}=-\mu,\\
&4a_1b_1\lambda^{'}+(a_1b_2+a_2b_1)\alpha^{'}+4a_2b_2\mu^{'}=-\alpha.
\label{Aut28.2}
\end{split}
\end{align}

$(3)$  $\mathfrak{U}_{29}(\lambda,\mu)\cong \mathfrak{U}_{29}(1,1)$ for $\lambda,\mu\neq 0$, and $\mathfrak{U}_{29}(1,1)\ncong \mathfrak{U}_{29}(0,0)$.

\end{pro}
\begin{proof} $(1)$ We prove the claim for $\Omega_{28}$. The proof for $\Omega_{29}$ follows the same lines. By Theorem $\ref{gr}$, we have that gr$(A)\cong \mathcal{B}(\Omega_{28})\sharp H$. Let $M_{17}=\mathbbm{k}\{ p_1,p_2\}=\mathbbm{k}\{ q_1,q_2\}$. Similar to Proposition \ref{14}, we only need to prove that  $(\ref{e28.1})-(\ref{e28.3})$ in Definition \ref{def 28} hold in $A$.
Let $u=p_1q_1-q_1p_1-p_2q_2+q_2p_2$, $v=p_1q_2+q_2p_1+p_2q_1+q_1p_2$,
$w=p_1q_2-q_2p_1-p_2q_1+q_1p_2$.
After a direct computation we have that $p_1^2+p_2^2$, $q_1^2+q_2^2$, $p_1q_1+q_1p_1+p_2q_2+q_2p_2$ are primitive elements and
\begin{flalign*}
&\triangle(p_1p_2+p_2p_1)=(p_1p_2+p_2p_1)\otimes 1+y\otimes (p_1p_2+p_2p_1),\\
&\triangle(p_1p_2-p_2p_1)=(p_1p_2-p_2p_1)\otimes 1+x^2y\otimes (p_1p_2-p_2p_1),\\
&\triangle(u)=u\otimes 1+x^2\otimes u,~\triangle(v)=v\otimes 1+y\otimes v,
~\triangle(w)=w\otimes 1+x^2y\otimes w.
\end{flalign*}
As
\begin{equation*}
t(p_1p_2-p_2p_1)=-(p_1p_2-p_2p_1)t,
\hspace{1em}tu=-ut,\hspace{1em}tw=-wt,
\end{equation*}
the relations $(\ref{e28.1})-(\ref{e28.3})$ hold in $A$.
Then there is a surjective Hopf homomorphism from $\mathfrak{U}_{28}(I_{28})$ to $A$. We can observe that all elements of $\mathfrak{U}_{28}(I_{28})$ can be expressed by linear
combinations of
\begin{equation*}
\{(p_1+p_2)^i(p_1-p_2)^j(q_1+q_2)^k(q_1-q_2)^l x^ey^ft^g;~i,j,k,l,f,g\in\mathbb{I}_{0,1},~ e\in\mathbb{I}_{0,3}\}.
\end{equation*}
 In fact, according to the Diamond Lemma, the set is a basis of
$\mathfrak{U}_{28}(I_{28})$. Then $\dim A =\dim \mathfrak{U}_{28}(I_{28})$,
whence $A\cong \mathfrak{U}_{28}(I_{28})$.

$(2)$ Suppose $\Phi : \mathfrak{U}_{28}(I_{28})\rightarrow \mathfrak{U}_{28}(I_{28}^{'})$ is an isomorphism of
Hopf algebras. Similar to Proposition $\ref{14}$,  $\Phi|_H\in\{\tau_1,\tau_3,\tau_{13},\tau_{15},\tau_{17},\tau_{19},\tau_{29},\tau_{31}\}$.
When $\Phi|_H\in\{\tau_1,\tau_{17}\}$, then
\begin{equation*}
\Phi(p_1)=a_1p_1^{'}+a_2q_1^{'}, \ \ \Phi(p_2)=a_1p_2^{'}+a_2q_2^{'}, \ \ \Phi(q_1)=b_1p_1^{'}+b_2q_1^{'}, \ \ \Phi(q_2)=b_1p_2^{'}+b_2q_2^{'};
\end{equation*}
when $\Phi|_H\in\{\tau_{13},\tau_{29}\}$, then
\begin{equation*}
\Phi(p_1)=a_1p_2^{'}+a_2q_2^{'},\ \Phi(p_2)=a_1p_1^{'}+a_2q_1^{'}, \ \Phi(q_1)=b_1p_2^{'}+b_2q_2^{'},\ \Phi(q_2)=b_1p_1^{'}+b_2q_1^{'},
\end{equation*}
thus the relation $(\ref{Aut28.1})$ holds.
When $\Phi|_H\in\{\tau_3,\tau_{19}\}$, then
\begin{equation*}
\Phi(p_1)=a_1p_1^{'}+a_2q_1^{'},\  \Phi(p_2)=-a_1p_2^{'}-a_2q_2^{'}, \ \Phi(q_1)=b_1p_1^{'}+b_2q_1^{'}, \ \Phi(q_2)=-b_1p_2^{'}-b_2q_2^{'};
\end{equation*}
when $\Phi|_H\in\{\tau_{15},\tau_{31}\}$, then
\begin{equation*}
\Phi(p_1)=a_1p_2^{'}+a_2q_2^{'}, \ \Phi(p_2)=-a_1p_1^{'}-a_2q_1^{'}, \ \Phi(q_1)=b_1p_2^{'}+b_2q_2^{'}, \ \Phi(q_2)=-b_1p_1^{'}-b_2q_1^{'},
\end{equation*}
thus the relation $(\ref{Aut28.2})$ holds.

$(3)$ The proof is the same as that of Proposition $\ref{14}~(4)$.
\end{proof}

\noindent
$\mathbf{Proof~ of ~Theorem ~B}.$ Let $M$ be one of the Yetter-Drinfeld modules listed in Theorem A. Let $A$ be a finite-dimensional Hopf algebra over $H$ such that its infinitesimal braiding is isomorphic to $M$.
By Theorem \ref{gr}, gr$A\cong \mathcal{B}(M)\sharp H$. Propositions \ref{1}, ~\ref{2}, ~\ref{4}, ~\ref{14}, ~\ref{20}, ~\ref{24}, ~\ref{25}, ~\ref{28} finish the proof.

\newpage

\begin{center}
$\mathbf{ACKNOWLEDGMENT}$
\end{center}

The authors would like to thank Dr. Yuxing Shi and Dr. Rongchuan Xiong for their helpful discussions.
The first author is indebted to Dr. Jiao Zhang for her seminar.
The second author is supported by NSERC of Canada. The third author is supported by the NSFC (Grant No. 11771142) and in part by the Science and Technology Commission of Shanghai Municipality (No. 18dz2271000).

\end{document}